\documentclass[12pt]{amsart}
\usepackage[margin=3.0cm]{geometry}

\usepackage{verbatim}
\usepackage{amssymb}
\usepackage{latexsym}
\usepackage[all,cmtip,arrow,2cell]{xy}
\usepackage{graphicx}
\usepackage{hyperref}

\usepackage[utf8]{inputenc}

\usepackage[dvipsnames]{xcolor}
\usepackage[normalem]{ulem}
\usepackage{pdfpages}

\theoremstyle{plain}
\newtheorem{Theorem}{Theorem}[section]
\newtheorem{Proposition}[Theorem]{Proposition}
\newtheorem{Lemma}[Theorem]{Lemma}
\newtheorem{Corollary}[Theorem]{Corollary}
\newtheorem{Conjecture}[Theorem]{Conjecture}
\newtheorem{Definition}[Theorem]{Definition}

\theoremstyle{definition}

\newtheorem{Remark}[Theorem]{Remark}
\newtheorem{Example}[Theorem]{Example}

\newtheorem{Question}[Theorem]{Question}

\newcommand{\calF}{\mathcal{F}}

\newcommand{\calO}{\mathcal{O}}

\newcommand{\calX}{\mathcal{X}}
 
\newcommand{\bbR}{\mathbb{R}} 

\newcommand{\bbZ}{\mathbb{Z}}

\newcommand{\frakg}{\mathfrak{g}}
\newcommand{\frakh}{\mathfrak{h}}
\renewcommand{\phi}{\varphi}

\DeclareMathOperator{\id}{id}
\DeclareMathOperator{\Lie}{\mathcal{L}}
\DeclareMathOperator{\Ad}{Ad}

\DeclareMathOperator{\End}{{End}}
\DeclareMathOperator{\Der}{{Der}}
\DeclareMathOperator{\Aut}{{Aut}}
\DeclareMathOperator{\Supp}{{supp}}
\DeclareMathOperator{\Hom}{{Hom}}
\DeclareMathOperator{\Diff}{{Diff}}

\newcommand{\DA}{\mathbf{d}}
\newcommand{\IA}{\mathbf{i}}
\newcommand{\LA}{\mathbf{L}}
\renewcommand{\nabla}{D}
\DeclareMathOperator{\Hor}{{hor}}

\newcommand{\DW}{\Hat{d}}
\newcommand{\IW}{\Hat{\iota}}
\newcommand{\LW}{\Hat{\Lie}}

\begin{document}

\title{Hamiltonian Lie algebroids}

\author[C.~Blohmann]{Christian Blohmann}
\address{Max-Planck-Institut f\"ur Mathematik, Vivatsgasse 7, 53111 Bonn, Germany}
\email{blohmann@mpim-bonn.mpg.de}

\author[A.~Weinstein]{Alan Weinstein}
\address{Department of Mathematics, University of California, Berkeley, CA 94720, USA\\ Department of Mathematics, Stanford University, Stanford, CA 94305 {\rm (Visiting)}}
\email{alanw@math.berkeley.edu}


\date{\today}


\begin{abstract}

In previous work with M.C.~Fernandes, we found a Lie algebroid symmetry for the Einstein evolution equations.  The present work was motivated by the effort to combine this symmetry with the hamiltonian structure of the equations to explain the coisotropic structure of the constraint subset for the initial value problem. In this paper, we extend the notion of hamiltonian structure from Lie algebra actions to general Lie algebroids over presymplectic manifolds.  Application of this construction to the problem in general relativity is still work in progress.

After comparing a number of possible compatibility conditions between an anchor map $A\to TM$ on a vector bundle $A$ and a presymplectic structure on the base $M$, we choose the most natural of them, best formulated in terms of a suitably chosen connection on $A$.  We define a notion of momentum section of $A^*$, and, when $A$ is a Lie algebroid, we specify a condition for compatibility with the Lie algebroid bracket. Compatibility conditions on an anchor, a Lie algebroid bracket, a momentum section, a connection, and a presymplectic structure are then the defining properties of a hamiltonian Lie algebroid. For an action Lie algebroid with the trivial connection, the conditions reduce to those for a hamiltonian action.  We show that the clean zero locus of the momentum section of a hamiltonian Lie algebroid is a coisotropic submanifold. To define morphisms of hamiltonian Lie algebroids, we express the structure in terms of a bigraded algebra generated by Lie algebroid forms and de Rham forms on its base. We give an Atiyah-Bott type characterization of a bracket-compatible momentum map; it is equivalent to a closed basic extension of the presymplectic form, within the generalization of the BRST model of equivariant cohomology to Lie algebroids. We show how to construct a groupoid by reduction of an action Lie groupoid $G\times M$ by a subgroup $H$ of $G$ which is not necessarily normal, and we find conditions which imply that a hamiltonian structure descends to such a reduced Lie algebroid.
  
We consider many examples and, in particular, find that the tangent Lie algebroid over a symplectic manifold is hamiltonian with respect to some connection if and only if the symplectic structure has a nowhere vanishing primitive.  Recent results of Stratmann and Tang show that this is the case whenever the symplectic structure is exact.
\end{abstract} 
\maketitle

\newpage

\tableofcontents


\section{Introduction}
\label{sec:Introduction}

We introduce in this article a notion of hamiltonian Lie algebroid over a presymplectic manifold.  The definition consists of three conditions, each of which generalizes a standard condition in the case where a Lie algebroid is the  action Lie algebroid $\mathfrak g\times M$ 
associated with an action of a Lie algebra $\mathfrak g$ on a presymplectic manifold $M$.
 
Recall that an action of a Lie algebra $\frakg$ on a presymplectic manifold $(M,\omega)$ is called hamiltonian when the following three conditions are satisfied:
\begin{enumerate}
\item 
the action of $\frakg$ leaves $\omega$ invariant;
\item
there is a momentum map $\mu:M\to \frakg^*$ for the action;
\item
$\mu$ is equivariant for the infinitesimal coadjoint action of $\mathfrak{g}$.
\end{enumerate}

To express these conditions in terms of the action Lie algebroid $A = \frakg \times M$, we first notice that the action is encoded in the anchor $\rho:A\to TM$. Next, the momentum map can be considered as a section of $A^*$; thus, for a general Lie algebroid we will look for a ``momentum section'' of $A^*$.  To express conditions (1-3) above in terms of the action Lie algebroid, though, we need to restrict attention to constant sections of $A$, which correspond to elements of $\frakg$.  

In a general Lie algebroid, which may not even admit a trivialization, there is no natural notion of constant section.  Instead, we will add to our structure a vector bundle connection $D$ on $A$. When $D$ is the product connection on $\frakg \times M$, the constant sections are the horizontal ones.  In general, we have only sections which are ``horizontal at a given point,'' so we may attempt to impose our conditions point-wise.  It is more useful, as we will see, to express conditions (1-3) by formulas involving $D$, applied to general sections.  The other ingredients in these formulas are $\omega$, $\rho$, and the $A^*$-valued 1 form $\gamma$ defined by $\langle\gamma(v),a\rangle = \omega(v,\rho(a))$ for vector fields $v$ and sections $a$ of $A$.  Starting with $D$, we get a dual connection on $A^*$, and from there a degree-1 operator (which we also denote by $D$) on the graded vector space $\Omega^\bullet(M,A^*)$ of $A^*$-valued differential forms on $M$.

We can now make the following definitions, which reduce, for an action Lie algebroid with the trivial connection, to the usual definitions for Lie algebra actions.

\begin{Definition}
\label{def:HamLAcomplete}
Let $(A,\rho, [~,~])$ be a Lie algebroid over a presymplectic manifold $(M,\omega)$. With $D$ and $\gamma$ as defined above:
\begin{itemize}

\item[(H1)]
$A$ is \textbf{presymplectically anchored} with respect to $D$ if
\begin{equation*}
  D\gamma = 0 \,.
\end{equation*}

\item[(H2)]
A section $\mu \in \Gamma(A^*) = \Omega^0(M, A^*)$ is a \textbf{$D$-momentum section} if
\begin{equation*}
  D\mu = \gamma \,.
\end{equation*}

\item[(H3)]
A $D$-momentum section $\mu$ is \textbf{bracket-compatible} if
\begin{equation*}
  (\DA \mu)(a,b) = 
  - \langle \gamma(\rho a), b \rangle
\end{equation*}
for all sections $a$ and $b$ of  $A$, where $\DA$ is the Lie algebroid differential, a degree-1 differential on the graded algebra $\Gamma(\wedge^\bullet A^*)$.
\end{itemize}
A Lie algebroid together with a connection $D$ and a section $\mu$ of $A^*$ satisfying (H1) and (H2) is called \textbf{weakly hamiltonian}. It is called \textbf{hamiltonian} if it satisfies (H1)-(H3).
\end{Definition}

Note that $\langle \gamma(\rho a), b \rangle = \omega(\rho a, \rho b)$, which shows that the right hand side of condition (H3) is always antisymmetric in $a$ and $b$.

In Section \ref{sec:anchored}, we arrive at a version of the condition (H1) in terms of local sections, choosing one from among several possible notions of what it should mean for a Lie algebroid to be compatible with a presymplectic structure. This choice is based partly on examples, but we see in Section \ref{sec:connections} that the definition has a very natural formulation in terms of connections.  

In Section \ref{sec:MomPre}, we introduce the notion of $D$-momentum section for a Lie algebroid presymplectically anchored with respect to a connection $D$, and we define in Section \ref{sec:HamMomMap} what it means for a momentum section to be compatible with a Lie algebroid bracket.   (Until this point, the bracket on sections of $A$ is irrelevant, and we assume simply that $A$ is provided with a connection and a bundle map $A\to TM$.)

We reach in Section \ref{sec:ZeroLocus} one of the original goals of our work, which is to show that, like the zero set of the momentum map for a hamiltonian Lie algebra action, the zero locus of the momentum section of any hamiltonian Lie algebroid is coisotropic.   It turns out to be useful here to express bracket-compatibility in terms of a torsion tensor $T:\wedge^2 A\to A$ associated to the connection $D$.  

Having established the definitions and some properties of hamiltonian Lie algebroids, we proceed in Section \ref{sec:Examples} to study a wide range of examples.  After dealing with the rather simple case of bundles of Lie algebras (where $\rho = 0$), we move to the other extreme case, where $\rho$ is an isomorphism, so that $A$ is (isomorphic to) the tangent Lie algebroid.  We will see that $TM$ is presymplectically anchored with respect to a connection $D$ exactly when the presymplectic form $\omega$ is parallel with respect to an ``opposite connection'' $\check{D}$.  Such a connection exists if and only if $\omega$ has constant rank.  In fact, it is known that the connection can be chosen so that its torsion takes values in the kernel of the map $\gamma$, which is here the map $\tilde{\omega}: TM\to T^*M$ associated to $\omega$. When $\omega$ is symplectic, the torsion can thus be made zero.  For our purposes, though, when it comes to finding bracket-compatible momentum sections (which are in this case 1-forms on $M$), it will be essential to admit connections with nonvanishing torsion.  We come to some surprising conclusions, including the result that the tangent bundle of a symplectic manifold has the structure of a hamiltonian Lie algebroid if and only if $\omega = d\mu$ for a {\em nowhere vanishing} 1-form $\mu$.  This raises a question of symplectic topology: for which exact symplectic structures does a nowhere vanishing primitive $\mu$ exist?  We give some classes of examples for which it does exist (such as the cotangent bundles of noncompact manifolds); we know of no example for which it does not.

To further develop the theory of hamiltonian Lie algebroids, we introduce and study in Section \ref{sec:Category} a notion of morphism, using a bicomplex which provides a unified framework for the various structures that make up a hamiltonian Lie algebroid.  The section ends with the example of transitive Lie algebroids.

Our original motivation came from previous work on the Einstein evolution equations for a gravitational field in a spacetime free of matter, which are a hamiltonian system on a symplectic phase space $\mathcal{P}$ of Cauchy data.  Because the corresponding lagrangian functional is degenerate, the initial conditions must belong to a constraint subset $\mathcal C \subset \mathcal{P}$.  It is known that the constraint subset is coisotropic and has conic singularities, geometric properties which are characteristic of the zero level sets of momentum maps for hamiltonian actions of Lie algebras  (see e.g.~\cite{OrtegaRatiu:Momentum}).  But the constraint functionals cannot actually be the components of a momentum map.  When their Poisson brackets are written as linear combinations of the constraint functions themselves, the coefficients are functions on $\mathcal{P}$ rather than constants, as would be the case with the momentum components of a hamiltonian action of a Lie algebra action on $\mathcal{P}$. 

In  \cite{BFW}, with Marco Cezar Fernandes, we found a Lie algebr{\em oid} over an enlarged space $\mathcal{P}'$ mapping onto $\mathcal{P}$, with a natural trivialization for which the bracket relations among constant sections match the bracket relations among the constraint functions.  What was missing was a connection between our Lie algebroid and the symplectic geometry of $\mathcal{P}$ which would be analogous to that of a hamiltonian action of a Lie algebra.  The present paper begins to fill that gap by introducing a notion of hamiltonian Lie algebroid for which the zero set of what we call the momentum section is coisotropic.  So far, though, we have not succeeded in finding the required presymplectic structure which is compatible with the Lie algebroid.

Section \ref{sec:ReductionAction}, independent of the earlier material, is closely related to the application to general relativity. It is easy to see that, if a group $G$ acts on a set $M$, and that $H$ is a normal subgroup of $G$, then the action descends to an action of $G/H$ on $M/H$.  We show that, even if $H$ is not normal, if it acts freely on $M$, then the action groupoid $G\times M$ still descends to a groupoid over $M/H$ (which is the action groupoid $G/H \times M/H$ when $H$ is normal).  This result holds for Lie groupoids\footnote{The corresponding result for Lie algebroids was obtained by Lu \cite{Lu:2008}.  See Section \ref{sec:related} for further discussion.} when the action is free and proper.  In any case, the reduced groupoid is Morita equivalent to the original one.  The example of interest for the Einstein equations is that where $G$ is the group of diffeomorphisms of a space time $S$, $\Sigma$ is a distinguished hypersurface in $S$, $M$ is the space of Lorentz metrics on $S$ for which $\Sigma$ is spacelike, and $H$ is the subgroup leaving $\Sigma$ pointwise fixed; we give more details in Example \ref{ex:grel}.  Then, in Sec. \ref{sec:HamHReduced}, we discuss the reduction of hamiltonian structures from action groupoids to their reductions.

In Section \ref{sec:CohomInterp}, the central result is an extension to Lie algebroids of the theorem of Atiyah and Bott \cite{AtiyahBott:1984}
that equivariant momentum maps for a symplectic action of a Lie algebra $\frakg$ on $(P,\omega)$ (i.e., maps  which make the action hamiltonian) correspond to extensions of $\omega$ to closed basic elements in a Weil algebra whose cohomology is, under a properness assumption, the equivariant cohomology for a group action which integrates the $\frakg$-action.  For Lie algebroids, the usual Weil algebra is replaced by the one introduced by Mehta 
\cite{MehtaAlgebroids:2009}.  We show that, in this setting,  bracket compatible momentum sections for presymplectically anchored Lie algebroids (i.e., sections  which make the Lie algebroid hamiltonian) correspond to closed basic elements in this more general Weil algebra.  This takes considerable effort, including developing a Cartan calculus extending that from Section \ref{sec:Category} in order to give suitable definitions of what it means for Weil algebra elements to be closed and basic.

The  paper ends with Section \ref{sec:OpenQuestions} on open questions of varying degrees of difficulty, followed by a discussion of previous work related to ours.

\subsection*{Acknowledgements}

For comments, encouragement, and advice we would like to thank Alejandro Cabrera, Marius Crainic, Noriaki Ikeda, Madeleine Jotz Lean, Honglei Lang, Kirill Mackenzie, Ioan Marcut, Jo\~ao Nuno Mestre, Michele Schiavina, and Thomas Strobl as well as the audiences over several years who heard us present preliminary versions of this work and gave invaluable feedback.  We also thank the referee for many useful suggestions, which have improved the paper substantially.

C.B.~would like to thank the math department of UC Berkeley for its hospitality.  A.W.'s research was supported in part by the UC Berkeley Committee on Research.

\section{Anchored vector bundles over presymplectic manifolds}
\label{sec:anchored}

\subsection{A review of presymplectic geometry}
\label{sec:conventions}

All manifolds will be assumed to be finite dimensional and second countable, which implies that the space $\Gamma(M,E)$ of smooth sections of a finite rank vector bundle $E$ over a manifold $M$ is finitely generated as a $C^\infty(M)$-module. (We will need this assumption for the proofs of Propositions \ref{prop:CompatStrength} and~\ref{prop:Connection1}.) For convenience, we will also assume that our manifolds are connected.

A \textbf{presymplectic form} on $M$ is a closed 2-form $\omega \in \Omega^2(M)$.  We denote by $\tilde{\omega}: TM \to T^*M$, $v \mapsto \iota_v \omega$ the associated map of vector bundles. When the rank of $\tilde{\omega}$ is constant, we will call  $\omega$ \textbf{regular}.\footnote{Unlike some authors, we will not reserve the term ``presymplectic'' for the regular case.} As usual, when $\omega$ is nondegenerate in the sense that $\tilde{\omega}$ is an isomorphism (equivalently, the characteristic distribution is zero), we call $\omega$ \textbf{symplectic}.  

Given a subspace $V \subseteq T_m M$, the subspace $\{v \in T_m M |\omega(v,w) = 0$ {\mbox for all} $w \in V$\} will be called the \textbf{presymplectic orthogonal} of $V$  and will be denoted by $V^\perp$.   The subspaces  $(T_m M)^\perp$, which form the kernel of $\tilde{\omega}$, will be called the \textbf{characteristic subspaces}. 

 The \textbf{characteristic distribution} of all characteristic subspaces will be denoted by $TM^\perp$, where we omit the parentheses to lighten the notation. For regular $\omega$,  $TM^\perp$ is integrable because $\omega$ is closed; the leaves form the \textbf{characteristic foliation}.   {The local leaf spaces (and the global one, if the foliation is globally a fibration) then inherit symplectic structures which pull back to $\omega$.   We will call these \textbf{reduced spaces} of $(M,\omega)$.  
In the regular case, any presymplectic structure descends to symplectic structures on these reduced spaces.

A subspace $V$ in any presymplectic space will be called symplectic if the induced 2-form is symplectic (equivalently, if $V\cap V^{\perp}$ is the zero subspace).  For instance, any complement of $(T_m M)^\perp$ in $T_m M$ is symplectic.
When $V\subseteq T_m M$ is contained in $V^\perp$, i.e.~ when the pullback of $\omega$ to $V$ is zero, $V$ is called \textbf{isotropic}.
When $V$ contains $V^\perp$, $V$ is called \textbf{coisotropic}.  Note that, since the characteristic subspace $(T_m M)^\perp$ is contained in $V^\perp$ for any $V$, every coisotropic subspace must contain the characteristic subspace.  

$V$ is called \textbf{lagrangian} when it is equal to its presymplectic orthogonal.  A subspace is lagrangian if and only if it is the sum of the characteristic subspace and a lagrangian subspace of a symplectic complement.

\begin{Remark}
Since $V^{\perp\perp}=V$ for any subspace of a symplectic space, it is immediate that a subspace is isotropic if and only if its symplectic orthogonal\footnote{We will sometimes replace the adjective ``presymplectic'' by ``symplectic'' when the ambient space is symplectic.} is coisotropic.  In the general presymplectic case, we have only $V^{\perp\perp} \subseteq V$.  It does follow from this inclusion that $V^{\perp\perp\perp}= V^\perp$, and from there that $V$ is isotropic if and only if $V^\perp$ is coisotropic, and that $V$ being coisotropic implies that $V^\perp$ is isotropic.   But $V^\perp$ can be isotropic without $V$ being coisotropic.  For instance, if $S$ is any complement of the characteristic subspace $C$, the latter being assumed nonzero, then $S^\perp = C$, which is certainly isotropic.  But $S$ does not contain its presymplectic orthogonal $C$, so it is not coisotropic.  
\end{Remark}

Passing from vector spaces to manifolds, we call a submanifold $N$ of $M$ (co)isotropic if all of its tangent spaces are (co)isotropic in $TM$.  For instance, the leaves of the characteristic foliation of a coisotropic submanifold $C$ (that is, the foliation tangent to the characteristic subbundle $(TC)^\perp$) are always isotropic.   Since coisotropic subspaces contain the characteristic subspaces, every (closed) coisotropic submanifold is a union of characteristic leaves.  Thus, coisotropic submanifolds correspond to certain subsets of the reduced space; when there is a nice reduced manifold, these subsets are coisotropic submanifolds of it.  

A vector field $v \in \calX(M)$ for which $\iota_v \omega$ is closed, i.e.~for which the Lie derivative $\Lie_v \omega$ is zero, is called \textbf{presymplectic}. When, additionally, $\iota_v\omega$ is exact, $v$ is called \textbf{hamiltonian}; any function $f$ for which $df = \iota_v \omega$ is called a \textbf{generator} for $v$, and the pair $(f,v)$ is called a hamiltonian pair.  Not every $f$ belongs to a hamiltonian pair; when it does, it is called a hamiltonian function. 

For two hamiltonian functions $f$ and $g$, the Poisson bracket $\{f,g\} := \omega(v,w)$, where $(f,v)$ and $(g,w)$ are hamiltonian pairs, is well-defined, i.e.~it does not depend on the choice of the hamiltonian vector fields $v$ or $w$. The Poisson bracket equips the vector space of hamiltonian functions with the structure of a Lie algebra.  

\begin{Remark}
 Note that a hamiltonian function, since its differential $df$ is in the image of $\tilde{\omega}$, must be constant along each leaf of the characteristic foliation of $M$.  If the characteristic foliation is not simple, then, non-constant hamiltonian functions may be scarce or even nonexistent.
\end{Remark}

When $\omega$ is symplectic, every function $f$ belongs to a unique hamiltonian pair $(f,v_f)$. Note that the map $C^\infty(M) \to \calX(M)$, $f \mapsto v_f$ is an \emph{anti}homomorphism of Lie algebras. On the other hand, to be consistent with the usual definition of Lie algebroids, we define an action of a Lie algebra $\frakg$ on a manifold $M$ to be a map of vector bundles $\rho:\frakg \times M \to TM$ that induces a \emph{homo}morphism of Lie algebras also denoted by $\rho: \frakg \to \calX(M)$. This will make it necessary to introduce minus signs at some points, e.g.~in Proposition \ref{def:MomentumEquiv}.

\subsection{A hierarchy of compatibility conditions}

The most basic compatibility condition for an action $\rho:\frakg \to \calX(M)$ of a Lie algebra on a presymplectic manifold is that the  vector fields $\rho(a)$ be presymplectic for all $a \in \frakg$.  We will often think of the action as a vector bundle map $\frakg \times M \to TM$, which we will also denote by $\rho$.\footnote{Given a map $\phi: A \to B$ of vector bundles over $M$, we will notationally distinguish between the value  $\phi(a)$ of $\phi$ at $a \in A$ and the composition $\phi \circ a$ of $\phi$ with a section $a \in \Gamma(M,A)$. If no confusion can arise we will also write, more succinctly, $\phi a$ for $\phi \circ a$.}  From this point of view, if we consider  $\frakg \times M$ simply as a vector bundle $A$, the condition that the action be symplectic is expressed in terms of the trivialization of $A$ in which the elements of $\frakg$ correspond to the constant sections of $A$.  (Note that this condition does not involve the Lie bracket.)  As a first step toward a notion of hamiltonian Lie algebroid, we have to find a suitable condition on a Lie algebroid anchor $\rho:A \to TM$ when $A$ is not equipped with a chosen trivialization. 

We will next identify five reasonable compatibility conditions on an anchor $\rho:A\to M$, ordered by decreasing strength, all of which are satisfied by the anchor of an action Lie algebroid associated to a presymplectic action of a Lie algebra: the first is that $\Gamma(M,A)$ be generated by sections that are mapped to presymplectic vector fields; the second is that the $\Gamma(M,A)$ be \emph{locally} generated by sections that are mapped to presymplectic vector fields (in which case subsets of the generators give  local frames around each point of $M$); the third is that the ideal generated by the 1-forms $\iota_{\rho a}\omega$, $a \in \Gamma(M,A)$ be differential, i.e.~closed under the exterior derivative $d$; the fourth is that the presymplectic orthogonal of the image of the anchor be involutive.  We will show that the last three conditions are equivalent if the anchor has constant rank, and we will give examples to show that all four conditions are mutually inequivalent.  We will see that the first three conditions imply that $A$ is equipped with a natural class of linear connections in terms of which the conditions can be formulated in a natural way; this formulation will then play a fundamental role in our theory. The fifth condition is that the hamiltonian functions that are constant on the leaves of the characteristic distribution $\rho(A)$ are closed under the Poisson bracket.

Until Section~\ref{sec:HamMomMap}, we will make use only of the anchor $\rho:A\to TM$ and not of the bracket on sections of $A$.  Therefore, we need now only the following structure, weaker than that of a Lie algebroid.

\begin{Definition}
Let $A \to M$ be a smooth vector bundle. A map of vector bundles $\rho: A \to TM$ is called an \textbf{anchor} for $A$. A smooth vector bundle together with an anchor is called an \textbf{anchored} vector bundle.  An anchored vector bundle will be called \textbf{regular} if the  rank of its anchor is constant.
\end{Definition}

\begin{Definition}
\label{def:DualizedAnchor}
For an anchored vector bundle $A$ over a presymplectic manifold $M$, the map  $\gamma := \tilde{\omega}\circ \rho:A \to T^*M$, $a \mapsto \iota_{\rho(a)}\omega$ will be called the \textbf{$\omega$-dualized anchor} (or simply ``dualized anchor'').  When $\gamma$ has constant rank, we  will call the anchored vector bundle \textbf{presymplectically regular.}
\end{Definition}

Since no bracket is involved, any linear map from a Lie algebra $\mathfrak g$ (or any vector space, for that matter) to the vector fields on $M$ gives the trivial bundle $\mathfrak{g} \times M$ the structure of an anchored vector bundle over $M$.

\begin{Definition}
\label{def:CompatCond}
Let $(A,\rho)$ be an anchored vector bundle over the  presymplectic manifold $(M,\omega)$. We will consider the following compatibility conditions on $\rho$ and $\omega$:
\begin{enumerate}

\item[(C1)] There is a set  $S \subseteq \Gamma(M,A)$ of sections generating $\Gamma(M,A)$ as a $C^\infty(M)$-module such that $\rho a$ is a presymplectic vector field for every $a \in S$.

\item[(C2)] Every point $m \in M$ has a neighborhood $U$ with a local frame $\{a_1, \dots, a_r \} \subseteq \Gamma(U,A)$ such that $\rho a_1, \ldots, \rho a_r$ are presymplectic vector fields. 

\item[(C3)] The algebraic ideal of $\Omega(M)$ generated by $\{ \gamma a\,|\, a \in \Gamma(M,A) \}$ is closed under the de Rham differential.

\item[(C4)] The presymplectic orthogonal $\rho(A)^\perp$ is an involutive distribution in the sense that its space of sections is closed under the bracket of vector fields.

\item[(C5)] For every open set $U \subseteq M$ the vector space of all hamiltonian functions on $U$ that are annihilated by $\rho(A)$ is closed under the Poisson bracket.

\end{enumerate}
\end{Definition}

\begin{Remark}
Recall that  there are two natural notions of involutivity for a family 
$\{\gamma_i\}$ of 1-forms  on $M$: (i) The ideal in $\Omega(M)$ generated by the $\gamma_i$ is closed under the de Rham differential; (ii) the space of vector fields that are annihilated by all $\gamma_i$ is closed under the Lie bracket. The difference between these two notions is the difference between (C3) and (C4). To make this paper reasonably self-contained, we include below proofs for the well-known facts that (i) implies (ii) (Proposition \ref{prop:CompatStrength}), that (i) and (ii) are equivalent when the annihilator is regular (Proposition \ref{prop:CompatRegular}), and that in the non-regular case (ii) does not imply (i) (Example~\ref{ex:C4example}). 
\end{Remark}

\begin{Proposition}
\label{prop:CompatStrength}
We have the following implications between the compatibility conditions: (C1) $\Rightarrow$ (C2) $\Rightarrow$ (C3) $\Rightarrow$ (C4) $\Rightarrow$ (C5).
\end{Proposition}
\begin{proof}

Assume (C1).  Since the sections in $S$ are a generating set, we can choose for any $m$ a subset $\{a_1, \ldots, a_r \}$ of $S$ whose values form a basis of the fibre at $m$.  It follows that these sections are a frame when restricted to some neighborhood of $m$. The vector fields $\rho a_1, \ldots, \rho a_r$ are presymplectic, so (C2) holds.

Assume (C2). Let $\{ U_p \}$ be a collection of open sets covering $M$ such that every $U_p$ has a local frame $\{a_{p1}, \ldots, a_{pr_p} \}$ that is mapped by $\rho$ to presymplectic vector fields, i.e.\ all of the 1-forms $\gamma a_{pk}$  are closed. Since we assume all manifolds to be second countable, we can assume without loss of generality that the cover $\{U_p\}$ has locally finite intersections. By Ostrand's theorem from dimension theory \cite[Theorem~II.6, p.~22]{Nagata:Modern} we can find a finite open cover $\{V_1, \ldots, V_{n+1}\}$ of $M$, where every $V_q$ is the disjoint union of open subsets of the sets in $\{U_p\}$. Over every $V_q$ there is still a frame $\{a_{q1}, \ldots, a_{qr}\}$ that is mapped by $\rho$ to presymplectic vector fields. Let $\{ \chi_q : V_q \to [0,1] \}$ be a partition of unity. Then the finite collection $S := \{\chi_q a_{qk}\}$ of sections generates the $C^\infty(M)$-module $\Gamma(M,A)$. Moreover, $d (\gamma\circ(\chi_q a_{qk})) = d (\chi_q\, \gamma a_{qk}) = d\chi_q \wedge \gamma a_{qk}$ (no summation over $q$), which shows that the ideal generated by the 1-forms $\gamma a$ is closed under the de Rham differential, so (C3) holds.

Assume (C3). It suffices to check the involutivity of a distribution locally, so let $U \subseteq M$ be an open set with a local frame $\{a_1, \ldots, a_r \} \subseteq \Gamma(U,A)$. Let $\gamma_k := \gamma a_k$. Property (C3) means that $d\gamma_k = \Omega^j_k \wedge \gamma_j$ (summation over $j$) for some matrix of 1-forms $\Omega^j_k$, $1 \leq j,k \leq r$. By definition, two vector fields $v$, $w$ lie in the presymplectic orthogonal of $\rho(A)$ if $\iota_v\gamma_k = 0 = \iota_w \gamma_k$ for all $k$. A short calculation shows that
\begin{equation*}
\begin{split}
 \iota_{[v,w]} \gamma_k 
 &= (\Lie_v \iota_w - \iota_w \Lie_v)\gamma_k = - \iota_w \iota_v d \gamma_k
 = -\iota_w \iota_v \bigl( \Omega^j_k \wedge \gamma_k \bigr) \\
 &= \iota_v\Omega^j_k\, \iota_w\gamma_k
 - \iota_w\Omega^j_k\, \iota_v\gamma_k
 \\
 &= 0 \,,
\end{split}
\end{equation*}
so (C4) holds.

Assume (C4). Let $(f,v)$ be a hamiltonian pair over $U$. We have $\rho(a) \cdot f = \iota_{\rho(a)} df = \omega(v,\rho(a))$ for all $a \in A$, so that $f$ is annihilated by $\rho(A)$ if and only if $v$ takes values in $\rho(A)^\perp$. Let $(g,w)$ be another hamiltonian pair for which $g$ is also annihilated by $\rho(A)$. Then $(\{f,g\}, [v,w])$ is a hamiltonian pair which satisfies $\rho(a) \cdot \{f, g\} = \omega([v,w], \rho(a)) = 0$ for all $a \in A$, where in the last step we have used that $\rho(A)^\perp$ is involutive. This shows that (C5) holds.
\end{proof}

\begin{Corollary}
Let $\rho$ be a presymplectic action of the Lie algebra $\frakg$ on the presymplectic manifold $(M,\omega)$. Then the compatibility conditions (C1)-(C5) are all satisfied for the action Lie algebroid $A=\frakg \times M$.
\end{Corollary}
\begin{proof}
The constant sections corresponding to any basis of $\frakg$ form a basis of the $C^\infty(M)$-module $\Gamma(M,A)$ which is mapped to presymplectic vector fields, so (C1) is satisfied. By Proposition \ref{prop:CompatStrength} this implies the other compatibility conditions.
\end{proof}

\begin{Proposition}
\label{prop:CompatRegular}
For a  presymplectically regular anchored vector bundle over a presymplectic manifold $M$, conditions (C2), (C3), and (C4) of Definition \ref{def:CompatCond} are equivalent.
\end{Proposition}
\begin{proof}
Assume (C4). If $p$ is the constant rank of the dualized anchor  $\gamma$, the involutive distribution given by the presymplectic orthogonal of $\rho(A)$ has constant dimension $n-p$, $n = \dim M$, so it is integrable by the Frobenius theorem. Let $(x^1, \ldots, x^{n-p}, y^1, \ldots, y^p): U \to \bbR^n$ be a local foliation chart, where $x^i$ are coordinates along the leaves and $y^i$ are transverse coordinates.  Then the $dy^i$ span the range of $\gamma$, so we can find a local frame $\{a^1, \ldots, a^r\} \subseteq \Gamma(U,A)$ such that $\gamma a_k = dy^k$ for $1 \leq k \leq p$ and the $a_k$ for $k > p$ form a basis of the kernel of $\gamma$. We conclude that the anchor maps the local frame to presymplectic vector fields. This shows that (C4) implies (C2). By Proposition \ref{prop:CompatStrength}, this implies (C3).
\end{proof}

Proposition \ref{prop:CompatRegular} tells us that, in the symplectic case, (C2), (C3), and (C4) are equivalent near points where the anchor $\rho$ itself has constant rank.   We will now give examples to show that, even in the symplectic case, the conditions are inequivalent without the regularity assumption, and that (C2) does not imply (C1), even in the regular symplectic case.

\begin{Example}
\label{ex:C4example}
Let $M = \bbR^2$ with the symplectic form $\omega = dx \wedge dy$. Let $A = \bbR \times M$  be the trivial line bundle and $\rho: A \to TM$ the anchor taking the constant section $a=1$ to the vector field $\rho\circ 1 = x \frac{\partial}{\partial x} + y \frac{\partial}{\partial y}$. On the one hand, the symplectic orthogonal of $\rho(A)$ is a distribution which is 1-dimensional on a open dense subset, so it is involutive. On the other hand, $\gamma\circ 1  = x\, dy - y\, dx$ and $d(\gamma \circ 1) = 2 dx \wedge dy$, which is not of the form $\alpha \wedge (\gamma \circ 1)$ for any $\alpha\in \Omega^1(M)$. We conclude that $A$ satisfies (C4) but not (C3).
\end{Example}

\begin{Example}
Once again, let $M = \bbR^2$ with the symplectic form $\omega = dx \wedge dy$ and $A = \bbR \times M$.  Now let the anchor $\rho: A \to TM$ take the constant section $a = 1$ to the vector field $\rho\circ 1 = 
  \bigl( y \frac{\partial}{\partial x} 
  - x \frac{\partial}{\partial y} \bigr)
  - (x^2 + y^2) \bigl( x\frac{\partial}{\partial x} 
  + y \frac{\partial}{\partial y} \bigr).$
This yields the 1-form $\gamma \circ 1= xdx +ydy +(x^2+y^2)(ydx-xdy) = \frac{1}{2}r dr - r^4 d\theta$, where $(r,\theta)$ are the usual polar coordinates.   Its differential is 
\begin{equation*}
\begin{split}
  d(\gamma \circ 1) 
  &= -4 r^3 dr \wedge d\theta
  = 8r^2 d\theta \wedge \bigl( \tfrac{1}{2}r dr - r^4 d\theta \bigr)
  = \alpha \wedge (\gamma \circ 1) \,,
\end{split}
\end{equation*}
for $\alpha = 8r^2 d\theta = -8(y dx - x dy)$. This shows that (C3) holds.

If (C2) were satisfied, there would be a nowhere vanishing function $f$ defined near the origin such that $f(\gamma\circ 1)$ is closed, hence exact.  The kernel of $\gamma$ is spanned by the vector field  
$\xi = \frac{\partial }{\partial \theta} + 2r^3 \frac{\partial }{\partial r}$.  It is easy to see from this expression that the trajectories of $\xi$, which are the integral manifolds of $\gamma \circ 1$, spiral away from the origin, so any function $g$ for which $dg=f(\gamma\circ 1)$ must be a constant equal to its value at the origin. This is a contradiction, which shows that (C3) does not imply (C2).
\end{Example}

\begin{Example}
\label{ex:cylinder}
Let $M = T^* S^1 \cong  S^1 \times \bbR$ be the cylinder with the canonical symplectic form $\omega = d\phi \wedge dz$. Let $A = M \times \bbR$ with the anchor $\rho: A \to TM$ mapping the constant section $a = 1$  to the vector field $\rho\circ 1 = \frac{\partial}{\partial \phi} - z \frac{\partial}{\partial z}$, whose trajectories spiral toward the zero section $z=0$ as they repeatedly circumnavigate the cylinder.
Since a symplectic vector field must be area preserving, the contraction property shows that $\rho\circ 1$ is not symplectic.  The same problem must hold for $a$ positive or negative, hence for any nonvanishing multiple of $a$.  So condition (C1) is not satisfied for this Lie algebroid.

Since $\rho\circ 1$ is nowhere vanishing the symplectic orthogonal of $\rho(A)$ is a 1-dimensional distribution, so it is regular and involutive. It follows from Proposition \ref{prop:CompatRegular} that this example satisfies condition (C2), which therefore does not imply (C1).
\end{Example}

\begin{Remark}
We can extend the argument that the last example does not satisfy (C1), showing that the only section of $A$ taken by $\rho$ to a symplectic vector field is the zero section.
Any section is of the form $f 1$ for some smooth function $f$. This section is taken by $\gamma$ to $f(\gamma\circ 1)$, where $\gamma \circ 1 = dz + z d \phi.$  For the section to be mapped to a symplectic vector field, we must have 
\begin{equation*}
 0 = d \bigl(f(\gamma \circ 1)\bigr)
 = df \wedge (\gamma \circ 1) + f\, d (\gamma \circ 1) 
 = \Bigl( \frac{\partial f}{\partial \phi} - z \frac{\partial f}{\partial z}
 - f \Bigr) \, d\phi \wedge dz \,,
\end{equation*}
which is equivalent to
\begin{equation*}
 (\rho\circ 1) \cdot f =\frac{\partial f}{\partial \phi} - z \frac{\partial f}{\partial z}
 = f \,.
\end{equation*}
This equation implies that, along each trajectory of $\rho \circ 1$, $f$ is either identically zero or grows exponentially.  One such trajectory is the zero section, along which $f$ must be periodic, so we see that $f$ must vanish there.  On the other hand, we see that any solution $f$ must blow up along the zero section if it is nonzero somewhere, so the only symplectic field in the image of $\rho$ is identically zero.  
\end{Remark}

The following equivalence is already contained in the work of Libermann \cite{Libermann:1983} on symplectically complete foliations.

\begin{Proposition}
\label{prop:PoissReg}
Let $(A,\rho)$ be an anchored vector bundle over a symplectic manifold such that $\rho$ is regular and $\rho(A)$ involutive (e.g.~if $\rho$ is the anchor of a Lie algebroid). Then conditions (C4) and (C5) are equivalent.
\end{Proposition}
\begin{proof}
Since $\rho(A)$ is regular and involutive, we can find in a neighborhood of every point local coordinates $(x^1, \ldots, x^p, y^1, \ldots, y^q)$ such that $\rho(A)$ is spanned by $\frac{\partial}{\partial x^1}, \dots, \frac{\partial}{\partial x^p}$. The hamiltonian vector fields generated by $(y^1, \ldots, y^q)$ are a local frame of the distribution $\rho(A)^\perp$. If the space of functions that depend only on the $y$-coordinates is closed under the Poisson bracket, then the distribution spanned by their hamiltonian vector fields is involutive. We conclude that (C5) implies (C4). The opposite direction was proved in Proposition \ref{prop:CompatStrength}.
\end{proof}

The following example shows that the condition in Proposition \ref{prop:PoissReg} that $\omega$ is non-degenerate is necessary:

\begin{Example}
\label{ex:PoissRegular1}
Let $M = \bbR^3$ with the presymplectic form $\omega = dx \wedge dy$. Let $A = M \times \bbR$ with the anchor $\rho: A \to TM$ mapping the constant section $a = 1$  to the vector field $\zeta := \rho\circ 1 = \frac{\partial}{\partial x} + z \frac{\partial}{\partial y}$. The anchor, the presymplectic form, and hence the dualized anchor are all regular. We will now show that this example satisfies (C5) but not (C4).

The presymplectic orthogonal $\rho(A)^\perp$ is a regular distribution which is spanned by the vector fields $\zeta$ and $\eta := \frac{\partial}{\partial z}$. Since $[\zeta,\eta] = - \frac{\partial}{\partial y}$, this distribution is not involutive (in fact, it is a version of the standard contact structure on $\bbR^3$), hence, $A$ does not satisfy (C4).

Let us now determine the hamiltonian pairs $(f,v)$ such that $f$ is annihilated by $\rho(A)$, that is, by $\zeta$. We have already seen in the proof of Proposition \ref{prop:CompatStrength} that this is the case iff $v$ takes its values in $\rho(A)^\perp$. Every vector field in $\rho(A)^\perp$ is of the form $v = g \zeta + h \eta$ for some $g,h \in C^\infty(M)$. The condition for $v$ to be locally hamiltonian is 
\begin{equation*}
\begin{split}
  0 
  &= d \iota_v \omega \\
  &= d \bigl[ g(dy - z\, dx) \bigr] \\
  &= dg \wedge (dy - z\, dx) - g\, dz \wedge dx \\
  &= 
  \Bigl(\frac{\partial g}{\partial x}
  -z \frac{\partial g}{\partial y}\Bigr)\, dx \wedge dy
  + \Bigl(g -z \frac{\partial g}{\partial z}\Bigr)\, dx \wedge dz
  - \frac{\partial g}{\partial z} \, dy \wedge dz
  \,,
\end{split}
\end{equation*}
which is satisfied if and only if $g = 0$. It follows that all hamiltonian pairs are of the form $(f,h\eta)$ where $f$ is a locally constant function. We conclude that the real vector space of hamiltonian functions annihilated by $\rho(A)$ is the space of locally constant functions, which is trivially closed under the Poisson bracket so that (C5) is satisfied.
\end{Example}

The next example,  a variant of Example~\ref{ex:PoissRegular1}, shows that the condition in Proposition \ref{prop:PoissReg} that $\rho(A)$ be involutive is necessary as well.

\begin{Example}
\label{ex:PoissRegular2}
Let $M = \bbR^4$ with the symplectic form $\omega = dx \wedge dy + dz \wedge dw$. Let $A = M \times \bbR^2$ with the anchor $\rho: A \to TM$ mapping the constant sections $(1,0)$ and $(0,1)$ to the vector fields $\zeta := \rho\circ (1,0) = \frac{\partial}{\partial x} + z \frac{\partial}{\partial y}$ and $\eta := \rho \circ (0,1) = \frac{\partial}{\partial z}$. The distribution spanned by $\zeta$ and $\eta$ is lagrangian, so that $\rho(A)^\perp = \rho(A)$. A calculation analogous to the one in example~\ref{ex:PoissRegular1} shows that all hamiltonian pairs $(f,v)$ for which $f$ is annihilated by $\rho(A)$ are of the form $\bigl(f,\frac{\partial f}{\partial w} \frac{\partial}{\partial z} \bigr)$, where $f$ is a smooth function of the coordinate $w$ only. We conclude that (C5) holds but not (C4).
\end{Example}

\section{Compatibility via connections}
\label{sec:connections}

It turns out that connections provide a basis-free way of interpreting some of the compatibility conditions above; this interpretation will lead us to choosing (C3) as our preferred condition.

\subsection{A brief review of connections}
\label{sec:ConnectionReview}
Although much of the material in this section is standard, we include it for the convenience of the reader and to set forth the notation which we will be using.

\subsubsection{Ordinary connections}

A linear connection on a vector bundle $A$ over a manifold $M$ may be given in terms of a covariant derivative $D: \Gamma(M,A) \to \Omega^1(M) \otimes \Gamma(M,A)$, satisfying $D(fa) = df \otimes a + f\,Da$ for every function $f$ and every section $a$ of $A$. For every vector field $v \in \calX(M)$, we denote by $D_v a := \iota_v Da$ the covariant derivative in the direction of $v$.

The connection $D$ on $A$ induces a dual connection on $A^*$, also denoted by $D$, which is defined by
\begin{equation*}
 d\langle \mu, a \rangle = \langle D\mu, a \rangle + \langle \mu, Da \rangle \,,
\end{equation*}
for all sections $\mu \in \Gamma(M,A^*)$ and $a \in \Gamma(M,A)$.  The dual connection extends to a degree 1 operator on the graded vector space 
\begin{equation*}
 \Omega^k(M,A^*) := \Gamma(M, \wedge^k T^* M \otimes A^*) \,,
\end{equation*}
of $A^*$-valued differential forms,\footnote{Of course, there is a similar extension of the original connection to the $A$-valued forms, but we will not be using it in this paper.}
as follows: $\Omega(M,A^*)$ is a graded module over the graded algebra $\Omega(M)$ with
\begin{equation*}
 \alpha \wedge (\beta \otimes \mu) = (\alpha \wedge \beta) \otimes \mu \,,
\end{equation*}
for all $\alpha, \beta \in \Omega(M)$ and $\mu \in \Gamma(M,A^*)$, where the tensor product of sections over $M$ is always understood to be the tensor product over $C^\infty(M)$. The covariant derivative extends by the graded Leibniz rule
\begin{equation}
\label{eq:ConnLeibniz}
 D(\alpha \otimes \mu) := 
 d\alpha \otimes \mu + (-1)^{\deg \alpha} \alpha \wedge D\mu
\end{equation}
to an $\bbR$-linear operator on $\Omega(M,A^*)$, which is well-defined since
\begin{equation*}
\begin{split}
 D(f\alpha \otimes \mu) 
 &= (df \wedge \alpha + f d\alpha) \otimes \mu + (-1)^{\deg\alpha} \alpha \otimes D\mu \\
 &= (-1)^{\deg\alpha} (\alpha \wedge df) \otimes \mu + d\alpha \otimes f\mu 
 + (-1)^{\deg\alpha} \alpha \otimes D\mu \\
 &= d\alpha \otimes f\mu + (-1)^{\deg\alpha} \alpha \otimes D(f\mu) \\
 &= D(\alpha \otimes f\mu)
\end{split}
\end{equation*}
for all $f \in C^\infty(M)$. The covariant derivative in the direction of a vector field $v \in \calX(M)$ extends to the graded vector space $\Omega(M,A^*)$ by a Cartan type formula,
\begin{equation*}
 D_v := \iota_v D + D \iota_v = [\iota_v, D] \,,
\end{equation*}
where the bracket denotes the graded commutator. The directional covariant derivative satisfies
\begin{equation}
\label{eq:DvSum}
\begin{split}
 D_v (\alpha \otimes \mu)
 &= \iota_v \bigl( d\alpha \otimes \mu + (-1)^{\deg \alpha} \alpha \wedge D\mu \bigr)
 + d \iota_v \alpha \otimes \mu + (-1)^{\deg \iota_v \alpha} \iota_v\alpha \wedge D\mu \\
 &= \Lie_v \alpha \otimes \mu + \alpha \otimes D_v \mu \,,
\end{split}
\end{equation}
where we have used that $\alpha \wedge D_v\mu = \alpha \otimes D_v\mu$. The extended curvature operator is defined by the usual expression
\begin{equation}
\label{eq:CurvDef}
 R(v,w) := [D_v, D_w] - D_{[v,w]} \,.
\end{equation}
Using Eq.~\eqref{eq:DvSum}, we see that
\begin{equation*}
 R(v,w) (\alpha \otimes \mu) = \alpha \otimes R(v,w) \mu \,,
\end{equation*}
that is, the curvature operator is $\Omega(M)$-linear. Identifying a section $\mu \in \Gamma(M,A^*)$ with the $0$-form $1 \otimes \mu \in \Omega^0(A^*)$ and using that $D_v \mu = \iota_v D\mu$, we obtain for the action of the curvature operator
\begin{equation*}
\begin{split}
 R(v,w) \mu
 &= ( \iota_v D \iota_w D - \iota_w D \iota_v D - \iota_{[v,w]}D ) \mu \\
 &= \bigl( \iota_v (D_w -  \iota_w D) - (D_w - D \iota_w ) \iota_v - \iota_{[v,w]}  \bigr)D\mu \\
 &= \iota_w \iota_v D^2\mu + (\iota_v D_w - D_w \iota_v - \iota_{[v,w]})D\mu + D \iota_v \iota_w D\mu \\
 &= \iota_w \iota_v D^2\mu \,,
\end{split}
\end{equation*}
where the second term vanishes due to Eq.~\eqref{eq:DvSum} and the the third term for degree reasons. We see that, on 1-forms $\mu \in \Omega^1(M,A^*) = \Gamma(M, A^*)$, the curvature operator acts by the square of the covariant derivative. When we apply $D^2$ to a general form we obtain
\begin{equation*}
\begin{split}
 D^2(\alpha \otimes \mu)
 &= D\bigl( d\alpha \otimes \mu + (-1)^{\deg\alpha} \alpha \wedge D\mu \bigr) \\
 &= d^2\alpha \otimes \mu + (-1)^{\deg\alpha + 1} d\alpha \wedge D\mu \\
 &\quad +(-1)^{\deg\alpha} d\alpha \wedge D\mu + (-1)^{2\deg\alpha} \alpha \wedge D^2 \mu \\
 &= \alpha \wedge D^2 \mu \,,
\end{split}
\end{equation*}
which shows that $D^2$ is $\Omega(M)$-linear. It follows that $D^2 = 0$ iff $D^2\mu = 0$ for all $\mu \in \Gamma(M, A^*)$, which is the case iff the connection $D$ has vanishing curvature.

\subsubsection{Associated Lie algebroid connections}

Let $(A,\rho)$ be an anchored vector bundle. An \textbf{$A$-connection} on a vector bundle $E \to M$ is given by a covariant derivative operator \cite{EvensLuWeinstein:1999,Fernandes:2002}
\begin{equation*}
\begin{aligned}
  D: \Gamma(M,A) \times \Gamma(M,E)
  &\longrightarrow \Gamma(M,E) \\
  (a,e) &\longmapsto D_a e
  \,,
\end{aligned}
\end{equation*}
which is $C^\infty(M)$-linear in $A$ and which satisfies
\begin{equation}
\label{eq:Aconnect}
  D_a (fe) = f D_a e + (\rho a \cdot f)\, e  \,.
\end{equation}
We call $D_a e$ the covariant derivative in the ``direction'' of $a$. When $A$ is the tangent bundle, with $\rho$ the identity, an $A$-connection is just an ordinary connection.

For later use in Sections \ref{sec:Category} and \ref{sec:HamHReduced}, we recall here a construction \cite[Sec.~1.2]{CrainicFernandes:2003} which takes an ordinary connection $D$ on an anchored vector bundle $A$ to an $A$-connection $\check{D}$ on $TM$.  When $A$ is $TM$ as well, this produces what is known in differential geometry as the ``opposite connection'' (see e.g.~\cite[p.~129]{deLeonRodriques:1989}), so we will use this term as well in the general case.  

\begin{Definition}
\label{def:OppConn}
Let $(A,\rho)$ be an anchored vector bundle over $M$.  For any connection $D$ on $A$, we define an $A$-connection $\check{D}$ on $TM$, called the 
\textbf{opposite connection}, by
\begin{equation}
\label{eq:DcheckDef}
  \check{D}_a v := [\rho a, v] + \rho( D_v a )
  \,,
\end{equation}
for every section $a$ of $A$ and vector field $v$. It is straightforward to check that $\check{D}$ satisfies the defining relation~\eqref{eq:Aconnect} of an $A$-connection.
\end{Definition}

In addition to the anchor, we need a Lie algebroid bracket on $A$ to generalize some other constructions with connections like the following: For any $A$-connection $D$ on a Lie algebroid $A$ the \textbf{Lie algebroid torsion} is defined as \cite{Fernandes:2002}
\begin{equation}
\label{eq:DefTorsion}
 T(a,b) = D_a b - D_b a - [a,b].
\end{equation}
for all $a,b \in \Gamma(M,A)$. $T$ is an $A$-valued 2-form on $A$. 

To any ordinary connection $D$ on $E$ we can associate an $A$-connection by composition with the anchor
\begin{equation*}
  \hat{D}_{a} e := D_{\rho a} e \,.
\end{equation*}
When $E = A$, we obtain an $A$-connection on $A$. The Lie algebroid torsion of $\hat{D}$ is 
\begin{equation}
\label{eq:DefTorsion2}
T(a,b) = D_{\rho a} b - D_{\rho b} a - [a,b].
\end{equation}
From here on, we will simply refer to this Lie algebroid torsion as ``the torsion'' of $D$.

\subsection{Induced connections}
\label{section-induced}

We return now to the setting where $\omega$ is a presymplectic form on $M$, which combines with the anchor $\rho$ of the anchored vector bundle $A$ to give the dualized anchor $\gamma=\tilde{\omega}\circ \rho:A \to T^*M$ as introduced in Definition \ref{def:DualizedAnchor}. Let $\{a_1, \ldots, a_r \} \in \Gamma(U,A)$ be a local frame and
\begin{equation*}
 \gamma_k = \gamma a_k \,.
\end{equation*}
where we recall that $\gamma a_k \equiv \gamma \circ a_k$.  Properties (C2) and (C3) can be described in terms of the Pfaffian system $\{ \gamma_1, \ldots, \gamma_r\}$. In view of Proposition \ref{prop:CompatRegular} we shall not assume that this Pfaffian system is regular. Property (C3) holds if there are 1-forms $\Omega^j_i \in \Omega^1(U)$, $1 \leq i,j \leq r$ such that
\begin{equation}
\label{omegaequation}
 d\gamma_i = \Omega^j_i \wedge \gamma_j \,.
\end{equation}
The 1-forms $\Omega^j_i$ are not unique. For example, we could also choose $\tilde{\Omega}^j_i = \Omega^j_i + B^{jk}_i \gamma_k$, where $B^{jk}_i$ are smooth functions with $B^{jk}_i = B^{kj}_i$.

Moreover, the 1-forms $\Omega^j_i$ depend on the chosen frame. Let $\{ \tilde{a}_i \}$ be another local frame. The frames are related as $\tilde{a}_i = A^j_i a_j$ by a smooth matrix function $(A^j_i) : U \to \mathrm{GL}(n,\bbR)$. Let $\tilde{\gamma}_i := \iota_{\rho\tilde{a}_i}\omega$. Then 
\begin{equation*}
 d\tilde{\gamma}_i 
 = d \bigl( A^j_i \gamma_j \bigr)
 = dA^j_i \wedge \gamma_j + A^j_i \Omega^k_j \wedge \gamma_i
 = \tilde{\Omega}^l_i \wedge \tilde{\gamma}_l \,,
\end{equation*}
where
\begin{equation*}
 \tilde{\Omega}^l_i = dA^j_i (A^{-1})^l_j + A^j_i \Omega^k_j\, (A^{-1})^l_k \,.
\end{equation*}
This tells us that we can interpret the 1-forms $\Omega^j_i$ as corresponding to a local connection $D: \Gamma(U,A) \to \Omega^1(U) \otimes \Gamma(U,A)$ given by
\begin{equation*}
 D a_i = \Omega^j_i \otimes a_j \,.
\end{equation*}
We recall that, if $\{ U_p \to M \}$ is an open cover with a subordinate partition of unity $\{\chi_p: U_p \to [0,1]\}$, and if $D_p$ is a local connection on each $A|_{U_p}$, then $D = \sum_p \chi_p D_p$ is a globally defined connection on $A$.  Let the connection 1-forms of $D_p$ be denoted by $\Omega^i_{j,p} \in \Omega^1(U_p)$. Then the connection 1-forms of $D$ are given by $\Omega^j_i = \sum_p \chi_p \Omega^i_{j,p}$, so that $d\gamma_i = \sum_p \chi_p d\gamma_i = \sum_p \chi_p \Omega^j_{i,p} \wedge \gamma_j = \Omega^j_i \wedge \gamma_j$, which shows that $\Omega^j_i$ still satisfies Eq.~\eqref{omegaequation}. We can use this connection to study the compatibility conditions (C1), (C2), and (C3).

\subsection{Connections and the compatibility conditions}
\label{sec:CompatConn}

 Using the natural isomorphisms $\Hom(A,T^*M) \cong \Gamma(M,T^*M \otimes A^*) = \Omega^1(M,A^*)$, we can identify the dualized anchor $\gamma$ with an $A^*$-valued 1-form, which, by abuse of notation, we will also denote by $\gamma$.  If $\{a_1, \ldots, a_n\}$ is a local frame of $A$ and $\{\theta^1, \ldots, \theta^n\}$ is the dual frame, $\langle \theta^i, a_j \rangle = \delta^i_j$, and $\gamma_i := \gamma a_i$, then
\begin{equation}
\label{eq:gammaFrame}
 \gamma =  \gamma_i \otimes \theta^i \,.
\end{equation}
We can now reexpress condition (C3) in terms of the induced connection.

\begin{Proposition}
\label{prop:Connection1}
Let $(A, \rho)$ be an anchored vector bundle over a  presymplectic manifold. Condition (C3) holds if and only if there is a linear connection $D$ on $A$ such that
\begin{equation}
\label{eq:Dsymplectic}
 D \gamma = 0 \,.
\end{equation}
\end{Proposition}

\begin{proof}
Let $\gamma = \gamma_i \otimes \theta^i$ in some local frame as in Eq.~\eqref{eq:gammaFrame}. Assume that (C3) holds. Then $d\gamma_i$ is of the form~\eqref{omegaequation} and $Da_i = \Omega^j_i \otimes a_j$ defines a local connection. The dual connection is given by $D\theta^i = - \Omega^i_j \otimes \theta^j$. The covariant derivative of $\gamma$ is given locally by
\begin{equation*}
\begin{split}
  D\gamma 
  &= d\gamma_i \otimes \theta^i - \gamma_i \wedge D\theta^i \\
  &= d\gamma_i \otimes \theta^i 
    + (\gamma_i \wedge \Omega^i_j ) \otimes \theta^j \\
  &= (d\gamma_i - \Omega^j_i \wedge \gamma_j) \otimes \theta^i 
  \,,
\end{split}
\end{equation*}
which vanishes due to Eq.~\eqref{omegaequation}. By adding the local connections using a partition of unity, we obtain a globally defined connection $D$ on $A$ which satisfies $D\gamma = 0$.

By pairing $D\gamma$ with a section $a$ of $A$, we obtain
\begin{equation}
\label{eq:DgammaPaired}
\begin{split}
  \langle D\gamma, a \rangle
  &= (d\gamma_i - \Omega^j_i \wedge \gamma_j) \langle \theta^i, a \rangle
  \\
  &= d(\gamma_i a^i) + \gamma_i \wedge da^i  + \gamma_j \wedge \Omega^j_i a^i
  \\
  &= d(\gamma_i a^i) + \gamma_i \wedge (Da)^i
  \\
  &= d\langle\gamma, a\rangle + \langle \gamma, Da \rangle
  \,,
\end{split}    
\end{equation}
where $a^i = \langle \theta^i, a \rangle$ and
\begin{equation*}
 \langle \gamma, Da  \rangle = \gamma_i \wedge (Da)^i \,.
\end{equation*}
Assume now that we have a connection $D$ such that $D\gamma = 0$. Since $M$ is assumed to be second countable, the covariant derivative of a section $a \in \Gamma(M,A)$ is given by a finite sum $Da =  \sum_{i=1}^n \beta^i \otimes b_i \in \Omega^1(M) \otimes \Gamma(M,A)$. From $D\gamma = 0$ and Eq.~\eqref{eq:DgammaPaired} it follows that
\begin{equation*}
  d \langle\gamma, a \rangle 
  = - \langle \gamma, Da \rangle  
  = - \gamma_j \wedge \sum_{i=1}^n \beta^i \langle \theta^j, b_i \rangle
  =\sum_{i=1}^n \beta^i \wedge \gamma a_i 
  \,,
\end{equation*}
which lies in the ideal generated by $\gamma\bigl( \Gamma(M,A) \bigr)$.
\end{proof}

Based on the results above, it is natural to introduce the following terminology.

\begin{Definition}
A connection $D$ on an anchored vector bundle $(A,\rho)$ over a  presymplectic manifold $(M,\omega)$ is called a \textbf{presymplectically anchored connection} if $D\gamma = 0$, where $\gamma\in \Omega^1(M,A^*)$ is the dualized anchor $\gamma = \tilde\omega \circ \rho$. We will also say that $(A,\rho)$ is \textbf{presymplectically anchored with respect to} $D$.
\end{Definition}

Thus, condition (C3) is equivalent to the property of being presymplectically anchored with respect to some connection.

\begin{Remark}
In addition to the ones we have already seen, there are other natural conditions that are equivalent to $D\gamma = 0$. In Proposition \ref{prop:PointSymp}, we will show that $D$ is presymplectically anchored if and only if the anchor maps sections of $A$ that are horizontal (locally or at a point) to presymplectic vector fields on $M$. In Proposition \ref{prop:FolSympComplete}, we will see that a regular foliation of a symplectic manifold, when viewed as Lie algebroid, is presymplectically anchored with respect to some connection if and only if it is symplectically complete. Finally, Proposition \ref{prop:Dcheckomega} expresses the equation $D\gamma = 0$ as the invariance $\check{D}\omega = 0$ of the presymplectic form under the parallel transport of the opposite $A$-connection $\check{D}$ on $TM$. All this has motivated our terminology and our choice of (C3) as the most natural compatibility condition.
\end{Remark}

If $D\gamma = 0$, then Eq.~\ref{eq:DgammaPaired} implies that
\begin{equation}
\label{eq:Dsymp2}
 d \iota_{\rho a} \omega 
 = - \langle \gamma, Da  \rangle \,,
\end{equation}
for all sections $a \in \Gamma(M,A)$. This shows that, for a presymplectically anchored connection, every horizontal section of $A$ is mapped by $\rho$ to a presymplectic vector field.

\begin{Proposition}
\label{prop:Connection2}
Let $(A, \rho)$ be an anchored vector bundle over a  presymplectic manifold. Condition (C2) holds if and only if every point of $M$ has a neighborhood $U$ with a flat presymplectically anchored connection $D$ on $A|_U$.

\end{Proposition}
\begin{proof}
The proof is analogous to that of Proposition \ref{prop:Connection1}. Assuming (C2), we have for every point in $M$ a neighborhood $U$ with a local frame $\{a_i\} \subseteq \Gamma(U,M)$ such that $d\gamma_i = 0$. This means that we can choose $\Omega^j_i = 0$, so that the induced connection is flat over $U$.

Conversely, assume that for every point $m \in M$ we have neighborhood $U$ with a flat connection satisfying $D\gamma = 0$. Then there is a local frame $\{a_i \}$ of horizontal sections, $D a_i = 0$. It follows that, $d \iota_{\rho a_i} \omega = d \gamma_i = 0$.
\end{proof}

From here on, we will use presymplectically anchored connections as our main tool for studying the compatibility of Lie algebroids and presymplectic structures.

\section{Momentum sections for presymplectically anchored Lie algebroids}
\label{sec:MomPre}

For a presymplectic action $\rho$ of a Lie algebra $\frakg$ on  $(M,\omega)$, a momentum map is a linear map $\frakg \to C^\infty(M)$ (often viewed as a map $\mu: M \to \frakg^*$) which maps every element $a$ of the Lie algebra to a function $f$ for which $(f,\rho a)$ is a hamiltonian pair.  When $\omega$ is symplectic, such a map always exists locally. In terms of the vector bundle $A = \frakg \times M$, the momentum map can be viewed as a section $\mu \in \Gamma(M,A^*)$. However, $\langle \mu, a \rangle$ for $a \in \Gamma(M,A)$ is a generator for the vector field $\rho a$ only when $a$ is a constant section of $A$. This tells us that a good notion of momentum ``map'' for a presymplectically anchored Lie algebroid will have to involve the connection $D$ on $A$ if it is to generalize the usual notion of momentum map for actions of Lie algebras.

We will define below the notion of momentum section for a pair $(A,D)$ and then motivate the definition by proving that it reduces to the usual notion for Lie algebra actions when $D$ is the trivial connection. However, for a different connection $D$, an action Lie algebroid may have a $D$-momentum section even when the action fails to have a momentum map in the usual sense. We will give an example for this phenomenon and then show that condition (C2) guarantees the existence of local momentum sections. In the following section, we will begin using the Lie algebroid bracket structure on $A$ and will investigate conditions for compatibility of a momentum section with the bracket. 

\subsection{Momentum sections}

So far we have considered the case that $\gamma= \tilde{\omega}\circ\rho$ is $D$-closed. A natural next step is to require that $\gamma$ be $D$-exact, i.e.
\begin{equation}
\label{eq:MomentumMapDef}
 D\mu = \gamma 
\end{equation}
for some $\mu \in\Omega^0(M,A^*)= \Gamma(M,A^*)$. We will call such a $\mu$ a \textbf{$D$-momentum section} for $A$. Note that, since $D^2 \neq 0$ unless the connection is flat, Eq.~\eqref{eq:MomentumMapDef} does not imply that $D\gamma = 0$, which must be required independently. Thus, unlike in the case of an action Lie algebroid with the trivial connection, the existence of a momentum section does not in general imply that $A$ is presymplectically anchored with respect to $D$, though we will often make this assumption.

When we evaluate Eq.~\eqref{eq:MomentumMapDef} on a section $a \in \Gamma(M,A)$ and use the definition of the dual connection, we obtain
\begin{equation}
\label{eq:MomentumMapDef2}
 d\langle\mu, a \rangle - \langle \mu, D a \rangle
  = \langle D\mu,a \rangle 
  = \langle \gamma, a\rangle = \iota_{\rho a} \omega \,,
\end{equation}
which makes it explicit how the anchor and the presymplectic structure enter into Eq.~\eqref{eq:MomentumMapDef}. In a local frame $\{a_i \}$, this relation can be further spelled out in terms of the functions $\mu_i = \langle\mu, a_i \rangle$, the 1-forms $\gamma_i = \gamma a_i$, and the connection 1-forms defined by $Da_i = \Omega^j_i \otimes a_j$. Eq.~\eqref{eq:MomentumMapDef2} now takes the form
\begin{equation*}
 d\mu_i - \mu_j \Omega^j_i = \gamma_i \,.
\end{equation*}

\bigskip

Our definition of a $D$-momentum section is justified by the following proposition, which shows that, for the action of a Lie algebra, it reduces to the usual notion of momentum map.

\begin{Proposition}
Let $\rho: \frakg \times M \to TM$ be an action of the Lie algebra $\frakg$ on the  presymplectic manifold $M$. Let $D$ be the trivial connection on $A := \frakg \times M$. Then: (i) $A$ is presymplectically anchored with respect to $D$ if and only if the action of $\frakg$ is presymplectic; (ii) when the action is presymplectic, $\mu \in \Gamma(U, A^*) = C^\infty(U, \frakg^*)$ is a $D$-momentum section if and only if it is a momentum map in the usual sense.  
\end{Proposition}
\begin{proof}
(i) Let $a \in \frakg$ be viewed as a constant section of $A$. Since $a$ is horizontal, Eq.~\eqref{eq:Dsymp2} is satisfied iff $d \iota_{\rho a}\omega = 0$.
(ii) Eq.~\eqref{eq:MomentumMapDef2} is satisfied for every $a \in \frakg$ iff $d\langle\mu, a \rangle = \iota_{\rho a} \omega$, which is the usual definition of a momentum map. Since $D^2 = 0$, the existence of a momentum section implies that $D\gamma = 0$.
\end{proof}

For a presymplectic action of a Lie algebra $\frakg$ on $(M,\omega)$, a natural condition is that it be  weakly{\footnote{We will reserve the unmodified term ``hamiltonian'' for those actions having equivariant momentum maps.  Some authors do not use this convention, calling actions with equivariant momentum maps ``strongly hamiltonian.''}} hamiltonian in the sense that the presymplectic vector field associated to each element of $\frakg$ is hamiltonian.  This condition is equivalent to the existence of a momentum map.  We therefore make the following definition.

\begin{Definition}
A vector bundle $A$ over $M$ with presymplectically anchored connection $D$ is \textbf{weakly hamiltonian} when it admits a (global) $D$-momentum section.  It is \textbf{locally weakly hamiltonian} if every point in $M$ has a neighborhood on which the restriction of $A$ is weakly hamiltonian, possibly with different connections on different open subsets.
\end{Definition}

\begin{Remark}
As explained in the proof of Proposition \ref{prop:Connection1}, a set $\{D_i\}$ of local presymplectically anchored connections defined on a cover of $M$ can always be merged by a partition of unity to a global presymplectically anchored connection $D$. However, the locally defined $D_i$-momentum section $\mu_i$ will generally fail to be a momentum section for $D$.
\end{Remark}

The following example shows that, even if an action is not weakly hamiltonian in the usual sense, its action Lie algebroid may be weakly hamiltonian as an anchored vector bundle, via a $D$-momentum section in which the constant sections are not horizontal for $D$. 

\begin{Example}
\label{ex:cylinder2}
We revisit Example~\ref{ex:cylinder} of a trivial line bundle over the symplectic cylinder $(S^1\times \bbR, d\phi \wedge dz)$. In this example $\gamma \in \Gamma(M, T^*M \otimes A^*)$ is given by
\begin{equation*}
  \gamma = (z\, d\phi + dz) \otimes \theta \,,
\end{equation*}
where $\theta = 1$ is the constant section of the dual bundle $A^*$. 

For the constant section $a=1$ of $A$ itself, we have $\gamma\circ 1 =  z\, d\phi + dz$ and $d(\gamma\circ 1) = -d\phi \wedge (\gamma\circ 1)$ so that the induced connection from Section \ref{section-induced} is given by $D1 = -d\phi \otimes 1$. The connection is flat but not trivial, with holonomy group $\bbZ$. A local horizontal section is given by $\tilde{a} = e^\phi 1$, which is defined for $\phi \in (-\pi, \pi)$. We have
\begin{equation*}
  \iota_{\rho\tilde{a}} \omega 
  = e^\phi \bigl( z\, d\phi + dz \bigr) 
  = d (z e^\phi) \,,
\end{equation*}
so $\rho\tilde{a}$ is a symplectic vector field on $(-\pi,\pi) \times \bbR$. Since the connection has nontrivial holonomy, however, there is no global horizontal section of $A$ (other than the zero section), so there is also no global symplectic vector field in the image of $\rho$.

Let $\mu := z \theta \in \Gamma(M,A^*)$. The covariant derivative acts on $\theta$ by $D\theta = d\phi \otimes \theta$, so we get
\begin{equation*}
  D\mu = dz \otimes \theta + z \, d\phi \otimes \theta = \gamma \,.
\end{equation*}
We conclude that $\mu$ is a $D$-momentum section.
\end{Example}

\begin{Proposition}
\label{prop:C2impliesLocHam}
Let $(A, \rho)$ be an anchored vector bundle over a presymplectic manifold satisfying condition (C2). Then $(A,\rho)$ is locally weakly hamiltonian.
\end{Proposition}
\begin{proof}
According to Proposition \ref{prop:Connection2}, there is for every $m\in M$ a flat presymplectically anchored connection $D$ over some neighborhood $U$ of $m$. Then there is a local frame $\{a_i\}$ in $\Gamma(U,A)$ of horizontal sections, $Da_i = 0$. The dual frame $\{\theta^i\}$ in $\Gamma(U,A^*)$ is also horizontal because $\langle \theta^i, a_j \rangle = \delta^i_j$. Expressing $D\gamma = 0$ in this frame reads
\begin{equation*}
 D(\gamma_i \otimes \theta^i) = d\gamma_i \otimes \theta^i = 0 \,, 
\end{equation*}
so that $d\gamma_i = 0$. Choosing $U$ to be contractible, we can find functions $\mu_i \in C^\infty(U)$ such that $\gamma_i = d\mu_i$. Let $\mu := \mu_i \theta^i$. The short calculation:
\begin{equation*}
 D\mu 
 = d\mu_i \otimes \theta^i + \mu \wedge D\theta^i 
 = \gamma_i \otimes \theta^i
 = \gamma \,,
\end{equation*}
shows that $\mu$ is a local $D$-momentum section.
\end{proof}

\subsection{Bracket-compatible momentum sections}
\label{sec:HamMomMap}

So far, we have studied anchored vector bundles over  presymplectic manifolds. Now we involve the bracket operation for a Lie algebroid and will define a compatibility condition with presymplectic structures which reduces to the usual equivariance condition for momentum maps of Lie algebra actions.

Up to now, we have used the complex of differential forms with values in $A^*$, which uses a connection on $A$ but no bracket.
When $A$ is a Lie algebroid, we have in addition the Lie algebroid complex,
which encodes all the information in the bracket as well as the anchor.
It consists of the sections of the exterior algebra of the dual $A^*$, with differential the unique derivation
\begin{equation*}
 \DA : \Gamma(M, \wedge^\bullet A^*) \to \Gamma(M, \wedge^{\bullet+1} A^*) \,,
\end{equation*}
of graded commutative algebras that satisfies
\begin{align}
 (\DA f)(a) 
 &= \rho a \cdot f \,, \notag \\
\label{eq:defDA}
 (\DA \mu)(a,b) 
 &= 
 \rho a \cdot \langle \mu,b \rangle - 
 \rho b \cdot \langle \mu,a \rangle - 
 \langle \mu, [a,b] \rangle \,,
\end{align}
for all $f \in C^\infty(M)$, $\mu \in \Gamma(M, A^*)$, and $a,b \in \Gamma(M, A)$.  

In terms of this Lie algebroid differential, we state our compatibility condition of a momentum section with the Lie algebroid bracket as follows:

\begin{Definition}
\label{def:MomentumEquiv}
Let $A$ be a Lie algebroid over the presymplectic manifold $(M,\omega)$. A $D$-momentum section $\mu \in \Gamma(M,A^*)$ is called \textbf{
bracket-compatible} if
\begin{equation}
\label{eq:MomEqui}
 (\DA\mu)(a,b) = -\omega(\rho a,\rho b)
\end{equation}
for all $a,b \in \Gamma(M,A)$.  
\end{Definition}

Note that, although the condition that $\mu$ be a momentum section involves the connection $D$, the bracket-compatibility condition does not. In particular, bracket-compatibility is independent of whether $A$ is presymplectically anchored with respect to $D$.

\begin{Remark}
Let $\rho^*: \Omega(M) \to \Omega(A)$ denote the map of differential complexes, induced by the anchor, from the de Rham complex of $M$ to the Lie algebroid cohomology complex. In terms of $\rho^*$, Eq.~\eqref{eq:MomEqui} can be written succinctly as $\DA \mu = - \rho^* \omega$. A necessary condition for \eqref{eq:MomEqui} is that $\rho^* \omega$ be $\DA$-closed, but since $\DA \rho^* \omega =  \rho^* d\omega = 0$, this is always the case.
\end{Remark}

Eq.~\eqref{eq:MomEqui} is equivalent to the vanishing of \begin{equation}
\label{eq:MomEquiv2}
\begin{split}
 (\DA\mu)(a,b) + \omega(\rho a, \rho b)
 &=
 (\DA\mu)(a,b) + \iota_{\rho b} \iota_{\rho a} \omega  \\
 &=
 (\DA\mu)(a,b) + \iota_{\rho b} \langle D \mu, a \rangle \\
 &= 
 (\DA\mu)(a,b)
 + \iota_{\rho b} (d\langle \mu, a \rangle 
    - \langle \mu, D a \rangle)  \\
 &=  
 \rho a \cdot \langle \mu, b\rangle
 - \langle \mu, [a,b] \rangle
 - \langle \mu, D_{\rho b} a \rangle \,,
\end{split}
\end{equation}
where we have used the defining relation~\eqref{eq:MomentumMapDef2} of a momentum section and definition~\eqref{eq:defDA} of the Lie algebroid differential. For an action Lie algebroid with the trivial connection we thus retrieve the usual notion of equivariance:

\begin{Proposition}
A momentum map $M\to \frakg^*$ for the action of a Lie algebra $\frakg$ on a  presymplectic manifold $M$ is equivariant if and only if it is bracket-compatible in the sense of Definition \ref{def:MomentumEquiv} when considered as a section of the dual of the action Lie algebroid $\frakg\times M$, with the trivial connection.
\end{Proposition}
\begin{proof}
Let $\mu: M \to \frakg^*$ be a momentum map. Elements of $\frakg$ can be identified with the sections of the action Lie algebroid $A = \frakg \ltimes M$ that are horizontal with respect to the trivial connection $D$. Inserting $a,b \in \frakg$ into \eqref{eq:MomEquiv2}, we obtain
\begin{equation*}
 (\DA\mu)(a,b) + \omega(\rho a,\rho b) = 
  \rho a \cdot \langle \mu, b \rangle
  - \langle \mu, [a,b] \rangle \,.
\end{equation*}
The left hand side vanishes iff $\mu$ is bracket compatible in the sense of Definition \ref{def:MomentumEquiv}. The right hand side vanishes iff $\mu$ is equivariant with respect to the $\frakg$-action $\rho$ on $M$ and the adjoint action $\langle\mathrm{ad}_a^* \mu, b \rangle = \langle\mu, [a,b] \rangle$ on $\frakg^*$ (see e.g.~Sec.~4.5.18 in \cite{OrtegaRatiu:Momentum}).
\end{proof}

Since a Lie algebra action with an equivariant momentum map is called hamiltonian, we will extend the terminology to Lie algebroids with the following definition.

\begin{Definition}
A Lie algebroid $A$ over $M$ with a presymplectically anchored connection $D$ and a bracket-compatible $D$-momentum section is called \textbf{hamiltonian}. It is \textbf{locally hamiltonian} if every point in $M$ has a neighborhood on which the restriction of $A$ is hamiltonian, possibly with different connections on different open subsets.
\end{Definition}

\begin{Remark}
It is well known that a momentum map $\mu$ for an action of $\frakg$ on a symplectic manifold $M$ is equivariant if and only if it is a Poisson map.  Less well known, but also true, is that an equivariant momentum map for a presymplectic action is a forward Dirac map \cite{brahic2014integrability}.  It is not hard to verify that the converse is  true as well.  In fact, a forward Dirac map from a presymplectic manifold to a Poisson manifold is characterized by the property that the pullback of any function is hamiltonian, and that the Poisson bracket of two pullbacks is the pullback of their Poisson brackets.

The following propositions and the torsion formulation in Section \ref{subsec:torsion} provide analogous results for general Lie algebroids.
\end{Remark}

\begin{Proposition}
\label{prop:muHomFlat}
Let $A$ be a Lie algebroid over a presymplectic manifold $(M,\omega)$ and $D$ a flat connection on $A$. A $D$-momentum section $\mu$ is bracket-compatible if and only if, for  all horizontal local sections $a, b \in \Gamma(U,A)$ defined on open subsets $U\subseteq M$, we have
\begin{equation*}
 \langle \mu, [a,b] \rangle
 = - \{ \langle\mu, a\rangle, \langle\mu, b\rangle \} \,,
\end{equation*}
where the bracket is the Poisson bracket of hamiltonian functions (see~Section \ref{sec:conventions}).
\end{Proposition}
\begin{proof}
Since $a$ and $b$ are horizontal, $(\langle\mu, a \rangle, \rho a)$ and $(\langle\mu, b \rangle, \rho b)$ are hamiltonian pairs, so that $\rho a \cdot \langle\mu, b\rangle = \iota_{\rho a} d\langle\mu, b\rangle = \omega( \rho b, \rho a) = - \{ \langle\mu, a \rangle, \langle\mu, b \rangle\}$. The proposition now follows from Eq.~\eqref{eq:MomEquiv2}. 
\end{proof}

Proposition \ref{prop:muHomFlat} shows that the definition of bracket-compatible $D$-momentum sections looks like the usual definition when we evaluate it on horizontal local sections of the Lie algebroid. When the connection has nonzero curvature, however, there might not be any horizontal local sections other than the zero section. For every point $m \in M$, though, there is still a basis of local sections that are horizontal \emph{at} $m$. We can then express the defining relations of hamiltonian Lie algebroids by relations that hold at every point for all sections that are horizontal at that point.

\begin{Proposition}
\label{prop:PointSymp}
A Lie algebroid $A$ over a presymplectic manifold $(M,\omega)$ is presymplectically anchored with respect to a connection $D$ if and only if
\begin{equation}
\label{eq:pointHam}
\Lie_{\rho a} \omega \,\bigr|_m = 0
\end{equation}
for all $m \in M$ and all sections $a$ of $A$ that are $D$-horizontal at $m$.
\end{Proposition}
\begin{proof}
Assume that $A$ is presymplectically anchored with respect to $D$. Then Eq.~\eqref{eq:Dsymp2} implies Eq.~\eqref{eq:pointHam} when $a$ is horizontal at $m$. Conversely, assume that Eq.~\eqref{eq:pointHam} holds. For every point $m \in M$ there is a local frame $\{a_i\}$ such that every $a_i$ is horizontal at $m$. Let $a = f^i a_i$ be an arbitrary local section. Then
\begin{equation*}
\begin{split}
 d \iota_{\rho a} \omega \,\bigr|_m
 &= df^i \wedge \gamma a_i\,\bigr|_m + f^i d \iota_{\rho a_i} \omega \,\bigr|_m \\
 &= df^i \wedge \gamma a_i \,\bigr|_m - f^i \langle \gamma \wedge Da_i \rangle \,\bigr|_m \\
 &= - \langle \gamma \wedge Da \rangle \,\bigr|_m \,,
\end{split}
\end{equation*}
which is Eq.~\eqref{eq:Dsymp2} at an arbitrary point $m$.
\end{proof}

\begin{Proposition}
\label{prop:PointHam}
 
Let $A$ be a Lie algebroid over a presymplectic manifold $(M,\omega)$. Then $\mu \in \Gamma(M,A^*)$ is a $D$-momentum section if and only if
\begin{equation*}
  d \langle \mu, a \rangle \bigr|_m
  = \iota_{\rho a} \omega \,\bigr|_m \,, 
\end{equation*}
for all $m \in M$ and all $a \in \Gamma(M,A)$ that are horizontal at $m$. The momentum section is bracket-compatible if and only if
\begin{equation*}
 \langle\mu, [a,b] \rangle \bigr|_m
 = \rho a \cdot \langle \mu, b \rangle \,\bigr|_m 
\end{equation*}
for all $m \in M$ and all sections $a$, $b$ of $A$ that are horizontal at $m$.
\end{Proposition}
\begin{proof}
The proof is analogous to the proof Proposition \ref{prop:PointSymp}. For the first part we use Eq.~\eqref{eq:MomentumMapDef2}, for the second part Eq.~\eqref{eq:MomEquiv2}.
\end{proof}

\section{The zero momentum locus of a hamiltonian Lie algebroid}
\label{sec:ZeroLocus}

One of the most important constructions involving a hamiltonian action of a Lie group $G$ or its Lie algebra $\frakg$ on a symplectic manifold $(M,\omega)$ is symplectic reduction, which proceeds in two steps. In the first step, the symplectic form is pulled back to the zero locus $Z \subset M$ of the momentum map. (If $Z$ is not a smooth embedded submanifold, we can instead consider the clean zero locus $Z_\mathrm{cl}$, i.e.~the smooth points of the zero locus at which a vector is tangent to the submanifold if and only if it annihilates the differential of the momentum map.) For the second step, as was observed in \cite{MarsdenWeinstein:1974}, $Z_\mathrm{cl}$ is a coisotropic submanifold, and its characteristic distribution is spanned by the fundamental vector fields of the action. It follows that,  if the leaf space $Z_\mathrm{cl}/\frakg$ of the induced $\frakg$-action is a smooth manifold, then the presymplectic form on $Z_\mathrm{cl}$ descends to a symplectic form on $Z_\mathrm{cl}/\frakg$, which is then called the symplectic reduction of $M$. In this section, we will show how this procedure generalizes to hamiltonian Lie algebroids.

\subsection{Bracket-compatibility in terms of torsion}
\label{subsec:torsion}

First, we will give an equivalent characterization of the bracket-compatibility of a $D$-momentum section in terms of the torsion defined in Eq.~\eqref{eq:DefTorsion2}. To motivate our characterization, we note that, when $D$ is the trivial connection on an action Lie algebroid $\frakg \ltimes M$, the torsion is given on constant sections $a,b \in \frakg$ by $T(a,b) = -[a,b]$.  (Note that it is nonzero unless $\frakg$ is abelian.)  Thus, the Lie algebroid torsion can be viewed as a substitute, in the case of any Lie algebroid with connection, for the (negative of the) fibrewise Lie algebra bracket of an action Lie algebroid, as well as the fibrewise Lie algebra bracket on the kernel of the anchor.  In this general case, there may be no local horizontal sections, but every element of $A$ is still the value of a section which is horizontal {\em at its basepoint} $m$.  For any two such sections, we have $T(a(m),b(m))=-[a,b](m)$.

\begin{Proposition}
\label{prop:TorsionSymp}
A $D$-momentum section $\mu$ is bracket-compatible if and only if
\begin{equation}
\label{eq:TorsionSymp}
  \langle \mu, T(a,b) \rangle = \omega(\rho a, \rho b)
\end{equation}
for all sections $a$ and $b$ of $A$.
\end{Proposition}

\begin{proof}
The pullback of $\omega$ by $\rho \otimes \rho$ is given by
\begin{equation*}
\begin{split}
 \omega\bigl(\rho a, \rho b \bigr)
 &=
   \frac{1}{2}( 
   \iota_{\rho b} \iota_{\rho a} \omega
 - \iota_{\rho a} \iota_{\rho b} \omega
  ) \\
 &=
 \frac{1}{2}( 
  \iota_{\rho b} \langle D\mu, a \rangle 
 -\iota_{\rho a} \langle D\mu, b \rangle
 ) \\
 &=
 \frac{1}{2}( 
 \langle D_{\rho b} \mu, a \rangle 
-\langle D_{\rho a} \mu, b \rangle
 ) \,,
\end{split}
\end{equation*}
where we have used that $\mu$ is a momentum section. The Lie algebroid differential of $\mu$ and the torsion are then related as follows:
\begin{equation*}
\begin{split}
 (\DA \mu)(a,b)
 &= 
 \rho a \cdot \mu(b) - \rho b \cdot \mu(a) - \mu([a,b]) \\
 &= 
 \iota_{\rho a} d \langle \mu, b\rangle
-\iota_{\rho b} d \langle \mu, a\rangle - \mu([a,b]) \\
 &= 
 \langle D_{\rho a}\mu, b\rangle
-\langle D_{\rho b}\mu, a\rangle 
+\langle \mu, D_{\rho a}b - D_{\rho b}a - [a,b] \rangle \\
 &=
 -2 \omega(\rho a, \rho b) 
 + \langle \mu, T(a,b) \rangle \,.
\end{split}
\end{equation*}
Adding $\omega(\rho a, \rho b)$ on both sides yields
\begin{equation*}
  \omega(\rho a, \rho b)
  + (\DA \mu)(a,b)
  =
  - \omega(\rho a, \rho b) 
  + \langle \mu, T(a,b) \rangle \,.
\end{equation*}
The left hand side vanishes iff the momentum section is bracket-compatible. The right hand side vanishes iff Eq.~\eqref{eq:TorsionSymp} holds.
\end{proof}

\subsection{The zero locus of a bracket-compatible momentum section}

We arrive now at one of the main goals of our work, establishing that the zero locus of a bracket-compatible momentum section for a presymplectic manifold is coisotropic and invariant, as is well-known to be the case for the zero locus of a momentum map for a hamiltonian Lie algebra action.
We first make a couple of remarks.

Coisotropic submanifolds are usually considered in the context of Poisson manifolds, where they appear as the graphs of Poisson maps and as the source and target images of lagrangian subgroupoids in symplectic groupoids (see \cite{BlohmannWeinstein:PoissonHamLA}). But they may also be defined in presymplectic manifolds as submanifolds whose tangent spaces all contain their presymplectic orthogonals.   Note that these tangent spaces must, in particular, contain the characteristic spaces of the presymplectic form, which are tangent to the characteristic foliation.  This means that, at least locally, the coisotropic submanifolds are the inverse images of coisotropic submanifolds in the (symplectic) leaf space of the characteristic foliation.  In fact, for our eventual application to the constraints in general relativity, it is the image in the symplectic reduced space which is the constraint manifold of interest to us (see \cite{BlohmannWeinstein:GRHamLA}).

The momentum zero locus may be defined in terms of the ideal $I$ of functions on $M$ generated by the ``components'' of $\mu$; more precisely, $I := \{ \langle\mu, b\rangle \in C^\infty(M) ~|~ b \in \Gamma(M,A)\}.$  The zero locus $Z$ is the set of common zeros of the elements of $I$.  

\begin{Proposition}
\label{Zeroisotropic}
In a hamiltonian Lie algebroid, the zero locus $Z$ of the momentum section is invariant in the sense that every orbit which meets  $Z$  is contained in $Z$.  Moreover, each such orbit is an isotropic submanifold.  
\end{Proposition}

\begin{proof}
In a hamiltonian Lie algebroid, the Lie derivative of a function $\langle\mu, b \rangle$
with respect to $\rho a$ for a section $a$ of $A$ can, by Eq.~\eqref{eq:MomEquiv2}, be expressed as
\begin{equation*}
  \Lie_{\rho a} \langle\mu,b \rangle
  = \langle\mu, [a,b] + D_{\rho b } a \rangle \in I \,.
\end{equation*}
This shows that $I$ is invariant under the Lie derivative of every vector field in the image of the anchor. It follows that $I$ and hence its set $Z$ of common zeros are invariant under the flow of every vector field in the image of the anchor. We conclude that every orbit of $A$ that meets $Z$ is contained in $Z$. By Eq.~\eqref{eq:TorsionSymp}, the image of the anchor at each point of $Z$ is isotropic.
\end{proof}

Since the zero locus of a momentum section is not necessarily a smooth submanifold, for the coisotropic property we will restrict attention to the \textbf{clean zero locus} 
$Z_{\mathrm{cl}}$, which consists of points of smoothness where the tangent space of the zero locus is the entire zero space of the differential of the momentum section.  The clean zero locus can be identified in algebraic terms as the set of points $m$ for which there is a neighborhood $U$ on which $Z$ is a smooth submanifold and on which the defining ideal $I$ is no smaller than the ideal $I_Z\supseteq I$ consisting of {\em all} functions vanishing on $Z$.  

\begin{Remark}
Note that points near which $I_Z = I$ are not necessarily smooth points of $Z$
(e.g.~when $I$ is the ideal generated by $q^2 + p^2$ in the plane). 
\end{Remark}

\begin{Theorem}
\label{thm:CoisoPresymp}
The clean zero locus $Z_{\mathrm{cl}}$ of the momentum section for a hamiltonian Lie algebroid over a  presymplectic manifold is a coisotropic submanifold which is invariant under the Lie algebroid. When the manifold is symplectic, then the characteristic distribution of $Z_{\mathrm{cl}}$ is equal to the image of the anchor.
\end{Theorem}
\begin{proof}
Since $Z$ is determined by $I$, so is $I_Z$, and hence the latter is invariant under all the diffeomorphisms generated by the image of $\rho$.  It follows that the subset where they agree locally is invariant.  So is the set of smooth points of $Z$, and hence so is $Z_{\mathrm{cl}}$.

Now let $m$ belong to the clean zero locus. A vector $v \in T_m M$ is tangent to $\mu^{-1}(0)$ if and only if it annihilates the differential of $\mu$.  This means that, for all $a \in \Gamma(M,A)$,
\begin{equation*}
\begin{split}
 0 
 &= v \cdot \langle\mu, a \rangle
 = \iota_v d\langle \mu, a \rangle
 =   \iota_v (\langle D\mu, a \rangle + \langle \mu, Da\rangle)
 \\
 &=  \omega( \rho a, v) \,,
\end{split}
\end{equation*}
where in the last step we have used that $\mu$ is a momentum section and that $\mu$ vanishes at $m$. In other words, the tangent space is the presymplectic orthogonal of the image of the anchor: 
\begin{equation*}
 T_m (\mu^{-1}(0)) = \rho(A_m)^\perp \,.
\end{equation*}
Now we suppose that $\mu$ is bracket-compatible. By Proposition \ref{Zeroisotropic}, $\rho(A_m)$ is an isotropic subspace of $T_m M$,  that is, $\rho(A_m) \subseteq \rho(A_m)^\perp$. So the image of the anchor is contained in the tangent space to the zero locus.  Furthermore, since  $V \subseteq W$ implies $V^\perp \supseteq W^\perp$  even when $\omega$ is degenerate, we obtain 
$T_m(Z_{\mathrm{cl}}) = \rho(A_m)^\perp \supseteq (\rho(A_m)^\perp)^\perp = (T_m Z_{\mathrm{cl}})^\perp$,
and so  $Z_{\mathrm{cl}}$ is coisotropic.

When $\omega$ is non-degenerate, then at every $m \in Z_{\mathrm{cl}}$ the characteristic distribution is $T_m (\mu^{-1}(0))^\perp = (\rho(A_m)^\perp)^\perp = \rho(A_m)$, the image of the anchor.
\end{proof}

Theorem  \ref{thm:CoisoPresymp} shows how symplectic reduction works for a hamiltonian Lie algebroid $A$ over a symplectic manifold: Since the anchor is tangent to $Z_\mathrm{cl}$, the Lie algebroid can be restricted to $Z_\mathrm{cl}$. If the leaf space of the characteristic distribution of $A|_{Z_\mathrm{cl}}$ is smooth, then it is a symplectic manifold, called the \textbf{symplectic reduction of $A$}.

\section{Examples}
\label{sec:Examples}

Our definitions of presymplectic Lie algebroids and (bracket-compatible) momentum sections were motivated by the special case of action Lie algebroids, but there are many more examples.   We present some of them in this section.

\subsection{Lie algebra bundles}
\label{subsec:liealgebrabundles}

A Lie algebra bundle is a Lie algebroid $A$ with zero anchor. In this case $\gamma = 0$, so that $D\gamma = 0$ for every linear connection $D$, and $A$ is presymplectically anchored with respect to any $D$.  As momentum section we can choose the zero section, $\mu = 0$, which is always bracket-compatible. We conclude that a Lie algebra bundle is  hamiltonian in this trivial way for every choice of connection.

On the other hand, there are obstructions to finding non-zero momentum sections, and even stronger ones if they are required to be bracket-compatible.
For any connection $D$ on $A$, $\mu$ is a momentum section if and only if it is horizontal, $D\mu = 0$. This implies that, if $\mu$ does not vanish at $m \in M$, then it cannot vanish anywhere on $M$.  So $A^*$ (and hence $A$) must admit a nowhere-zero section; this is just a topological condition on the bundle.  Conversely, if $\mu$ is any nowhere zero section of $A$, one may use a splitting of $A$, with one summand the trivial line bundle spanned by $\mu$, to construct connections on $A$ for which $\mu$ is horizontal. 

Since $\rho = 0$ here, the condition~\eqref{eq:TorsionSymp} for the bracket-compatibility of a momentum section becomes
\begin{equation*}
  \langle\mu, [a,b] \rangle = 0 \,.
\end{equation*}
Since all $a_m, b_m \in A_m$ can be extended to sections $a$, $b$ of $A$, this equation holds if and only if $\mu$ vanishes on the first derived ideal $[A_m, A_m]$ of the Lie algebra $A_m$ for all $m \in M$.  When a fibre $A_m$ is semisimple then $[A_m, A_m] = A_m$. In that case, a momentum section $\mu$ can be bracket-compatible only if it vanishes at $m$, so that it must vanish on all of $M$. Consequently, if some fibre of $A$ is semisimple (or more generally has a trivial abelianization), then $\mu = 0$ is the unique bracket-compatible momentum section.

\begin{Remark}
Since the vanishing of the anchor implies that $\gamma = 0$, whatever $\omega$ might be, the compatibility conditions here do not depend on the presymplectic structure at all.   We will see in \cite{BlohmannWeinstein:PoissonHamLA} that, for compatibility with Poisson structures, the image in $TM$ of this structure does intervene in the compatibility conditions.
\end{Remark}

\subsection{Tangent bundles of presymplectic manifolds}
\label{sec:tangent}

The anchor of the tangent bundle Lie algebroid is $\rho = \id_{TM}$, so the dualized anchor $\gamma$ is equal to  $\tilde{\omega}: TM \to T^*M$, $v \mapsto \iota_v \omega$, viewed as an element of $\Gamma(M, T^* M \otimes A^*) = \Gamma(M, T^* M \otimes T^* M)$. The condition $D\gamma = 0$ for $TM$ to be $D$-presymplectic  can be written in the following form:

\begin{Lemma}
The tangent bundle of a presymplectic manifold is presymplectically anchored with respect to a connection $D$ if and only if 
\begin{equation}
\label{eq:TangSympAnch}
  (D_v \omega)(u,w) - (D_w \omega)(u,v) + \omega(u, T(v,w)) = 0 
\end{equation} for all vector fields $u$, $v$, and $w$. 
\end{Lemma}

\begin{proof}
For a general Lie algebroid $A$ over a presymplectic manifold $(M,\omega)$ the dualized anchor satisfies
\begin{equation*}
\begin{split}
  \iota_w \iota_v \langle \gamma, Da \rangle 
  &= \iota_w \iota_v \bigl(\gamma_i \wedge (Da)^i \bigr)
  \\
  &= \iota_w \bigl( (\iota_v\gamma_i) (Da)^i 
  - \gamma_i (D_v a)^i\bigr)
  \\
  &= (\iota_v\gamma_i) (D_w a)^i 
  - (\iota_w \gamma_i) (D_v a)^i
  \\
  &= \iota_v \gamma(D_w a) 
  -  \iota_w \gamma(D_v a)
  \\
  &= \omega( \rho D_w a, v)
   - \omega( \rho D_v a, w)
  \,,
\end{split}    
\end{equation*}
for all vector fields $v$, $w$ on $M$ and all sections $a$ of $A$. Using this relation and Eq.~\eqref{eq:DgammaPaired}, we obtain
\begin{equation}
\label{eq:TangSympConn}
\begin{split}
  \iota_w \langle D_v\gamma, a \rangle
  &= \iota_w \iota_v \langle D \gamma, a \rangle
  \\
  &= \iota_w \iota_v \bigl( 
    d \langle \gamma, a \rangle 
    + \langle \gamma, Da \rangle \bigr)
  \\
  &= v \cdot \omega(\rho a, w) - w \cdot \omega(\rho a, v) 
     - \omega(\rho a, [v,w])
  \\
  &\quad{}
     - \omega( \rho D_v a, w) + \omega( \rho D_w a, v) \,.
\end{split}
\end{equation}
For the tangent Lie algebroid $A = TM$ the anchor is the identity $\rho = \id_{TM}$, so that we obtain
\begin{equation*}
\begin{split}
  \iota_w \langle D_v\gamma, a \rangle
  &= v \cdot \omega(a, w) - w \cdot \omega(a, v)
     - \omega(a, [v,w])
     - \omega(D_v a, w) + \omega(D_w a, v)
  \\
  &= (D_v \omega)(a, w) + \omega(D_v a, w) + \omega(a, D_v w)
  \\
  &\quad{}
   - (D_w \omega)(a, v) - \omega(D_w a, v) - \omega(a, D_w v)
  \\
  &\quad{}
     - \omega(a, [v,w])
     - \omega(D_v a, w) + \omega(D_w a, v)
     \\
     &=(D_v \omega)(a, w) - (D_w \omega)(a, v) +\omega\bigl(a,D_v w - D_w v - [v,w]\bigr)
  \\
  &= (D_v \omega)(a, w) - (D_w \omega)(a, v) + \omega\bigl(a, T(v,w) \bigr)
  \,.
\end{split}
\end{equation*}
Using the antisymmetry of $\omega$ and denoting $a$ by $u$ to emphasize that it is also a vector field on $M$, we obtain condition~\eqref{eq:TangSympAnch}.
\end{proof}

\begin{Remark}
One might have expected the condition for $TM$ be to be presymplectically anchored to be $D\omega=0$, the usual definition of a presymplectic connection.  This fails because, in our situation, the connection $D$ is applied only to the target of $\tilde{\omega}$ and not to the source.  On the other hand, we will see in Prop~\ref{prop:Dcheckomega} below that Eq.~\eqref{eq:TangSympAnch} has the simple formulation $\check{D}\omega = 0$, where $\check{D}$ is the opposite $A$-connection on $TM$ as defined in Definition \ref{def:OppConn}.
\end{Remark}

A special case where Eq.~\eqref{eq:TangSympAnch} is satisfied is when each term vanishes separately, that is, if the connection satisfies $D\omega = 0$ and if its torsion takes values in the characteristic distribution $TM^\perp$. In other words, \eqref{eq:TangSympAnch} is satisfied for what we may call a presymplectic connection with torsion in $TM^\perp$. In the symplectic case, there is an abundance of symplectic connections with vanishing torsion on every symplectic manifold \cite{Tondeur:1961}. In the presymplectic case, there can be a presymplectic connection only when $\omega$ is regular, since parallel translation along curves shows that the rank of $\gamma$, and hence the rank of $\omega$, must be the same everywhere. Vaisman proved in \cite{Vaisman:2000} that the regularity of the presymplectic form is sufficient for the existence of presymplectic connections whose torsion takes its values in $TM^\perp$. The conclusion is:

\begin{Proposition}
\label{prop:SympConn}
The tangent Lie algebroid of any regular presymplectic manifold is presymplectically anchored with respect to some connection.
\end{Proposition}

\begin{Remark}
In \cite{BCGRS:SymplecticConnections}, connections satisfying $D\omega = 0$ are called symplectic only if they are torsion free. We will follow other authors \cite{MarsdenRatiuRaugel:1991,Tondeur:1961} and not require the torsion to vanish, since this would, as a consequence of Proposition \ref{prop:TorsionSymp}, preclude the existence of bracket-compatible momentum sections.  Fortunately, as we will see, there is an abundance of connections with torsion which also presymplectically anchor any tangent Lie algebroid.
\end{Remark}

Another simple way to prove Proposition \ref{prop:SympConn} is to observe that, since $\omega$ is closed, the characteristic distribution $TM^\perp$ is involutive. It then follows from Proposition \ref{prop:CompatRegular} that the tangent Lie algebroid $TM$ satisfies (C3). Moreover, Proposition \ref{prop:C2impliesLocHam} tells us that $TM$ is \emph{locally} weakly hamiltonian.

We now study momentum sections.   Locally, it is easy to find
such sections explicitly. When the dimension of $M$ is $s$ and the rank of the presymplectic form is $2r$, the Darboux normal form in coordinates $(q^1,\cdots,q^r,p_1\cdots, p_r, z_1,\dots , z_{s-2r})$ is $\omega = dq^i \wedge dp_i$, so that
$$\gamma = dp_i \otimes dq^i - dq^i \otimes dp_i \,.$$
For the trivial connection $D$ in which the coordinate vector fields are all parallel, we have $D\omega = 0$ and $T = 0$. The 1-form
\begin{equation}
\label{eq:DarbouxMom1b}
  \mu = p_i dq^i - q^i dp_i
\end{equation}
satisfies $D\mu = \gamma$, which shows that $\mu$ is a local $D$-momentum section. 

\begin{Lemma}
\label{lem:ConnSymp1b}
Let $(M,\omega)$ be a presymplectic manifold and $D$ a connection on $TM$ that satisfies Eq.~\eqref{eq:TangSympAnch}. Another connection $D'$ satisfies Eq.~\eqref{eq:TangSympAnch} if and only if the difference tensor, the $C^\infty(M)$-bilinear map $\Gamma: \calX(M) \times \calX(M) \to \calX(M)$ defined by $\Gamma(v,w) = D'_v w - D_v w$, satisfies
\begin{equation}
\label{eq:ConnSymp2b}
  \omega\bigl(u, \Gamma(v,w) \bigr)
  = \omega\bigl(v, \Gamma(u,w) \bigr)  
\end{equation}
for all $u, v, w \in \calX(M)$.
\end{Lemma}

\begin{proof}
Let us denote the tensor on the right hand side of Eq.~\eqref{eq:TangSympAnch} by
\begin{equation*}
  A(u,v,w) := (D_v \omega)(u,w) - (D_w \omega)(u,v) + \omega(u,T(v,w)) \,,
\end{equation*}
and for $D'$ by $A'$. We want to compute the difference $A'-A$ in terms of $\Gamma$.

A short calculation shows that
\begin{equation*}
\begin{split}
  (D'_u \omega)(v,w)
  &= u \cdot \omega(v,w) - \omega(D'_u v, w) - \omega(v, D'_u w) \\
  &= (D_u \omega)(v,w) 
  - \omega\bigl( \Gamma(u,v), w \bigr) 
  - \omega\bigl( v, \Gamma(u,w) \bigr)
  \\
  &= (D_u \omega)(v,w) 
  + \omega\bigl(w, \Gamma(u,v) \bigr) 
  - \omega\bigl( v, \Gamma(u,w) \bigr)
  \,.
\end{split}
\end{equation*}
Similarly, the torsion $T'$ of $D'$ is related to the torsion $T$ of $D$ by
\begin{equation*}
  T'(v,w) = T(v,w) + \Gamma(v,w) - \Gamma(w,v) \,.
\end{equation*}
With these relations we obtain
\begin{equation*}
\begin{split}
  A'(u,v,w) - A(u,v,w) 
  &=
    \omega\bigl(w, \Gamma(v,u) \bigr) 
  - \omega\bigl( u, \Gamma(v,w) \bigr)
  \\
  &\quad{}
  - \omega\bigl(v, \Gamma(w,u) \bigr) 
  + \omega\bigl(u, \Gamma(w,v) \bigr)
  \\
  &\quad{}
  + \omega\bigl(u, \Gamma(v,w) - \Gamma(w,v) \bigr)
  \\
  &= 
      \omega\bigl(w, \Gamma(v,u) \bigr)
    - \omega\bigl(v, \Gamma(w,u) \bigr) \,.
\end{split}
\end{equation*}
Assume that $A(u,v,w) = 0$. Then $A'(u,v,w) = 0$ if and only if the right hand side of the last equation vanishes. After renaming the vector fields, we obtain condition~\eqref{eq:ConnSymp2b}.
\end{proof}

\begin{Proposition} 
\label{prop:NonvanishMomentumB}
Let $\mu \in \Omega^1(M)$ be a nowhere vanishing 1-form on a regular presymplectic manifold. If $\mu$ is basic in the sense that $\mu$ and $d\mu$ annihilate the characteristic distribution $TM^\perp$ (which is always the case when $\omega$ is symplectic) then it is a momentum section for some presymplectically anchored connection.
\end{Proposition}

\begin{proof}
Let $\mu \in \Omega^1(M)$ be a nowhere vanishing 1-form. Assume that $\mu$ annihilates $TM^\perp$. Then there is a vector field $n$ such that $\iota_n \omega = \mu$. Since $\mu$ is nowhere vanishing, there is another vector field $\bar{n}$ such that $1 = \langle \mu, \bar{n}\rangle = \omega(n, \bar{n})$. 

Given an arbitrary connection $D$, the covariant derivative of $\mu$ can be expressed as
\begin{equation*}
\begin{split}
  \langle D_v\mu, w\rangle 
  &= v \cdot \mu(w) - \mu(D_v w) \\
  &= v \cdot \omega(n,w) - \omega(n, D_v w) \,.
\end{split}
\end{equation*}
If $D'$ is another connection with $D'_v w - D_v w = \Gamma(v,w)$, then
\begin{equation*}
  \langle D'_v \mu, w\rangle - \langle D_v \mu, w\rangle
  = - \omega\bigl(n,\Gamma(v,w)\bigr) \,.
\end{equation*}

Now let $D$ be a presymplectic connection satisfying $D\omega=0$ with torsion taking its values in $TM^\perp$. As we see from Eq.~\eqref{eq:TangSympAnch}, $TM$ is $D$-presymplectically anchored. Since $D\omega = 0$, we have $\langle D_v \mu, w\rangle = \omega(D_v n, w)$. With this, the condition $\langle D'_v \mu, w\rangle = -\omega(v,w)$ for $\mu$ to be a $D'$-momentum section can be written as the following condition for $\Gamma$,
\begin{equation}
\label{eq:GammaMom}
\begin{split}
  0 
  &= \langle D'_v \mu, w\rangle + \omega(v,w) \\
  &= \langle D_v \mu, w\rangle - \omega\bigl(n,\Gamma(v,w)\bigr)
     + \omega(v,w) \\
  &= \omega(D_v n, w) - \omega\bigl(n,\Gamma(v,w)\bigr)
     + \omega(v,w) \\
  &= \omega(D_v n + v, w) 
   - \omega\bigl(n,\Gamma(v,w)\bigr)
   \,.
\end{split}
\end{equation}
In addition to this equation, condition~\eqref{eq:ConnSymp2b} for $\Gamma$ must be satisfied for $TM$ to be presymplectically anchored with respect to $D'$.

Conditions~\eqref{eq:GammaMom} and~\eqref{eq:ConnSymp2b} can be viewed as conditions for the 3-form $C(u,v,w) := \omega(u, \Gamma(v,w))$ given by
\begin{subequations}
\begin{align}
\label{eq:Ccond1}
  C(n,v,w) - \omega(D_v n + v, w) &= 0 \\
\label{eq:Ccond2}
  C(u,v,w) - C(v,u,w) &= 0 \,.
\end{align}
\end{subequations}
Our strategy to find a $C$ satisfying both conditions is the following: We start from $C = 0$. In the first step, we eliminate any non-vanishing term on the left hand side term of \eqref{eq:Ccond1}. In the second step we symmetrize $C$ in the first two arguments in order to satisfy \eqref{eq:Ccond2}. This will lead to a new non-zero term on the left hand side of \eqref{eq:Ccond2}. So we must repeat this procedure until both conditions are satisfied.

In the first step let 
\begin{equation*}
  C_1(u,v,w) 
  := \omega(u,\bar{n})\,\omega(D_v n + v, w) \,,
\end{equation*}
which satisfies \eqref{eq:Ccond1} because $\omega(n,\bar{n}) = 1$. The symmetrization of $C_1$ is given by
\begin{equation*}
\begin{split}
  C_2(u,v,w) 
  &:= C_1(u,v,w) + C_1(v,u,w) \\
  &= \omega(u,\bar{n})\,\omega(D_v n + v, w)
  + \omega(v,\bar{n})\,\omega(D_u n + u, w) \,,
\end{split}
\end{equation*}
which satisfies \eqref{eq:Ccond2}. However,
\begin{equation*}
  C_2(n,v,w) - \omega(D_v n + v, w)
  = \omega(v,\bar{n})\,\omega(D_n n + n, w) \,,
\end{equation*}
so that \eqref{eq:Ccond1} is not satisfied by $C_2$. The non-vanishing term can be cancelled by setting
\begin{equation*}
  C_3(u,v,w) 
  = C_2(u,v,w) - 
  \omega(u,\bar{n})\,\omega(v,\bar{n})\,\omega(D_n n + n, w) \,.
\end{equation*}
The additional term is manifestly symmetric in $u$ and $v$, so that $C_3$ satisfies both \eqref{eq:Ccond1} and \eqref{eq:Ccond2}. We conclude that $TM$ is $D'$-presymplectically anchored with momentum section $\mu = \iota_n \omega$ if $\Gamma$ satisfies
\begin{equation*}
\begin{split}
  \omega(u, \Gamma(v,w))
  &= C_3(u,v,w) 
  \\
  &= \omega(u,\bar{n})\,\omega(D_v n + v, w)
  + \omega(v,\bar{n})\,\omega(D_u n + u, w)
  \\
  &\quad{}
  - \omega(u,\bar{n})\,\omega(v,\bar{n})\,\omega(D_n n + n, w) 
  \\
  &= \omega\bigl( 
  u, \omega(D_v n + v, w) \bar{n}
  - \omega(v,\bar{n})\, \omega(D_n n + n, w) \bar{n}
  + \omega(v,\bar{n}) w \bigr)
  \\
  &\quad{}
  + \omega(v, \bar{n}) \, \omega(D_u n, w)
  \,.
\end{split}
\end{equation*}

The remaining question is whether a $\Gamma$ satisfying this equation exists, which is not immediately clear since $\omega$ may be degenerate. A sufficient condition is that $C_3(u,v,w)$ vanishes for all $u \in TM^\perp$ and $v,w \in TM$. For $u \in TM^\perp$ we have $C_3(u,v,w) = \omega(v,\bar{n})\, \omega(D_u n, w)$, which vanishes for all $v,w \in TM$ if $\omega(D_u n, w) = 0$. Therefore, we it suffices to show that $\iota_u \omega = 0$ implies $\omega(D_u n, w) = 0$. 

So let $u, w$ be vector fields with $\iota_u \omega = 0$. We get
\begin{equation*}
\begin{split}
  \iota_w \iota_u d\mu 
  &= u \cdot \mu(w) - w \cdot \mu(u) - \mu([u,w]) \\
  &= u \cdot \omega(n,w) - w \cdot \omega(n,u) - \omega(n, [u,w]) \\
  &= (D_u \omega)(n,w) + \omega(D_u n, w) + \omega(n, D_u w) - \omega(n, [u,w]) \\
  &= \omega(D_u n, w) + \omega(n, T(u,w) + D_w u) \\
  &= \omega(D_u n, w) + \omega(n, D_w u) \\
  &= \omega(D_u n, w) + w \cdot \omega(n,u) - (D_w \omega)(n,u) - \omega(D_w n, u) \\
  &= \omega(D_u n, w) \,,
\end{split}
\end{equation*}
where we have used that $D\omega = 0$ and $T = 0$. By assumption $\iota_u d\mu = 0$, so that $\omega(D_u n, w) = 0$ as well. This shows that a map $\Gamma: TM \times TM \to TM$ satisfying $\omega(u, \Gamma(v,w)) = C_3(u,v,w)$ exists. The conclusion is that $TM$ is $D'$-presymplectically anchored and that $\mu$ is a $D'$-momentum section.
\end{proof}

\begin{Remark}
The converse of Proposition \ref{prop:NonvanishMomentumB} is not true: Let $M$ be a presymplectic vector space with Darboux coordinates $(q,p,z)$. $TM$ is presymplectically anchored by the trivial connection for which $dq$, $dp$, and $dz$ are horizontal. Then $\mu = p\,dq - q\, dp + dz$ is a nowhere vanishing momentum section that does not annihilate the characteristic distribution $TM^\perp$ spanned by $\frac{\partial}{\partial z}$.
\end{Remark}

In the symplectic case the existence of a global weakly hamiltonian momentum section turns out to be equivalent to a simple topological condition:

\begin{Theorem}
\label{thm:TMweaklyhamB}
When $(M, \omega)$ is a symplectic manifold, the following are equivalent:
\begin{itemize}
\item[(i)] $TM$ is weakly hamiltonian.
\item[(ii)] $M$ is either non-compact or compact with non-negative Euler characteristic.
\end{itemize}
\end{Theorem}

\begin{proof}
We will first show that (ii) implies (i). 

If $M$ is non-compact or compact with zero Euler characteristic, then there is a nowhere vanishing section of $T^*M$. It follows from Proposition \ref{prop:NonvanishMomentumB} that $TM$ is weakly hamiltonian.

Suppose now that $M$ is compact with positive Euler characteristic $k$. We may construct a vector field with index $k$ as follows. Let $U_1$ be a disc in $M$ with Darboux coordinates $(q^1, \ldots, q^n, p_1, \ldots, p_n)$, and let $D_1$ be the connection on $TM|_{U_1}$ for which the coordinate vector fields are horizontal. The 1-form $\mu = p_i dq^i - q^i dp_i$ of Eq.~\eqref{eq:DarbouxMom1b} is a momentum section with respect to the trivial connection given by the Darboux coordinates. The vector field $n$ satisfying $\iota_n \omega = \mu$ is the negative Euler vector field $n = - q^i\frac{\partial}{\partial q^i} - p^i\frac{\partial}{\partial p^i}$ which has a zero of index 1 at the origin (since $M$ has even dimension). We repeat this construction to get vector fields of the same form on Darboux coordinate discs $U_2, \ldots, U_k$ with disjoint closures. The vector field we thus obtain on $U := U_1 \cup \ldots \cup U_k$ has total index $k$. We can extend this vector field to a vector field $\xi$ with non-degenerate isolated zeros on all of $M$. By the Poincar\'e-Hopf theorem the restriction of the vector field to $M \setminus \bar U$ must have index $0$.  The index of a non-degenerate zero is $\pm 1$, so that we must have the same number of zeros of index $1$ as of index $-1$. Since $M$ is assumed to be connected, so is $M \setminus \bar U$, so we can pairwise cancel the zeros of positive and negative index by modifying them along paths connecting the elements of a pair. The result is a vector field $v$ which is the negative Euler vector field of the Darboux coordinates on $U_1, \ldots, U_k$ and nowhere vanishing on $M \setminus \bar U$.

Using Proposition \ref{prop:NonvanishMomentumB}, we can find a connection $D'$ over a neighborhood of $M \setminus U$ such that $\mu := \iota_v \omega$ is a $D'$-momentum section over $M \setminus U$. By construction, $\mu$ is also a $D_i$-momentum section over every $U_i$. Since the condition~\eqref{eq:MomentumMapDef} for $\mu$ to be a momentum section does not involve derivatives of the  connection, we can use a partition of unity to merge $D'$ with the trivial connections $D_i$ on $U_i$ to a connection $D$ on all of $M$ with respect to which $\mu$ is a momentum section. 

For the converse direction, we let $n$ be the unique vector field such that $\mu = \iota_n \omega$. Let $x$ be a zero of $n$. The condition for $\mu$ to be a $D$-momentum section can be expressed in terms of $n$ as
\begin{equation*}
  \omega(D_v n, w) = \omega(-v, w) + (D_v\omega)(n,w) \,,
\end{equation*}
for all $w$, which is equivalent to
\begin{equation*}
  D_v n = -v + \tilde{\omega}^{-1}(\iota_n D_v\omega) \,.
\end{equation*}
When we evaluate this equation at $x$, the second term on the right hand side vanishes, so that we obtain
\begin{equation*}
  D_{v_x}n = -v_x  \,.
\end{equation*}
This shows that the derivative of $n$ at any zero $x$ is $-1$, so that $x$ is an isolated zero with index 1. Since this is true for every zero of $n$, the Poincar\'e-Hopf theorem implies that if $M$ is compact the Euler characteristic cannot be negative.

\end{proof}

We now move on to bracket compatibility.  Recall that a vector field $n$ on $(M,\omega)$ is called a \textbf{Liouville vector field} when $\Lie_n \omega = -\omega$.  (Some authors do not use the minus sign in the definition.)

\begin{Proposition}
\label{prop:TangentHamB}
Let $(M,\omega)$ be a manifold with a non-zero 
regular presymplectic form. The following are equivalent:
\begin{itemize}

\item[(i)] There is a Liouville vector field on $M$ that is nowhere tangent to $TM^\perp$.

\item[(ii)] There is a nowhere vanishing 1-form $\mu \in \Omega^1(M)$ that 
annihilates $TM^\perp$ and satisfies $d\mu = -\omega$. 
\end{itemize}
These equivalent conditions imply:
\begin{itemize}

\item[(iii)] The tangent Lie algebroid of $M$ is hamiltonian.

\end{itemize}
\end{Proposition}

\begin{proof}
Assume (ii).  Choose a complement $C$ to $TM^\perp$ in $TM$.  The restriction $\omega_C$ of $\omega$ to the vector bundle $C$ is nondegenerate, and the pullback map from  $(TM^\perp)^\circ$ to $C^*$ is an isomorphism.   Since $\mu$ annihilates $TM^\perp$  it can be viewed as a section of $(TM^\perp)^\circ$; since $\mu$ is nowhere zero, it pulls back to a nowhere zero section $\mu_C$ of $C^*$. Hence there is a (unique) nonzero section $n$ of $C\subseteq TM$ for which $i_n \omega_C = \mu_C$. Now let $v \in TM$ be decomposed as $v_C + v_\perp$, where $v_C \in C$ and $v_\perp \in TM^\perp$. Then $(i_n\omega)(v) = (i_n\omega) (v_C + v_\perp) = (i_n \omega )(v_C) = (i_n\omega_C)(v_C) = \mu_C (v_C) = \mu (v)$ (with the last equality due to the fact that $\mu$ annihilates $TM^\perp$).  So $i_n\omega = \mu$, and hence $n$ is a Liouville vector field. Moreover, since $\mu = \iota_n \omega$ is nowhere vanishing, $n$ is nowhere tangent to $TM^\perp$.
So (ii) implies (i).

Next, assume (i). Since the Liouville vector field $n$ is nowhere tangent to $TM^\perp$, $\mu = \iota_n\omega$ is nowhere vanishing. By definition, $\mu$ annihilates $TM^\perp$. Since $d\mu = \Lie_n \omega = -\omega$, the differential $d\mu$ annihilates $TM^\perp$ as well. So (i) implies (ii).

Assume (ii). Then we can apply Proposition \ref{prop:NonvanishMomentumB} which states that there is a connection $D$ such that $(TM, D, \mu)$ is a weakly hamiltonian structure for $TM$. Since $d\mu = -\omega$, the momentum section is bracket compatible. So (ii) implies (iii).
\end{proof}

\begin{Remark}
\label{rmk:HamNotmuvanish}
If (iii) of Proposition \ref{prop:TangentHamB} holds, there is a hamiltonian $D$-momentum section $\mu \in \Omega^1(M)$ which is bracket-compatible, $d\mu = - \omega$. By Eq.~\eqref{eq:TorsionSymp} and the assumption that $\omega$ is nowhere zero, $\mu$ must be nowhere vanishing. However, we could neither prove nor disprove that there is always such a $\mu$ that annihilates $TM^\perp$.
\end{Remark}

\begin{Remark}
In the case that $\omega = 0$, every tangent Lie algebroid is hamiltonian with momentum section $\mu = 0$ and any connection $D$, so that (iii) always holds. However, since $TM^\perp = TM$, every (Liouville) vector field is tangent to $TM^\perp$ and every 1-form $\mu$ annihilating $TM^\perp$ is zero, so that (i) and (ii) are never true.
\end{Remark}

\begin{Corollary}
\label{cor:HamSympSpaceB}
The tangent Lie algebroid of every presymplectic vector space is hamiltonian.
\end{Corollary}
\begin{proof}
Once again, we use the Darboux coordinate system
$$(q^1,\cdots,q^r,p_1\cdots, p_r, z_1,\cdots , z_{s-2r}),$$ in which 
$\omega = dq^i \wedge dp_i$, so that the 1-form $\mu = p^i dq_i + dp_1$ is nowhere vanishing, annihilates $TM^\perp = \mathrm{Span}\{\frac{\partial}{\partial z^1},\ldots, \frac{\partial}{\partial z^{s-2r}}\}$, and satisfies $d\mu = -\omega$. We conclude that (ii) of Proposition \ref{prop:TangentHamB} holds.
\end{proof}

\begin{Corollary}
\label{cor:TangLocHamB}
The tangent Lie algebroid of every regular presymplectic manifold is locally hamiltonian.
\end{Corollary}

In the symplectic case, Proposition \ref{prop:TangentHamB} takes the following simpler form:

\begin{Proposition}
\label{prop:SympTangHam}
For every symplectic manifold $(M,\omega)$ the following are equivalent:
\begin{itemize}

\item[(i)] There is a Liouville vector field on $M$ that is nowhere zero.

\item[(ii)] There is a nowhere vanishing 1-form $\mu \in \Omega^1(M)$ such that $d\mu = -\omega$.  

\item[(iii)] The tangent Lie algebroid of $M$ is hamiltonian. 

\end{itemize}
\end{Proposition}
\begin{proof}
Since $\omega$ is symplectic, we have $TM^\perp = 0$, so that the conditions in Proposition \ref{prop:TangentHamB} for the Liouville vector field not to be tangent to $TM^\perp$ and for $\mu$ to annihilate $TM^\perp$ are vacuous. Moreover, by Remark \ref{rmk:HamNotmuvanish} condition (iii) now implies (ii), so that all three conditions are equivalent.
\end{proof}

\begin{Remark}
The tangent bundle of a symplectic vector space has the structure of an action Lie algebroid $T\bbR^{2n} \cong \bbR^{2n} \ltimes \bbR^{2n}$, where $\bbR^{2n}$ is the commutative Lie algebra acting freely by the coordinate vector fields of Darboux coordinates. The action is by hamiltonian vector fields, but the action is only weakly hamiltonian.  (A bracket-compatible momentum section for the trivial connection would have to map the basis of the Lie algebra to $2n$ functionally independent, Poisson commuting functions. However, a maximum of $n$ such functions may exist for a completely integrable system.) We conclude that the connection of a hamiltonian structure, which is guaranteed to exist by Corollary \ref{cor:HamSympSpaceB}, cannot be the trivial one which is customarily attached to an action Lie algebroid. In fact if, for a connection making the Lie algebroid hamiltonian, there is any subbundle of $TM$ on which both the curvature and the torsion vanish, its rank may not exceed $n$.
\end{Remark}

We will see below that the tangent bundle of any exact (hence noncompact) symplectic manifold is hamiltonian.  Since the proof of that fact depends on a deep result in symplectic topology, we present next a class of symplectic manifolds whose tangent Lie algebroids are more easily proven to be hamiltonian.

\begin{Proposition}
If a manifold $Q$ is non-compact or compact with zero Euler characteristic, then the tangent Lie algebroid of $M = T^*Q$ is hamiltonian.
\end{Proposition}
\begin{proof}
Let $\theta$ be the canonical 1-form on $T^*Q$, which satisfies $d\theta = -\omega$. Let $n$ be the unique vector field satisfying $\iota_n \omega = \theta$. In local Darboux coordinates $\theta = p_i dq^i$ so that $n = - p_i \frac{\partial}{\partial p_i}$, which is a negative Liouville vector field, $\Lie_n \omega = -\omega$. We cannot apply Proposition \ref{prop:SympTangHam} yet, because $n$ vanishes on the zero section of the bundle $\pi: T^*Q \to Q$.

Since $Q$ is non-compact or compact with zero Euler characteristic, there is a nowhere vanishing vector field $\xi \in \calX(Q)$. Let $X$ be the hamiltonian vector field on $T^*Q$ generated by the function $f$ defined as $\iota_\xi \theta$ for each cotangent vector $\theta$. In local coordinates, we can write $\xi = \xi^i(q) \frac{\partial}{\partial q^i}$, so that $f =  p_i \xi^i$ and, hence, 
\begin{equation*}
  X = \xi^i \frac{\partial}{\partial q^i} 
  - p_i \frac{\partial \xi^i}{\partial q^j}
  \frac{\partial}{\partial p^i} \,.
\end{equation*}
Let $n' := n + X$. Since $X$ is a hamiltonian vector field, it satisfies $\Lie_X \omega = 0$, so that $\Lie_{n'} \omega = \Lie_n \omega + \Lie_X \omega = -\omega$. Moreover $n'$ projects by $T\pi$ to $\xi$, which is nowhere vanishing. Hence, $n'$ is nowhere vanishing as well. From Proposition \ref{prop:SympTangHam} it now follows that $TM$ is hamiltonian with momentum section $\mu = \iota_{n'} \omega$.
\end{proof}

For the general case, to prove that the tangent Lie algebroid of every exact symplectic manifold is hamiltonian, we apply recent results of Tang, Stratmann, and Karshon/Tang.  This work, inspired by the application to Lie algebroids, is also of considerable independent interest in symplectic topology.  

To find a nonvanishing primitive of an exact symplectic manifold, the first step is to begin with any primitive and to modify it by adding the differential of a function so that the new primitive $\lambda$ is transversal to the zero section.   (See for example, Lemma 3.4 of \cite{KarshonTang:2021}, for a proof.)  The zeros of this new primitive are then isolated.  Under the usual assumption that $M$ is separable, there are at most countably many zeros. In \cite{Stratmann:2020a}, Stratmann constructs a primitive with further special features and then pushes them out to infinity one at a time in such a way that there is a limiting primitive with no zeros.

Another general approach, which we proposed in a previous draft of this paper, is to remove rays containing the zeros of $\lambda$, either one at a time or all at once, showing the the resulting manifold is still symplectomorphic to $M$.   This has been carried out by Tang in \cite{Tang:2020}, with some conditions on the behavior at infinity of $M$, and then in general by Stratmann in \cite{Stratmann:2020b} and by Karshon and Tang in \cite{KarshonTang:2021}.

The conclusion follows:

\begin{Theorem}
\label{thm-hamiltoniantangent}
The tangent Lie algebroid of a symplectic manifold is hamiltonian if and only if the symplectic structure is exact.
\end{Theorem}

\subsection{Regular foliations}

Let $\calF$ be a regular foliation of a manifold $M$. The tangent spaces to the leaves form the integrable distribution $T\calF \subseteq TM$, which can be viewed as a Lie algebroid over $M$ with the embedding of $T\calF$ into $TM$ as anchor and the commutator of vector fields as Lie bracket. Let $\omega$ be a symplectic\footnote{We will restrict our attention here to the nondegenerate case in order to avoid issues arising from the interaction between $\calF$ and the characteristic foliation of $\omega$.} form on $M$. We recall that the foliation $\calF$ is called \textbf{symplectically complete} if the symplectic orthogonal $(T\calF)^\perp$ is an integrable distribution \cite{Libermann:1983}.  When this is the case, the foliation integrating $(T\calF)^\perp$ will be denoted by $\calF^\perp$, so that $(T\calF)^\perp = T\calF^\perp$.

\begin{Proposition}
\label{prop:FolSympComplete}
The tangent Lie algebroid $T\calF$ of a regular foliation $\calF$ of a symplectic manifold $M$ is symplectically anchored if and only if $\calF$ is symplectically complete.
\end{Proposition}

\begin{proof}
(C4) from Section \ref{def:CompatCond} is the condition that $\rho(A)^\perp = (T\calF)^\perp$ is involutive, which by Frobenius' theorem is equivalent to $(T\calF)^\perp$ being integrable. In other words (C4) is equivalent to $\calF$ being symplectically complete. Proposition \ref{prop:CompatRegular} tells us that in the regular case (C4) is equivalent to (C3) which is equivalent to $A$ being presymplectically anchored.
\end{proof}

Let us now turn to momentum sections. In a first step, we observe that $\tilde{\omega}: TM \to T^*M$ induces an isomorphism from the normal bundle $N \calF^\perp := TM/T\calF^\perp$ to the dual of the tangent bundle of $\calF$, given by
\begin{equation*}
\begin{aligned}
  \theta: N\calF^\perp 
  &\stackrel{\cong}{\longrightarrow} T^*\calF
  \\
  \langle \theta_{\pi(v)}, a \rangle &= \omega(v,a) 
\end{aligned}
\end{equation*}
for all $v \in TM$ and $a \in T\calF$, where $\pi: TM \to N\calF^\perp$ denotes the canonical projection. As is the case for the normal bundle of any foliation, we have the \textbf{Bott connection} on $N\calF^\perp$, which is the flat $T\calF^\perp$-connection given by 
\begin{equation*}
  D_v \pi(w) := \pi([v,w])
\end{equation*}
for all $v \in \Gamma(M, T\calF^\perp)$ and $w \in \calX(M)$. The isomorphism $\theta$ maps the Bott connection to a $T\calF^\perp$-connection on $T^*\calF$, which we will also call the Bott connection.

\begin{Proposition}
\label{prop:BottInduce}
Let $\calF$ be a symplectically complete foliation of $(M,\omega)$. Then the Bott connection on $T^*\calF \cong N\calF^\perp$ is given for $v \in \Gamma(M,T\calF^\perp)$ and $\mu \in \Gamma(M, T^*\calF)$ by
\begin{equation*}
  \langle \nabla_v \mu, a \rangle
  = \omega([v,w],a) \,,
\end{equation*}
where $w \in \calX(M)$ is a vector field such that $\langle\mu, a \rangle = \omega(w,a)$.
\end{Proposition}
\begin{proof}
This follows immediately from the definition of the isomorphism $\theta$.
\end{proof}

\begin{Proposition}
\label{prop:BottRestrict}
Let $D$ be a symplectically anchored connection on the tangent bundle $T\calF$ of a regular foliation of a symplectic manifold. Then the dual connection on $T^*\calF$ restricted to the direction of $T\calF^\perp$ is the Bott connection.
\end{Proposition}
\begin{proof}
From Eq.~\eqref{eq:TangSympConn} we deduce that the condition for the connection $D$ on $T\calF$ to be presymplectically anchored is
\begin{equation}
\label{eq:FolAnchConn}
\begin{split}
  0 &= v \cdot \omega(w, a) - w \cdot \omega(v, a) 
     - \omega([v,w], a)
  \\
  &\quad{}
     - \omega(w, D_v a) + \omega(v, D_w a) \,,
\end{split}
\end{equation}
for all $v, w \in \calX(M)$ and $a \in \Gamma(M, T\calF)$.

Let now $\mu \in \Gamma(M,T^*\calF)$ be given by $\langle \mu, a \rangle = \omega(w,a)$ for some vector field $w$. Then $\langle D_v \mu, a\rangle = v\cdot \langle\mu, a\rangle - \langle\mu, D_v a \rangle =  v \cdot \omega(w,a) - \omega(w, D_v a)$. With this we can write Eq.~\eqref{eq:FolAnchConn} as
\begin{equation}
\label{eq:FolAnchConn2}
  \langle D_v \mu, a \rangle
  = \omega([v,w],a) + w \cdot \omega(v,a) - \omega(v, D_w a)
  \,.
\end{equation}
For $v \in \Gamma(M, T\calF^\perp)$ the last two terms on the right hand side of Eq.~\eqref{eq:FolAnchConn2} vanish, so that we obtain $\langle D_v \mu, a \rangle = \omega([v,w],a)$. This is the Bott connection of Proposition \ref{prop:BottInduce}.
\end{proof}

\begin{Proposition}
\label{prop:FoliWeakHam}
Let $\calF$ be a symplectically complete foliation of $M$. Let $\mu$ be a nowhere vanishing section of $T^* \calF$. If $\mu$ is horizontal with respect to the Bott connection, then it is the momentum section for some symplectically anchored connection on $T\calF$.
\end{Proposition}
\begin{proof}
We follow the strategy of the proof of Proposition \ref{prop:NonvanishMomentumB}. Let $n$ be a vector field such that $\langle\mu, a\rangle = \omega(n,a)$. Since $\mu$ is nowhere vanishing, there is an $\bar{n} \in \Gamma(M, T\calF)$ such that $\langle\mu, \bar{n} \rangle = \omega(n, \bar{n}) = 1$.

Let $D$ be a symplectically anchored connection on $T\calF$, which exists by Proposition \ref{prop:FolSympComplete}. Using Eq.~\eqref{eq:FolAnchConn2}, we see that $\mu$ is a $D$-momentum section if and only if
\begin{equation*}
\begin{split}
  B(v,a) 
  &:= \langle D_v \mu, a \rangle + \omega(v,a)
  \\
  &= \omega([v,n],a) + \omega(v,a)
  + n \cdot \omega(v,a) - \omega(v, D_n a)
\end{split}
\end{equation*}
vanishes for all $v \in \calX(M)$ and $a \in \Gamma(M,T\calF)$. It is easy to check that $B(v,a)$ is $C^\infty(M)$-linear in both arguments, so that $B$ can be viewed as a bilinear function on $TM \times T\calF$.

By assumption, $\mu$ is annihilated by the covariant derivative of the Bott connection of Proposition \ref{prop:BottInduce}, which means that $\omega([v,n],a) = 0$ for all $v \in \Gamma(M,T\calF^\perp)$. Since the last three terms of $B(v,a)$ also vanish for $v$ tangent to $T\calF^\perp$, it follows that $B(v,a) = 0$ for all $v$ in $T\calF^\perp$. We conclude that for every $a \in T\calF$ there is an $\tilde{B}(a) \in T\calF$ such that $B(v,a) = \omega(v, \tilde{B}(a))$. This defines a bundle map $\tilde{B}: T\calF \to T\calF$.

We will now construct a symplectically anchored connection $D'$, such that $\mu$ is a $D'$-momentum section. The new connection is given by $D'_v a = D_v a + \Gamma(v,a)$ for some bilinear bundle map $\Gamma: TM \times_M T\calF \to T\calF$. From Eq.~\eqref{eq:FolAnchConn} we deduce that $D'$ is symplectically anchored, if and only if
\begin{equation}
\label{eq:FConnGamma}
  \omega\bigl(v, \Gamma(w,a) \bigr)
  = \omega\bigl(w, \Gamma(v,a) \bigr)
\end{equation}
for all $v, w \in TM$ and $a \in T\calF$. Moreover, the condition for $\mu$ to be a $D'$-momentum section is equivalent to
\begin{equation}
\label{eq:FConnGamma2}
  \omega\bigl(v,\Gamma(n,a) \bigr)
  = B(v,a)
  \,.
\end{equation}

Let $\Gamma$ be defined by
\begin{equation*}
  \Gamma(v,a) 
  := B(v,a)\,\bar{n} + \omega(v,\bar{n}) \tilde{B}(a)
  - \omega(v,\bar{n})\, B(n,a)\, \bar{n}
  \,.
\end{equation*}
Then we have
\begin{equation*}
\begin{split}
  \omega\bigl( w, \Gamma(v,a) \bigr)
  &= B(v,a)\, \omega(w,\bar{n})
  + \omega(v,\bar{n})\, B(w,a)
  - \omega(v,\bar{n})\, B(n,a)\,\omega(w,\bar{n})
  \\
  &= \omega\bigl( v, \Gamma(w,a) \bigr) \,,
\end{split}
\end{equation*}
so that Eq.~\eqref{eq:FConnGamma} is satisfied. This shows that $D'$ is symplectically anchored. For $w = n$ we obtain
\begin{equation*}
\begin{split}
  \omega\bigl( v, \Gamma(n,a) \bigr)
  &= B(v,a)\, \omega(n,\bar{n})
  + \omega(v,\bar{n})\, B(n,a)
  - \omega(v,\bar{n})\, B(n,a)\,\omega(n,\bar{n})
  \\
  &= B(v,a) \,,
\end{split}
\end{equation*}
so that Eq.~\eqref{eq:FConnGamma2} holds as well. This shows that $\mu$ is a $D'$-momentum section.
\end{proof}

The compatibility of a momentum section $\mu$ with the bracket of the Lie algebroid $A = T\calF$ can be written as $\DA\mu = -\rho^* \omega$. For $\langle\mu, a\rangle = \omega(n,a)$ we can write the Lie algebroid differential as
\begin{equation*}
  \DA\mu 
  = \DA \rho^* \iota_n \omega
  = \rho^* d\iota_n \omega 
  = \rho^* \Lie_n \omega
  \,.
\end{equation*}
It follows that the condition of bracket compatibility is equivalent to
\begin{equation}
\label{eq:FoliLiouville}
  \rho^*(\Lie_n \omega + \omega) = 0 \,.
\end{equation}
We thus arrive at the following sufficient conditions for $T\calF$ to be hamiltonian.  As expected, they reduce in the case where  $T\calF = TM$ to the ones found above.

\begin{Proposition}
\label{prop:FoliHam}
Let $\calF$ be a regular foliation of the symplectic manifold $(M,\omega)$.
\begin{itemize}

\item[(i)] $T\calF$ is symplectically anchored if and only if $\calF$ is symplectically complete.

\item[(ii)] $T\calF$ is weakly hamiltonian if, in addition to (i), there is a vector field $n$ that is nowhere tangent to $\calF^\perp$ and a symmetry of $\calF^\perp$, i.e.~$\Gamma(M,T\calF^\perp)$ is closed under the Lie bracket with $n$.

\item[(iii)] $T\calF$ is hamiltonian if, in addition to (i) and (ii), the pullback of $\Lie_n\omega + \omega$ to $\calF$ vanishes.

\end{itemize}
\end{Proposition}
\begin{proof}
Statement (i) is Proposition \ref{prop:FolSympComplete}. For (ii) we first observe that if $n$ is nowhere tangent to $\calF^\perp$, then $\mu := \rho^*\iota_n \omega$ is nowhere vanishing. Moreover, the Bott connection of $\mu$ is given by $\langle D_v \mu, a\rangle = \omega([v,n],a)$, for $v \in \Gamma(M,T\calF^\perp)$, so that $D_v\mu = 0$ iff $[v,n]$ lies in $\Gamma(M,T\calF^\perp)$ for every $v$. In that case we can apply Proposition \ref{prop:FoliWeakHam}. Statement (iii) follows from Eq.~\eqref{eq:FoliLiouville}.
\end{proof}

\subsection{Lie algebroids of rank 1}
\label{sec:Rank1}

By $\Supp \rho$ we denote the support of $\rho$ viewed as a section of $A^* \otimes TM \to M$. In other words, $\Supp \rho$ is the closure of the open set of points $m \in M$ where $\rho(A_m) \neq 0$.

\begin{Proposition}
\label{prop:Rank1Flat}
Let $(A,\rho)$ be an anchored vector bundle of rank 1 over the symplectic manifold $M$, symplectically anchored with respect to $D$.  If $\Supp\rho = M$, and $\mu$ is a $D$-momentum section, then $D$ is flat, and the support of $\mu$ is $M$ as well.
\end{Proposition}
\begin{proof}
Let $\mu$ be a $D$-momentum section. It follows from the defining properties of a $D$-momentum section $\mu$ that $0 = D^2 \mu =  \mu R$, where $R$ is the curvature 2-form. This shows that $R$ must vanish on the open set of all points where $\mu$ is non-zero and so, by continuity, $R$ must vanish also on the closure of $U$, that is, on the support of $\mu$.

If $\mu$ is zero on an open set $U \subseteq M$, then $\gamma = D\mu$ vanishes on $U$, so that $\rho$ also vanishes on $U$, since $\omega$ is nondegenerate.   This contradicts the assumption that $\Supp\rho = M$, and so the support of $\mu$ must be all of $M$.
\end{proof}

\begin{Corollary}
Let $A$ be a Lie algebroid of rank 1 over a symplectic manifold $M$ with $\Supp\rho = M$. Then $A$ is locally hamiltonian if and only if (C2) holds.
\end{Corollary}
\begin{proof}
Assume $A$ is locally hamiltonian, then Proposition \ref{prop:Rank1Flat} and Proposition \ref{prop:Connection2} imply (C2). The converse statement was proved in Proposition \ref{prop:C2impliesLocHam}.
\end{proof}

\begin{Proposition}
\label{prop:Rank1Obstr}
Let $A$ be a vector bundle of rank $1$ that is symplectically anchored with respect to the connection $D$. If $A$ has a $D$-momentum section, then the curvature of $D$ vanishes at the singular points of the anchor.
\end{Proposition}

\begin{proof}
Let $m \in M$ be a singular point. Assume that there is an open neighborhood $U \ni m$ on which the $D$-momentum section $\mu$ vanishes. Then $\gamma = D\mu = 0$ on $U$ which implies that $\rho = 0$ over $U$, so that the rank of $\rho$ is constant $0$. This contradicts the assumption that $m$ is singular, so that no such neighborhood  exists. It follows that every neighborhood of $m$ contains a point at which $\mu$ does not vanish. Thus, we can find a sequence $(m_i)$ converging to $m$ with $\mu_{m_i} \neq 0$. We conclude that $m$ lies in the support of $\mu$, so that $R$ must vanish at $m$. 
\end{proof}

As corollary to the last statement we obtain an obstruction to the existence of a local momentum section in the neighborhood of a singular point $m$. If the connection $D$ has nonzero curvature at $m$, then there cannot be a local $D$-momentum section. This still leaves open the much more difficult question of whether there can be another connection $D'$ for which $A$ is $D'$-symplectic, such that there is a local $D'$-momentum section.

\section{The category of hamiltonian Lie algebroids}
\label{sec:Category}

To further develop the concept of hamiltonian Lie algebroids, we will now introduce and study the notion of morphisms. This will lead to a unified framework for the various structures that make up a hamiltonian Lie algebroid. The theory developed here will be used in the example in  Section \ref{sec:ReduceIsotropy} on transitive Lie algebroids and in Section \ref{sec:HamHReduced}, where we study important classes of hamiltonian structures on quotients of Lie algebroids, as well as in Section \ref{sec:CohomInterp}, where we give a cohomological interpretation of hamiltonian Lie algebroids.

Stating the definition of a morphism of Lie algebroids directly in terms of the anchor and the Lie bracket is rather cumbersome. (See \cite{HigginsMackenzie:1990} and references therein.) The most succinct definition is in terms of the graded vector space $\Omega^q(A) = \Gamma^\bullet(M,\wedge^q A^*)$ of exterior forms on $A$ and the Lie algebroid differential $\DA$ \cite{Vaintrob:1997}: A morphism of vector bundles $\phi: A \to \tilde{A}$  over a map $\phi_0:M\to \tilde M$ is called a \textbf{morphism of Lie algebroids} if its pullback $\phi^*: \Omega^q(\tilde{A}) \to \Omega^q(A)$ is a map of differential complexes; i.e.~$\DA \phi^* = \phi^* \tilde{\DA}$.  

\begin{Example}
When $A = TM$ and $\tilde{A} = T\tilde{M}$ are tangent Lie algebroids, and $\phi = T\phi_0$ is the tangent map of a smooth map $\phi_0: M \to \tilde{M}$, then $\phi^*$ is the usual pullback of differential forms along $\phi_0$, which always intertwines the de Rham differentials.
\end{Example}

A connection on $A$ induces a dual connection on $A^*$, which is given by maps $D^p: \Omega^p(M,A^*) \to \Omega^{p+1}(M,A^*)$  between the spaces
\begin{equation*}
  \Omega^p(M,A^*) :=  \Gamma(M, \wedge^p T^* M \otimes A^*) \cong \Omega^p(M) \otimes_{C^\infty(M)} \Omega^1(A)
\end{equation*}
of $A^*$-valued differential forms. (Note, that the while the tensor product of vector bundles is fibre-wise over $\bbR$, the tensor product of the spaces of sections is over the ring of smooth functions.  We omit the subscript in the former case.) In order to formulate the connection and the hamiltonian Lie algebroid structure within the same algebraic framework and study their compatibility with vector bundle morphisms we are, therefore, led to the bigraded algebra
\begin{equation}
\label{eq:OmegaMA}
  \Omega^{p,q}(M,A) := \Omega^p (M) \otimes_{C^\infty(M)} \Omega^q(A) \,.
\end{equation}
We will show that the Lie algebroid structure, the connection, and the momentum section can all be simultaneously encoded either as derivations or as elements of $\Omega^{\bullet,\bullet}(M,A)$. This leads to a conceptually clear and useful notion of morphisms of hamiltonian Lie algebroids.

\subsection{Lie algebroid with a connection in terms of derivations}
\label{sec:CatLieAlgConn}

The bigraded ring $\Omega^{\bullet,\bullet}(M,A)$ can be viewed as ring of functions on the bigraded manifold $T[1,0]M \oplus A[0,1]$,
\begin{equation*}
\begin{split}
 \Omega^{p,q}(M,A) 
 &= \calO(T[1,0]M \oplus A[0,1])_{p,q} \\ 
 &= \Gamma(M, \wedge^p T^*M \otimes \wedge^q A^*)
 \,.
\end{split}
\end{equation*}
We may think of the elements of this bigraded algebra as exterior forms on the vector bundle $TM \oplus A$, as $\wedge^\bullet T^*M$-valued forms on the vector bundle $A$, or as $\wedge^\bullet A^*$-valued differential forms on $M$. The point of view of graded manifolds is explained in Section \ref{sec:CohomInterp}.

As we will now show, the anchor, the connection, and the Lie algebroid bracket can all be encoded as derivations on $\Omega(M,A)$. Any derivation $\delta$ of an algebra is uniquely determined by its action on a set of generators. $\Omega(M,A)$ is generated by functions $f \in \Omega^{0,0}(M,A)$, differential 1-forms $\tau \in \Omega^{1,0}(M,A)$, and Lie algebroid 1-forms $\theta \in \Omega^{0,1}(M,A)$. The derivation laws on these generators read
\begin{equation}
\label{eq:DerDef}
\begin{aligned}
  \delta(fg) &= (\delta f)g + f(\delta g)
  \\
  \delta(f\tau) &= (\delta f)\tau + f(\delta\tau)
  \\
  \delta(f\theta) &= (\delta f)\theta + f(\delta\theta)
  \,,
\end{aligned}
\end{equation}
Every linear map satisfying these relations is a local (first order differential) operator. Since, locally, $\Omega(M,A)$ is freely generated as a graded commutative $C^\infty(M)$-algebra, any linear map $\delta$ defined on functions and 1-forms that satisfies~\eqref{eq:DerDef} extends uniquely to a derivation of $\Omega(M,A)$.

\subsubsection{Interior derivative and anchor}

An important space of derivations on the algebra of exterior forms on a vector bundle is given by the insertion of  sections of the vector bundle as the first entry of the forms. These operations can be viewed a bigraded map
\begin{equation*}
\begin{aligned}
  \iota: \Gamma(M, T[1,0]M \oplus A[0,1])
  &\longrightarrow \Der(\Omega(M,A)) 
  \\ 
  X &\longmapsto \iota_X
  \,,
\end{aligned}
\end{equation*}
called the \textbf{interior derivative}, where $\Der(\Omega(M,A))$ is the space of bigraded derivations on $\Omega(M,A)$. When $X = v$ is a vector field on $M$, so that $\deg(v) = (-1,0)$, then $\iota_v$ is a derivation on $\Omega(M,A)$ of bidegree $(-1,0)$, which acts by $\iota_v \tau =  \langle \tau, v \rangle$ on $\tau \in \Omega^{1,0}(M,A)$ and by $\iota_v \theta = 0$ on $\theta \in \Omega^{0,1}(M,A)$. When $X = a$ is a section of $A$, then $\iota_a$ is a derivation of bidegree $(0,-1)$, which acts by $\iota_a \tau = 0$ and $\iota_a \theta = \langle \theta, a \rangle$. In both cases $\iota_X f = 0$, so that the derivation property~\eqref{eq:DerDef} is satisfied trivially.

It is useful to extend the interior derivative $\Omega(M,A)$-linearly to a map
\begin{equation*}
\begin{aligned}
  \iota: \Omega(M,A) \otimes_{C^\infty(M)} 
  \Gamma(M,T[1,0] \oplus A[0,1])
  &\longrightarrow \Der(\Omega(M,A))
  \\
  \phi \otimes X &\longmapsto
  \iota_{\phi \otimes X} := ( \psi
  \mapsto \phi\, \iota_X \psi )
  \,.
\end{aligned}
\end{equation*}
This map can be applied to an anchor $A \to TM$, which can be viewed as an element
\begin{equation*}
  \rho \in \Omega^{0,1}(M,A) 
  \otimes_{C^\infty(M)} \Gamma(M, T[1,0] M) \,.
\end{equation*}
We  thus obtain a derivation $\iota_\rho$ of bidegree $(-1,1)$. Let us make the definition of this derivation explicit:

\begin{Definition}
Let $(A,\rho)$ be an anchored vector bundle over $M$. The anchor gives rise to a bidegree $(-1,1)$ derivation $\iota_\rho$ of $\Omega(M,A)$ defined by
\begin{equation*}
  \iota_\rho f = 0
  \,,\quad
  \langle \iota_\rho \tau , a \rangle
  = \iota_{\rho a} \tau 
  \,,\quad
  \iota_\rho \theta = 0 \,,
\end{equation*}
for all $f \in \Omega^{0,0}(M,A)$, $\tau \in \Omega^{1,0}(M,A)$, $\theta \in \Omega^{0,1}(M,A)$, and all sections $a$ of $A$.
\end{Definition}

\subsubsection{Connection and Lie algebroid structure}

A connection on $A$ gives rise to a bidegree $(1,0)$ derivation on $\Omega(M,A)$ defined by
\begin{equation*}
   D f := df
   \,,\quad
   D \tau := d\tau
   \,,\quad
   D\theta := D\theta 
   \,,
\end{equation*}
where $df$ and $d\tau$ is the de Rham differential and $D\theta$ the covariant derivative. The derivation property~\eqref{eq:DerDef} follows from the derivation property of the de Rham differential and the Leibniz rule~\eqref{eq:ConnLeibniz} of the connection.

Let $\check{D}$ be the $A$-connection on $TM$ defined in Eq.~\eqref{eq:DcheckDef}. As is the case for every $A$-connection, $\check{D}$ gives rise to a dual $A$-connection on $T^*M$, also denoted by $\check{D}$, which is defined implicitly by
\begin{equation}
\label{eq:DualAconn}
  \rho a \cdot \langle \tau, v \rangle
  = \langle \check{D}_a \tau, v \rangle
  + \langle \tau, \check{D}_a v \rangle \,,
\end{equation}
for every differential 1-form $\tau \in \Omega^1(M)$ and vector field $v$. When $A$ is a Lie algebroid, the dual $A$-connection extends to a derivation on $\Omega(M,A)$ defined by
\begin{equation*}
   \check{D} f := \DA f
   \,,\quad
   \check{D} \tau := \check{D}\tau
   \,,\quad
   \check{D}\theta := \DA\theta 
   \,,
\end{equation*}
where $\DA f = \iota_\rho df$ and $\DA \theta$ is the Lie algebroid differential, and $\check{D}\tau$ the covariant derivative of the $A$-connection. The derivation property~\eqref{eq:DerDef} follows from the derivation property of the Lie algebroid differential and the Leibniz rule~\eqref{eq:Aconnect} of the $A$-connection.

We recall from Section \ref{subsec:torsion} that every connection $D$ induces an $A$-connection $\hat{D}$. Its torsion $T$ defined in Eq.~\eqref{eq:DefTorsion2} can be viewed as an $A$-valued exterior 2-form on $A$, i.e. as an element 
\begin{equation*}
  T \in \Omega^{0,2}(M,A) \otimes_{C^\infty(M)}
  \Gamma(M, A[0,1]),
\end{equation*}
so that $\iota_T$ is a derivation of bidegree $(0,1)$, defined by
\begin{equation*}
   \iota_T f = 0
   \,,\quad
   \iota_T \tau = 0
   \,,\quad
  \iota_T \theta := \langle \theta, T\rangle 
   \,.
\end{equation*}

\subsubsection{Commutators of the derivations}
We have the three bigraded derivations
\begin{align*}
  \iota_\rho &\in \Der(\Omega(M,A))_{-1,1}
  \\
  D &\in \Der(\Omega(M,A))_{1,0}
  \\
  \check{D} &\in \Der(\Omega(M,A))_{0,1} 
  \,,
\end{align*}
which encode the anchor, the connection, and the Lie algebroid bracket.

The bigraded commutator of a pair of bigraded derivations is again a bigraded derivation. Therefore, these three derivations generate a bigraded Lie subalgebra of the graded Lie algebra of all derivations on $\Omega(M,A)$. For example, we have already seen in Section \ref{sec:ConnectionReview} that $[D,D] = 2D^2 = 2(\id \otimes R)$ acts as the curvature operator on elements of $\Omega^{0,1}(M,A)$, which extends to a derivation on $\Omega(M,A)$. And the derivation $\iota_\rho$ is of total degree $0$, so that $[\iota_\rho,\iota_\rho] = 0$. 

The following result will be used in the proofs of Proposition \ref{prop:MorphcompatLA} and Theorem  \ref{thm:EquivExt}:

\begin{Proposition}
\label{prop:CommirhoD}
The differential $\check{D}$ can be expressed as
\begin{equation}
\label{eq:DcheckTorsion}
  \check{D} = [\iota_\rho, D] + \iota_T \,.
\end{equation}
\end{Proposition}

\begin{proof}
Since $\check{D}$, $[\iota_\rho, D]$, and $\iota_T$  are all derivations on $\Omega(M,A)$ and are thus uniquely determined by their action on the generators of $\Omega(M,A)$, we only need to show that Eq.~\eqref{eq:DcheckTorsion} holds when applied to functions $f \in \Omega^{0,0}(M,A)$, to differential 1-forms $\tau \in \Omega^{1,0}(M,A)$, and to exterior $A$-forms $\theta \in \Omega^{0,1}(M,A)$.

We start with a function $f$:
\begin{equation*}
\begin{split}
 \bigl\langle ([\iota_\rho, D] + \iota_T) f, 
   a \bigr\rangle
 &= \langle \iota_\rho df, a \rangle 
 = \rho a \cdot f = \langle \DA f,a\rangle
 \\
 &= \langle \check{D} f, a \rangle \,.
\end{split}
\end{equation*}
Next, we consider a differential 1-form $\tau$, for which we have
\begin{equation*}
  \iota_\rho D\tau = \iota_\rho d\tau \,.
\end{equation*}
Before we compute $D\iota_\rho\tau$, we observe that, in the definition of the dual connection of $\check{D}$, there appears an additional sign: for every differential 1-form $\tau$, section $a$ of $A$, and vector field $v$ on $M$ we have
\begin{equation*}
\begin{split}
  \langle \check{D}_a \tau, v \rangle
  &= \iota_v \check{D}_a \tau 
  = \iota_v \iota_a \check{D}\tau
  = - \iota_a \iota_v \check{D} \tau \\
  &= - \langle \check{D}\tau, v \otimes a \rangle \,.
\end{split}
\end{equation*}
With this relation, we obtain
\begin{equation*}
\begin{split}
  \langle D \iota_\rho \tau, v \otimes a \rangle
  &= \langle D_v \iota_\rho \tau, a \rangle
  \\
  &= v \cdot \langle \iota_\rho \tau, a \rangle
     - \langle \iota_\rho \tau, D_v a \rangle
  \\
  &= \langle \Lie_v \tau, \rho a \rangle
     + \langle \tau, \Lie_v \rho a \rangle
     - \langle \iota_\rho \tau, D_v a \rangle
  \\
  &= \langle \Lie_v \tau, \rho a \rangle
     - \langle \tau, [v, \rho a] + \rho(D_v a) \rangle
  \\
  &= \langle \iota_\rho \Lie_v \tau, a \rangle 
     - \langle \tau, \check{D}_a v \rangle
  \\
  &= \langle \iota_\rho (\iota_v d + d \iota_v) \tau, a \rangle 
     - \langle \tau, \check{D}_a v \rangle
  \\
  &= \langle \iota_\rho \iota_v d \tau, a \rangle
     + \rho a \cdot \langle\tau, v \rangle
     - \langle \tau, \check{D}_a v \rangle
  \\
  &= \langle \iota_v \iota_\rho d \tau, a \rangle
     + \langle \check{D}_a \tau, v \rangle
  \\
  &= \langle \iota_\rho d \tau - \check{D}\tau, v \otimes a \rangle \,,
\end{split}
\end{equation*}
where we have used the definition of the dual connection, the defining Eq.~\eqref{eq:DcheckDef} of the opposite $A$-connection $\check{D}$, Cartan's magic formula, and that $\iota_\rho \iota_v = \iota_v \iota_\rho$ because $\iota_\rho$ is of total degree $0$. Subtracting the two terms of the commutator yields $[\iota_\rho, D] \tau = \check{D}\tau - \iota_T \tau$, where we have used the condition $\iota_T \tau = 0$ in the definition of $\iota_\tau$.

Finally, we turn to an exterior 1-form $\theta$ on $A$. In order to apply Eq.~\eqref{eq:DcheckTorsion} to $\theta$ we first observe that 
\begin{equation*}
  \iota_{\rho a} = [\iota_a, \iota_\rho] \,,
\end{equation*}
which can be checked by letting the derivation $[\iota_a, \iota_\rho]$ act on the generators $\tau$ and $\theta$. Using this relation, we obtain 
\begin{equation*}
\begin{split}
  ([\iota_\rho, D] \theta)(a,b)   
  &= \iota_b \iota_a (\iota_\rho D - D \iota_\rho) \theta
  = \iota_b \iota_a \iota_\rho D \theta
  \\
  &= (\iota_b \iota_{\rho a} - \iota_a \iota_{\rho b}  
  + \iota_\rho \iota_b \iota_a)D\theta
   \\
  &= \langle D_{\rho a} \theta, b \rangle 
  - \langle D_{\rho b} \theta, a \rangle
  \\
  &= \rho a \cdot \langle\theta, b\rangle
   - \rho b \cdot \langle\theta, a\rangle
   - \langle \theta, D_{\rho a} b \rangle
   + \langle \theta, D_{\rho b} a \rangle
  \\
  &= \rho a \cdot \langle\theta, b\rangle
   - \rho b \cdot \langle\theta, a\rangle
   - \langle \theta, [a,b]\rangle
   \\
   &{}\quad + \langle \theta, [a,b]\rangle
   - \langle \theta, D_{\rho a} b \rangle
   + \langle \theta, D_{\rho b} a \rangle 
  \\
  &= (\DA \theta)(a,b) 
    - \langle\theta, T(a,b)\rangle
  \\
  &= (\check{D} \theta)(a,b) 
    - (\iota_T\theta)(a,b)
\end{split}
\end{equation*}
for all sections $a$, $b$ of $A$. This finishes the proof.
\end{proof}

\subsection{Morphisms of anchored vector bundles with connections}

A morphism $\phi: A \to A'$ of vector bundles is given explicitly by a commutative diagram of the form
\begin{equation*}
\xymatrix{
A \ar[r]^-{\phi_1} \ar[d] & A' \ar[d] \\
M \ar[r]^{\phi_0} & M'
}
\end{equation*}
The morphism induces a pullback operator
\begin{equation*}
  \phi^*: \Omega(M',A')
  \longrightarrow \Omega(M,A) \,,
\end{equation*}
which is the bigraded linear map defined by $\phi^*(\tau \otimes \theta) := \phi_0^* \tau \otimes \phi^* \theta$, where $\phi_0^* \tau$ is the usual pullback of the differential form $\tau$ on $M'$ along $\phi_0$, and where $\phi^* \theta$ is given by the commutative diagram
\begin{equation*}
\xymatrix@C+2ex{
\wedge^q A^*
& \wedge^q A'^* \ar[l]_-{\wedge^q \phi_1^*}
\\
M \ar[u]^{\phi^* \theta} \ar[r]^{\phi_0} & M' \ar[u]_{\theta}
}
\end{equation*}

A smooth map of (not necessarily linear) fibre bundles with (Ehresmann) connections is said to be \textbf{compatible with the connections} if its derivative maps horizontal vectors to horizontal vectors (or, equivalently, the map takes horizontal paths to horizontal paths). For a map of {\em vector} bundles with {\em linear} connections, this compatibility can be expressed in terms of the covariant derivatives as follows:

\begin{Proposition}
\label{prop:MorphcompatConn}
Let $(A,D)$ and $(A',D')$ be vector bundles with connections. A morphism $\phi: A \to A'$ of vector bundles is compatible with the connections if and only if $D \phi^* = \phi^* D'$.
\end{Proposition}

While the statement of this proposition is folklore knowledge, we could not find a proof in the literature, so we give one here:

\begin{proof}
$T\phi_1: TA \to TA'$ maps horizontal vectors to horizontal vectors if and only if the horizontal lifts $h: TM \times_M A \to TA$ and $h': TM' \times_{M'} A' \to TA'$ of the connections satisfy
\begin{equation}
\label{eq:ConnectCompat}
  T\phi_1 \, h(v_m, a_m)
  = h'( T\phi_0\, v_0, \phi_1(a_m))
\end{equation}
for all $m \in M$, $v_m \in T_m M$, and $a_m \in A_m$.

The decomposition of a tangent vector $\dot{a} \in T_{a_m} A$ into its horizontal and vertical part is given by
\begin{equation*}
  \dot{a}^\parallel = h(T\pi\, \dot{a}, a_m) \,,\quad
  \dot{a}^\perp =  \dot{a} - \dot{a}^\parallel \,,
\end{equation*}
where $\pi: A \to M$ is the bundle projection. Since $A$ is a vector bundle, we can identify the vertical part of $\dot{a}$ with an element in $A_m$. The dual connection is defined by requiring that paring a (locally or at a point) horizontal section $\theta^\parallel$ of $A^*$ with a horizontal section $a^\parallel$ of $A$ yields a constant function $\langle \theta^\parallel, a^\parallel \rangle$ on $M$. For a tangent vector $v_m \in T_m M$ represented by the smooth path $t \mapsto \gamma_t \in M$, we have
\begin{equation*}
\begin{split}
  v_m \cdot \langle \theta^\parallel, a^\parallel \rangle 
  &= \frac{d}{dt} \bigl\langle \theta^\parallel(\gamma_t), 
    a^\parallel(\gamma_t) \bigr\rangle_{t=0}
  \\
  &= 0
  \,.
\end{split}
\end{equation*}
For arbitrary sections $\theta$ and $a$ it follows that the derivative of $\langle \theta, a \rangle$ depends only on the horizontal components of the derivatives of $\theta$ and $a$, so that we obtain
\begin{equation*}
  v_m \cdot \langle \theta, a \rangle
  = 
    \bigl\langle (T\theta\,v_m)^\perp , a(m) \bigr\rangle
  + \bigl\langle \theta(m), (Ta\, v_m)^\perp v_m \bigr\rangle
  \,.
\end{equation*}
Using this relation, we can write the derivative in the direction of $v_m \in T_m M$ of the pairing of the pullback of a section $\theta'$ of $A'^*$ with a section $a$ of $A$  as follows:
\begin{equation*}
\begin{split}
  v_m \cdot \langle \phi^* \theta', a \rangle
  &= v_m \cdot \langle \theta' \circ \phi_0, \phi_1 \circ a \rangle
  \\
  &= \bigl\langle (T\theta'\, T\phi_0 \, v_m)^\perp, 
     \phi_1(a(m)) \bigr\rangle
   + \bigl\langle \theta'(\phi_0(m)), 
     (T\phi_1\, Ta\, v_m)^\perp \bigr\rangle
  \\
  &= \bigl\langle D'_{T\phi_0\, v_m} \theta', 
     \phi_1(a(m)) \bigr\rangle
  \\ &\quad{}   
   + \bigl\langle \theta'(\phi_0(m)), 
     T\phi_1\, Ta\, v_m 
     - h'\bigl(T\phi_0\, v_m, 
     \phi_1(a(m))\bigr) \bigr\rangle
  \\
  &= \iota_{v_m} \bigl\langle \phi^* D' \theta', 
     a\bigr\rangle 
   + \bigl\langle \theta'(\phi_0(m)), 
     T\phi_1\, Ta\, v_m 
     - h'\bigl(T\phi_0\, v_m, 
     \phi_1(a(m))\bigr) \bigr\rangle
   \,,
\end{split}
\end{equation*}
where we have used that the vertical part of the derivative of a section is its covariant derivative, $(T\theta'\,T\phi_0\, v_m)^\perp = D_{T\phi_0 v_m} \theta'$, and that $\phi$ being a morphism of bundles implies $T\pi'\, T\phi_1\, Ta \, v_m = T\phi_0\, T\pi\, Ta\, v_0 =  T\phi_0\, v_m$. On the other hand, we can express the derivative in terms of the covariant derivatives as
\begin{equation*}
\begin{split}
  v_m \cdot \langle \phi^* \theta', a \rangle
  &= \langle D_{v_m} \phi^* \theta', a \rangle
    + \langle \phi^* \theta', D_{v_m} a \rangle
  \\  
  &= \langle D_{v_m} \phi^* \theta', a \rangle
    + \bigl\langle \theta'(\phi_0(m)),
    T\phi_1 \,\bigl( Ta\, v_m - h(v_m, a(m)) 
    \bigr) \bigr\rangle 
  \,.
\end{split}
\end{equation*}
Subtracting the last two equations we obtain
\begin{equation*}
  \iota_{v_m} \bigl\langle (\phi^* D' - D \phi^*) 
    \theta', a \bigr\rangle
  =
  \bigl\langle \theta'(\phi_0),
    h'\bigl(T\phi_0\, v_m, \phi_1(a(m))\bigr)
    - T\phi_1 \, h(v_m, a(m))
    \bigr\rangle \,, 
\end{equation*}
which holds for all sections $\theta$, $a$, and all vectors $v_m$. We conclude that $\phi^* D' = D \phi^*$ if and only if Eq.~\eqref{eq:ConnectCompat} holds.
\end{proof}

\begin{Proposition}
\label{prop:MorphcompatAnch}
Let $(A,\rho)$ and $(A',\rho')$ be anchored vector bundles. A morphism $\phi: A \to A'$ of vector bundles is compatible with the anchors, $T\phi_0 \circ \rho = \rho' \circ T\phi_1$, if and only if $\iota_\rho \phi^* = \phi^* \iota_{\rho'}$.
\end{Proposition}

\begin{proof}
Evaluating the condition $T\phi_0 \circ \rho = \rho' \circ \phi_1$ at $a \in A$ and pairing it with a differential 1-form $\tau \in \Omega^1(M')$, we obtain the equivalent condition
\begin{equation*}
\begin{split}
  \langle \iota_\rho \phi^* \tau, a \rangle
  &= \langle \phi^* \tau, \rho(a) \rangle
  = \langle\tau, T\phi_0 \rho(a) \rangle
  = \langle\tau, \rho'(\phi_1(a)) \rangle
  = \langle \iota_{\rho'}
     \tau, \phi_1(a) \rangle
  \\
  &= \langle \phi^*\iota_{\rho'} \tau, a \rangle \,.
\end{split}
\end{equation*}
This shows that the compatibility is equivalent to $\iota_\rho \phi^*\tau = \phi^*\iota_{\rho'} \tau$ for all differential 1-forms on $M'$. By the derivation property of $\iota_\rho$ and $\iota_{\rho'}$, $\iota_\rho \phi^*\tau = \phi^*\iota_{\rho'} \tau$ then holds for differential forms of all positive degrees. Since $\iota_\rho$ and $\iota_{\rho'}$ annihilate all functions in $C^\infty(M)$ and $C^\infty(M')$, as well as all exterior $A$-forms in $\Omega(A)$ and $\Omega(A')$, respectively, we conclude that $\iota_\rho \phi^*  = \phi^*\iota_{\rho'}$ holds when applied to all forms in $\Omega(M,A)$.
\end{proof}

\begin{Proposition}
\label{prop:MorphcompatLA}
Let $(A,D)$ and $(A', D')$ be Lie algebroids with connections and $\phi: A \to A'$ a morphism of vector bundles that is compatible with the connections. Then $\phi$ is a morphism of Lie algebroids if and only if $\check{D} \phi^* = \phi^* \check{D'}$.
\end{Proposition}

\begin{proof}
Assume that $\phi$ is a morphism of Lie algebroids. This means that for every function  and every exterior $A$-form $\theta' \in \Omega(A)$ we have $\check{D} \phi^* \theta' = \DA \phi^* \theta' = \phi^* \DA' \theta' = \phi^* \check{D'} \theta'$. In particular, for every function $f' \in \Omega^0(A') = C^\infty(M')$ we have
\begin{equation*}
\begin{split}
  \iota_\rho \phi^* df' 
  &= \iota_\rho d \phi^* f'
  = \DA \phi^* f'
  = \phi^* \check{D}' f'
  \\
  &= \phi^* \iota_{\rho'} df' \,,
\end{split}
\end{equation*}
which implies that $\iota_\rho \phi^* = \phi^* \iota_{\rho'}$. Proposition \ref{prop:CommirhoD} implies that $\check{D}$ acts on differential forms by the commutator $[\iota_\rho, D]$. Since $\phi^*$ intertwines the interior derivatives of the anchors and by assumption the connections, it follows that $\check{D} \phi^* \tau' = \phi^* \check{D'} \tau'$ for all $\tau' \in \Omega(M')$. Differential forms and exterior $A$-forms generate $\Omega(M,A)$, which implies that the relation $\check{D} \phi^* = \phi^* \check{D'}$ holds when applied to any element of $\Omega(M,A)$.

Conversely, assume that $\check{D} \phi^* = \phi^* \check{D'}$. Since $\check{D}$ acts on exterior $A$-forms by the Lie algebroid differential, it follows that $\DA \phi^* \theta' = \phi^* \DA' \theta'$ for all $\theta' \in \Omega(A')$. We conclude that $\phi$ is a morphism of Lie algebroids.
\end{proof}

\subsection{Morphisms of hamiltonian Lie algebroids}
The structure of a hamiltonian Lie algebroid $(A,\omega, D, \mu)$ can be interpreted in terms of the bigraded ring $\Omega(M,A)$ as well. The presymplectic 2-form, the dual anchor, and the momentum section are elements
\begin{equation*}
  \omega \in \Omega^{2,0}(M,A) \,,\quad 
  \gamma \in \Omega^{1,1}(M,A) \,,\quad 
  \mu \in \Omega^{0,1}(M,A) \,. 
\end{equation*}
The connection is given by the derivation $D$ of bidegree $(1,0)$ and the Lie algebroid structure by the derivation $\check{D}$ of bidegree $(0,1)$. Moreover, the anchor is encoded in the derivation $\iota_\rho$. 

\begin{Proposition}
\label{prop:OmegaHamConds}
The conditions in Definition \ref{def:HamLAcomplete} for $A$ to be presymplectically anchored with respect to $D$, for $\mu$ to be a momentum section, and for $\mu$ to be bracket-compatible can be written as
\begin{equation}
\label{eq:HamConds2}
  D\iota_\rho\omega = 0 \,,\,\,
  D\mu = -\iota_\rho \omega,\, {\it and} \,\,
  \check{D}\mu 
  = - \tfrac{1}{2} \iota_\rho \iota_\rho \omega \,,
\end{equation}
respectively.
\end{Proposition}

\begin{proof}
We must explain the minus sign in the second equation and the factor $\frac{1}{2}$ in the third equation, which are a bit surprising. By Eq.~\eqref{eq:gammaFrame}, the dualized anchor $\gamma$, viewed as an $A^*$-valued 1-form on $M$, can be written in coordinates as 
\begin{equation*}
  \gamma 
  = \gamma_i \otimes \theta^i 
  = \rho^\alpha_i \omega_{\alpha\beta} dx^\beta \otimes \theta^i
\end{equation*}
as an element of the graded algebra $\Omega(M,A)$.  By definition, $\iota_\rho$ is the graded derivation on $\Omega(M,A)$ that is given by its action on local generators as
\begin{equation*}
  \iota_\rho dx^\alpha = \rho^\alpha_i\theta^i 
  \,,\quad
  \iota_\rho \theta^i = 0 \,.
\end{equation*}
Applying $\iota_\rho$ to $\omega = \frac{1}{2} \omega_{\alpha\beta} dx^\alpha \wedge dx^\beta \equiv \omega = \frac{1}{2} \omega_{\alpha\beta} dx^\alpha dx^\beta$, we obtain
\begin{equation*}
\begin{split}
  \iota_\rho \omega
  &= \omega_{\alpha\beta}
  dx^\alpha \iota_\rho dx^\beta
  = \omega_{\alpha\beta}  dx^\alpha \rho^\beta_i\theta^i
  = - \rho^\alpha_i \omega_{\alpha\beta}  dx^\beta \otimes \theta^i
  \\
  &= -\gamma \,.
\end{split}
\end{equation*}
It follows that $\gamma = - \iota_\rho \omega$, so that the condition $D\gamma = 0$ is equivalent to $D\iota_\rho \omega = 0$, and the condition $D\mu = \gamma$ is equivalent to $D\mu = -\iota_\rho \omega$.

Using $\iota_{\rho a} = [\iota_a, \iota_\rho]$, we compute for the third condition
\begin{equation*}
\begin{split}
  \iota_b \iota_a (\iota_\rho \iota_\rho \omega)
  &= 
  \iota_b (\iota_\rho \iota_a 
  + \iota_{\rho a}) \iota_\rho \omega)
  \\
  &= 
  (\iota_b \iota_\rho) (\iota_a \iota_\rho) \omega 
  + \iota_b \iota_{\rho a} \iota_\rho \omega
  \\
  &= 
  (\iota_b \iota_\rho) \iota_{\rho a} \omega 
  + (\iota_b \iota_\rho) \iota_{\rho a} \omega
  \\
  &= 
  2\omega(\rho a, \rho b)
  \,.
\end{split}
\end{equation*}
This shows that $\omega(\rho a, \rho b) = (\frac{1}{2} \iota_\rho \iota_\rho \omega)(a,b)$. It follows that the condition $(\DA\mu)(a,b) = -\omega(\rho a, \rho b)$ is equivalent to $\check{D}\mu = - \frac{1}{2}\iota_\rho \iota_\rho \omega$.
\end{proof}
}
For the first condition of Eqs.~\eqref{eq:HamConds2}, there is an equivalent condition in terms of $\check{D}$:

\begin{Proposition}
\label{prop:Dcheckomega}
An anchored vector bundle $A$ over a presymplectic manifold $(M,\omega)$ is presymplectically anchored with respect to the connection $D$ if and only if
\begin{equation*}
  \check{D} \omega = 0 \,.
\end{equation*}
\end{Proposition}
\begin{proof}
The dualized anchor is given by $\gamma = - \iota_\rho \omega$. Using Proposition \ref{prop:CommirhoD} and that $D\omega = d\omega = 0$, we get
\begin{equation*}
  D\gamma 
  = - D\iota_\rho \omega
  = [D, \iota_\rho] \omega
  = - \check{D} \omega \,,
\end{equation*}
so the left hand side vanishes if and only if the right hand side does.
\end{proof}

The conditions for a map of vector bundles to be compatible with hamiltonian Lie algebroid structures are now clear.

\begin{Definition}
Let $(A,\omega, D,\mu)$ and $(A', \omega', D', \mu')$ be (weakly) hamiltonian Lie algebroids over presymplectic manifolds. A map $\phi: A \to A'$ of vector bundles is a \textbf{morphism of (weakly) hamiltonian Lie algebroids} if it is a morphism of Lie algebroids and if the following three conditions hold:
\begin{itemize}

\item[(i)] $\omega = \phi_0^* \omega'$,

\item[(ii)] $D \phi^* = \phi^* D'$,

\item[(iii)] $\mu = \phi^* \mu'$.

\end{itemize}
\end{Definition}

\begin{Proposition}
\label{prop:HamPullback}
Let $(A,D)$ and $A', D')$ be Lie algebroids with connections. Let $\phi: A \to A'$ be a morphism of Lie algebroids that is compatible with the connections. Let $\omega'$ be a presymplectic form on the base manifold of $A'$, let $\mu'$ be a section of $A'^*$, and let $\omega := \phi^* \omega'$ and $\mu := \phi^* \mu'$ denote their pullbacks.

If $(A', \omega')$ is presymplectically anchored, then $(A,\omega)$ is presymplectically anchored. If $\mu'$ is a (bracket-compatible) $D'$-momentum section for $(A', \omega')$, then $\mu$ is a (bracket-compatible) $D$-momentum section for $(A, \omega)$. Moreover, when $\phi_1$ is surjective and $\phi_0$ a submersion, then the converses of these statements are also true.
\end{Proposition}

\begin{proof}
By Propositions \ref{prop:MorphcompatConn} and \ref{prop:MorphcompatLA} $\phi$ satisfies $D\phi^* = \phi^* D'$, $\iota_\rho \phi^* = \phi^* \iota_{\rho'}$, and $\check{D}\phi^* = \phi^* \check{D'}$. Therefore,
\begin{align*}
  D  \iota_\rho \omega 
  &= D\phi^* \iota_{\rho'} \omega'
  = \phi^* D' \iota_{\rho'} \omega'
  \\
  D\mu + \iota_\rho \omega 
  &= D\phi^* \mu' + \iota_{\rho} \phi^* \omega' 
  = \phi^*(D'\mu' + \iota_{\rho'} \omega')
  \\
  \check{D} \mu 
  + \tfrac{1}{2} \iota_\rho \iota_\rho \omega 
  &= \check{D} \phi^* \mu' 
    + \tfrac{1}{2} \iota_\rho \iota_\rho \phi^* \omega' 
  = \phi^*( \check{D'} \mu' 
    + \tfrac{1}{2} \iota_{\rho'} \iota_{\rho'} \omega' )
    \,.
\end{align*}
If $D' \iota_{\rho'}\omega'$, $D'\mu' + \iota_{\rho'} \omega'$, and $\check{D'} \mu' + \frac{1}{2}\iota_{\rho'} \iota_{\rho'} \omega'$ on the right hand sides of these equations vanish, then the left hand sides vanish as well. Moreover, when $\phi_1$ is surjective and $\phi_0$ a submersion, then $\phi^*$ is injective, so that in the last statement the ``if'' becomes ``if and only if''.
\end{proof}

\subsection{ Quotient by the isotropy subbundle}
\label{sec:ReduceIsotropy}

As an application of the results of this section, we will study what happens to a (weakly) hamiltonian structure on a Lie algebroid when we take the quotient by its isotropy bundle. We begin with action Lie algebroids.

Let $\rho: \frakg \times M \to TM$ be the anchor of an action Lie algebroid over a presymplectic manifold. The kernel of the corresponding homomorphism $\frakg \to \calX(M)$ is an ideal $\frakh \subseteq \frakg$, so the action descends to an action $\underline{\rho}$ of the quotient Lie algebra $\underline{\frakg} := \frakg/\frakh$.  On the other hand, a momentum map $\mu:M\to \frakg^*$ descends to a map $\underline{\mu}:M\to (\frakg/\frakh)^*$ if and only its image annihilates $\frakh$ under the natural pairing.

\begin{Example}
Let $\frakg$ be the 3-dimensional real Heisenberg Lie algebra spanned by the generators $Q$, $P$, $I$ subject to the relations $[Q,P] = I$, $[Q,I] = 0 = [P,I]$. The action $\rho: \frakg \to \calX(\bbR^2)$ on the symplectic plane $(\bbR^2, \omega = dq \wedge dp)$ given by
\begin{equation*}
  \rho(Q) = - \frac{\partial}{\partial p}
  \,,\quad
  \rho(P) = \frac{\partial}{\partial q}
  \,,\quad
  \rho(I) = 0 \,,
\end{equation*}
is hamiltonian with the momentum map $\mu$ defined by $\langle \mu, Q \rangle = q$, $\langle \mu, P\rangle = p$, and $\langle\mu, I \rangle = 1$.

The kernel of $\rho$ is the 1-dimensional Lie algebra ideal spanned by $I$. The action $\rho$ descends to an action of the quotient $\underline{\frakg} = \frakg/\ker\rho$, which is a 2-dimensional abelian Lie algebra. Note, though, that the momentum map $\mu$ does {\it not} descend to $\underline\frakg$ because it does not annihilate the kernel of $\rho$. On the other hand, there is another momentum map, $\mu'$, obtained from $\mu$ by change of the value on $I$ from $1$ to $0$, which {\it does} descend to a momentum map $\underline{\mu}'$ for the action of $\underline\frakg$. This map, which makes the reduced action weakly hamiltonian, is given by $\langle \underline{\mu}', Q + \ker \rho \rangle = q$ and $\langle \underline{\mu}', P + \ker \rho \rangle = p$. Now neither $\mu'$ nor $\underline{\mu}'$ is equivariant. In fact, the action of $\underline{\frakg}$ is not hamiltonian for any choice of momentum map. This is a special case of Proposition \ref{prop:LieActQuot2} below.
\end{Example}

Our example shows that a momentum map for $\rho$ does generally not descend to a momentum map for $\underline{\rho}$, but we found (at the possible expense of equivariance) a different momentum map which does descend.  This is a general phenomenon, as the next proposition shows.

\begin{Proposition}
\label{prop:LieActQuot1}
The action $\rho$ of a Lie algebra $\frakg$ on a presymplectic manifold is weakly hamiltonian if and only if the induced action of $\frakg/\ker \rho$ is.
\end{Proposition}
\begin{proof}
Let $\rho: \frakg \to \calX(M)$ be a weakly hamiltonian action with momentum map $\mu: M \to \frakg^*$. Denote by $i: \ker \rho \hookrightarrow \frakg$  the embedding of the kernel. Since the hamiltonian vector fields of elements in $\ker \rho$ are zero, the induced map $\iota^*\mu: M \to (\ker \rho)^*$ is constant. Choosing some linear projection $p: \frakg \to \ker\rho$ satisfying $p\,i = \id$, we get a constant map $\nu: p^*i^* \circ \mu: M \to \frakg^*$, which satisfies $\langle \nu, X \rangle = \langle \mu, X \rangle$ for all $X \in \ker \rho$. Since $\nu$ is constant, $\mu' := \mu -  \nu$ is also a momentum map for $\rho$. It satisfies $\langle \mu', \ker\rho \rangle = 0$, so that $\mu'$ descends to a momentum map for the induced action of $\underline{\frakg} := \frakg/\ker\rho$.

Conversely, assume that $\mu$ is a momentum map for the induced action of $\underline{\frakg}$. Let $\pi: \frakg \to \underline{\frakg}$ be the canonical projection. Then $\pi^* \circ \mu: M \to \frakg^*$ is a momentum map for $\rho$.
\end{proof}

For a momentum map of the induced action of the quotient to be equivariant, we need an extra condition:

\begin{Proposition}
\label{prop:LieActQuot2}
Let $\rho: \frakg \to \calX(M)$ be a hamiltonian action on a presymplectic manifold with equivariant momentum map $\mu$. Then the induced action of $\frakg/\ker\rho$ is hamiltonian if and only if $\langle\mu, [\frakg,\frakg] \cap \ker\rho \rangle = 0$.
\end{Proposition}

\begin{proof}
By assumption, the momentum map $\mu$ is equivariant, which means that 
\begin{equation}
\label{eq:ChevEilEqui1}
  \delta\mu(X,Y) = - \omega(\rho(X), \rho(Y))
  \,,  
\end{equation}
for all $X,Y \in \frakg$, where $\omega$ is the presymplectic form and where
\begin{equation*}
  \delta \mu(X, Y) 
  = \rho(X) \cdot \langle\mu, Y\rangle 
  - \rho(Y) \cdot \langle\mu, X\rangle 
  - \langle \mu, [X,Y] \rangle \,,
\end{equation*}
is the Chevalley-Eilenberg differential. As we have seen in the proof of Proposition \ref{prop:LieActQuot1}, another map $\mu': M \to \frakg^*$ is a momentum map if and only if the difference $\nu = \mu - \mu'$ is constant. (As stated in Section \ref{sec:conventions}, $M$ is assumed to be connected.) Moreover, $\mu'$ is equivariant if and only if 
\begin{equation*}
  0 
  = \delta\nu(X,Y)
  = - \langle \nu, [X,Y] \rangle
\end{equation*}
for all $X, Y \in \frakg$.

Assume that $\langle\mu, [\frakg,\frakg] \cap \ker\rho \rangle = 0$. Then we can find a constant map $\nu: M \to \frakg^*$, such that $\langle \nu, X\rangle = \langle \mu, X \rangle$ for all $X \in \ker \rho$ and $\langle\nu, [X,Y] \rangle = 0$ for all $X,Y \in \frakg$. The first condition ensures that $\mu' = \pi^*\underline{\mu}$ for a unique map $\underline{\mu}: M \to \underline{\frakg} = \frakg/\ker\rho$, where $\pi: \frakg \to \underline{\frakg}$ is the quotient map. As in the proof of Proposition \ref{prop:LieActQuot1}, it follows that $\underline{\mu}$ is a momentum map.

The second condition $\langle\nu, [X,Y] \rangle = 0$ ensures that $\mu'$ is equivariant. Let $\underline{\delta}$ denote the Chevalley-Eilenberg differential of $\underline{\frakg}$ and let $\underline{\rho}: \underline{\frakg} \to \calX(M)$ denote the induced action. As is the case for any Lie algebra homomorphism, $\pi^*$ commutes with the differentials, $\delta \circ \pi^* = \pi^* \circ \underline{\delta}$. It follows that \begin{equation}
\label{eq:ChevEilEqui2}
\begin{split}
  &\underline{\delta}\underline{\mu}\bigl( \pi(X), \pi(Y) \bigr) 
  +  \omega\bigl( \underline{\rho}(\pi(X)),
    \underline{\rho}(\pi(Y)) \bigr)
  \\
  ={} &\pi^*(\underline{\delta}\underline{\mu})(X,Y)
  +  \omega\bigl( 
  \pi^*\!\underline{\rho}(X),
  \pi^*\!\underline{\rho}(Y) \bigr)
  \\
  ={} & 
  \delta(\pi^*\underline{\mu})(X,Y) 
  +  \omega\bigl( \rho(X), \rho(Y) \bigr)
\end{split}
\end{equation}
for all $X, Y \in \frakg$. The right hand side vanishes because $\mu' = \pi^*\underline{\mu}$ is equivariant. Since $\pi$ is surjective, the left hand side vanishes if and only if $\underline{\mu}$ satisfies the equivariance condition~\eqref{eq:ChevEilEqui1}. We conclude that $\underline{\rho}$ is a hamiltonian action.

Conversely, assume that $\underline{\mu}$ is an equivariant momentum map for the induced action $\underline{\rho}$. Let $\mu' := \pi^* \underline{\mu}$. As in the proof of Proposition \ref{prop:LieActQuot1}, it follows that $\mu'$ is a momentum map. From Eq.~\eqref{eq:ChevEilEqui2} it follows that $\mu'$ is equivariant.

The difference $\nu = \mu - \mu'$ of two momentum maps is a constant map. Since both $\mu$ and $\mu'$ are equivariant, we also have $\langle \nu, [X,Y] \rangle = 0$ for all $X, Y \in \frakg$. Since $\langle \mu', X\rangle = \langle \underline{\mu}, \pi(X) \rangle$ vanishes on $\ker\pi = \ker\rho$, it follows that $\langle \mu, X \rangle = \langle\nu, X \rangle$ for all $X \in \ker \rho$. From $\langle \nu, [X,Y] \rangle = 0$ we conclude that $\langle\mu, [\frakg,\frakg] \cap \ker\rho \rangle = 0$.
\end{proof}

In the case of Lie algebroids, the kernel of the action is replaced by the (generally non-regular) isotropy bundle $\ker \rho \subseteq A$. The space of sections of $\ker\rho$ is still a Lie algebra ideal, so that when $\ker\rho$ is regular, the quotient $\underline{A} := A/\ker\rho$ is a Lie algebroid with injective anchor $\underline{\rho}$. Then $\underline{A}$ can be identified with the regular involutive distribution $\underline{\rho}(\underline{A}) = \rho(A) \subseteq TM$.

When do the additional compatibility conditions of $A$ being presymplectically anchored or (weakly) hamiltonian descend to $\underline{A}$? As a first step, we observe that for a presymplectically anchored vector bundle the kernel of the dualized anchor is ``invariant'' under the connection.

\begin{Proposition}
\label{prop--kernel}
If  $(A,\rho)$ is presymplectically anchored with respect to $D$, then, for all sections $a$, $\langle \gamma , a\rangle = 0$ implies that $\langle \gamma , Da\rangle = 0$.
\end{Proposition}

\begin{proof}
Consider $\gamma$ as an $A^*$-valued 1-form on $M$.  In~\eqref{eq:DgammaPaired} we have seen that, for any section $a$ of $A$, the scalar 2-form $\langle D\gamma,a \rangle$ is equal to $d\langle \gamma, a \rangle + \langle \gamma ,Da \rangle$. If the connection satisfies $D\gamma = 0$, then $\langle \gamma, Da \rangle = - d\langle \gamma, a \rangle$, and the proposition follows.
\end{proof}

If $(A,\rho)$ is regular and presymplectically
anchored with respect to $D$, then the kernel of $\gamma=\tilde\omega \circ\rho$ is a vector subbundle of $A$, and Proposition \ref{prop--kernel} shows that $D$ induces a connection on it. For $D$ to descend to the quotient Lie algebroid $\underline{A}$, however, we need $D$ to induce a connection on $\ker\rho \subseteq \ker\gamma$, i.e.~ that $\ker \rho$ is invariant under $D$.  By Proposition \ref{prop--kernel}, this happens, for instance, when $\ker\rho = \ker\gamma$, e.g.~when $\omega$ is symplectic.

\begin{Remark}
Suppose that $\omega=0$, so that any connection on $A$ is presymplectically anchored.  For a Lie algebroid for which $\ker \rho$ is neither $0$ nor all of $A$, most connections $D$ on $A$
do not leave $\ker\rho$ invariant, so they do not induce connections on $\underline{A}$. 
\end{Remark}

\begin{Proposition}
\label{prop--kernel2}
If $(A,\rho)$ is a regular anchored vector bundle over a presymplectic manifold, then for any presymplectically anchored connection $D$ on $A$ which leaves $\ker\rho$ invariant, the induced connection $\uline D$ on $\underline{A} = A/\ker\rho$ is presymplectically anchored.
\end{Proposition}

\begin{proof}
Let $D$ be any presymplectically anchored connection on $A$. To define the induced connection $\underline{D}$ on $\underline{A}$, we set for each section $\uline a$, $\uline D \,\uline a = \pi\circ Da$, where $\pi: A \to \uline A$ is the projection and $a$ is chosen so that $\uline a = \pi \circ a$.  Since $\ker\rho$ is invariant under $D$, the result is independent of the choice of $a$.
To show that $\uline D$ is presymplectically anchored, we must check that, for any section $\uline a$ of $\uline A$ and vector field $v$ on $M$, the ordinary 1-form $\langle \underline{D}_v \underline{\gamma}, \underline{a} \rangle$ is equal to zero. 

Viewing the dualized anchors $\gamma$ and $\underline{\gamma} := \tilde\omega\circ\underline{\rho}$ as 1-forms on $M$ with values in $A^*$ and ${\uline A}^*$ respectively, we note first that $\gamma = \uline \gamma \circ \pi$ since $\rho = \underline{\rho} \circ \pi$. Now, according to the definition of the dual connection $\uline D$ on $\uline A^*$-valued forms, we have
\begin{equation}
\label{difference}
\langle\uline D_v \uline\gamma,\uline a\rangle 
  = \iota_v d \langle\uline\gamma,\uline a \rangle 
  - \langle\uline\gamma,\underline{D}_v\, \underline{a} \rangle \,.
\end{equation}
The pairing of the $A^*$-valued 1-form $\underline{\gamma}$ with the section $\underline{D}_v \underline{a}$ of $\underline{A}$ can be viewed as the composition of the map $\underline{D}_v\underline{a}: M \to \underline{A}$ with the map $\underline{\gamma}: \underline{A} \to T^*M$. In terms of such composition of maps, the second term on the right hand side of Eq.~\eqref{difference} can be rewritten (without the $-$ sign) as
\begin{equation*}
\begin{split}
  \langle\uline\gamma,\underline{D}_v\, \underline{a} \rangle
  &= \underline{\gamma} \circ \underline{D}_v\, \underline{a}
  = \underline{\gamma} \circ \pi \circ D_v a
  = \gamma \circ  D_v a \\
  &= \langle \gamma, D_va \rangle\,,
\end{split}
\end{equation*}
for any choice of $a$ with $\uline a= \pi\circ a$.

Since $D\gamma=0$ by assumption, we have, again by duality,
\begin{equation*}
\begin{split}
  \langle \gamma,D_v a \rangle 
  &= \iota_v d \langle\gamma, a \rangle 
  = \iota_v d (\gamma \circ a ) 
  = \iota_v d (\uline\gamma\circ \pi\circ a) 
  = \iota_v d (\uline\gamma \circ \underline{a}) \\
  &= \iota_v d\langle\uline\gamma,\uline a \rangle\,,
\end{split}
\end{equation*}
which cancels the first term on the right side of Eq.~\eqref{difference}. So $\underline{D}_v \underline{\gamma} = 0$, and hence the connection $\uline D$ is presymplectically anchored.
\end{proof}

\begin{Proposition}
\label{prop--kernel3}
Let $(A,\rho)$ be a regular anchored vector bundle such that $\ker\rho = \ker\gamma$, let $D$ be a presymplectically anchored connection on $A$, and let $(\underline{A}, \underline{\rho})$ be the anchored  quotient vector bundle with the presymplectically anchored connection $\underline{D}$ from Proposition \ref{prop--kernel2}. Then $\underline{\mu} \in \Gamma(M, \underline{A}^*)$ is a $\underline{D}$-momentum section for $(\underline{A}, \underline{\rho})$ if and only if its pullback $\mu \in \Gamma(M, A^*)$ is a $D$-momentum section for $(A,\rho)$.
\end{Proposition}

\begin{proof}
Viewing the section $\underline{\mu}$ as a map $A \to \bbR$, we may write  the pullback section as $\mu = \underline{\mu} \circ \pi$, where $\pi: A \to \underline{A}$ is the projection. As in the proof of \ref{prop--kernel2}, we have, for any section $\underline{a}$ and choice of section $a$ of $A$ with $\underline{a} = \pi \circ a,$ 
\begin{equation*}
\begin{split}
  \langle\underline{D}\,\underline{\mu}, \underline{a} \rangle
  &= d \langle \underline{\mu}, \underline{a} \rangle
  + \langle\underline{\mu}, \underline{D}\,\underline{a} \rangle
  \\
  &= d (\underline{\mu} \circ \underline{a})
  + \underline{\mu} \circ \underline{D}\,\underline{a}
  \\
  &= d (\underline{\mu} \circ \pi \circ a)
  + \underline{\mu} \circ \pi \circ Da
  \\
  &= d (\mu \circ a)
  + \mu \circ Da
  \\
  &= d \langle \mu, a \rangle
  + \langle\mu, Da \rangle
  \\
  &= \langle D \mu, a \rangle \,,
\end{split}
\end{equation*}
where we have used Eq.~\eqref{eq:DgammaPaired} twice. Using this equation, we conclude that $\underline{D}\underline{\mu} = \underline{\gamma}$ implies that $\langle\gamma, a \rangle = \langle\underline{\gamma}, \underline{a}\rangle = \langle\underline{D}\,\underline{\mu}, \underline{a} \rangle = \langle D\mu, a\rangle$ for all sections $a$, so that $D\mu  = \gamma$. By an analogous argument we see that $D\mu = \gamma$ implies that $\underline{D}\,\underline{\mu} = \underline{\gamma}$.
\end{proof}

It does \emph{not} follow from Proposition \ref{prop--kernel2} that $(A,\rho)$ being weakly hamiltonian implies that $(\underline{A}, \underline{\rho})$ is weakly hamiltonian as well. For it is not clear whether a momentum section $\mu$ can be chosen to annihilate $\ker \rho$, which is the necessary and sufficient condition for $\mu$ to descend to a section of $\underline{A}^*$. The following lemma gives a partial answer to this question:

\begin{Lemma}
\label{lem--kernel3b}
Let $(A,\rho)$ be a regular presymplectically anchored vector bundle for which $\ker\rho = \ker\gamma$. If the subbundle $\ker\rho$ has a complement that is invariant under parallel transport, then a $D$-momentum section, if it exists, can always be chosen to annihilate $\ker\rho$.
\end{Lemma}
\begin{proof}
Let $A_2 \subseteq A$ be a $D$-invariant complement of $A_1 = \ker\rho$, $A = A_1 \oplus A_2$. There is a corresponding splitting of the dual bundle $A^* = A_1^* \oplus A_2^* $. By Proposition \ref{prop--kernel} $\ker\rho$ is $D$-invariant, so that the connection splits as $D =  D_1 + D_2$ into components acting on $A_1$ and $A_2$. Let $\mu$ be a $D$-momentum section, which splits as $\mu = \mu_1 + \mu_2$. For $a \in \ker\rho$, the condition that $\mu$ is a momentum section implies $\langle D\mu, a\rangle = \langle D_2 \mu_2, a \rangle = \gamma(a) = 0$, that is, $\mu_1$ is a $D_1$-flat section of $A_1^*$. Since $D_2 \mu_1 =  0$, it follows that $\mu_1$ is a $D$-flat section of $A^*$, so that $\mu_2$ is a $D$-momentum section with values in $A_2^* = (\ker\rho)^\circ$, the annihilator in $A^*$ of $\ker\rho$.
\end{proof}

\begin{Remark}
\label{rmk--kernel3c}
Since $M$ is connected, $\ker\rho$ has a $D$-invariant complement if each fibre has a complement as a submodule of the representation of the holonomy group. For an action Lie algebroid with the trivial connection, the holonomy is trivial, so that such a complement always exists. But in general the holonomy representation may not be reducible, e.g.~when the holonomy group is not compact.
\end{Remark}

We now suppose that $A$ is a Lie algebroid rather than simply an anchored vector bundle.

\begin{Proposition}
\label{prop--kernel4}
Let $(A, \rho, [~,~])$ be a regular presymplectically anchored Lie algebroid such that $\ker\rho = \ker\gamma$. Then a momentum section for the  quotient Lie algebroid $\underline{A} = A/\ker \rho$ is bracket-compatible if and only if the pullback momentum section for $A$ from Proposition \ref{prop--kernel3} is bracket-compatible.
\end{Proposition}

\begin{proof}
Since $\rho$ is a regular homomorphism of Lie algebroids, the quotient $\underline{A} = A/\ker\rho$ is a Lie algebroid and the projection $\pi: A \to \underline{A}$ a homomorphism of Lie algebroids. Assume that $\underline{\mu} \in \Gamma(M,\underline{A}^*)$ is a momentum section for $(\underline{A}, \underline{\rho})$. It was shown in Proposition \ref{prop--kernel3} that the pullback $\mu = \pi^* \underline{\mu} \in \Gamma(M,A^*)$ is a momentum section for $(A,\rho)$.

As is the case for any homomorphism of Lie algebroids, the induced pullback $\pi^*: \Omega(\underline{A}) \to \Omega(A)$ of exterior forms on $\underline{A}$ is a map of differential complexes, $\pi^* \DA_{\underline{A}} = \DA_{A} \pi^*$. Applying this to $\underline{\mu}$, we obtain $\pi^* \DA_{\underline{A}} \underline{\mu} = \DA_{A} \pi^* \underline{\mu} = \DA_{A} \mu$. It follows that $(\DA_A \mu)(a,b) = -\omega(\rho a, \rho b)$ if and only if
\begin{equation*}
\begin{split}
  (\DA_{\underline{A}} \underline{\mu})(\underline{a}, \underline{b})
  &= (\DA_{\underline{A}} \underline{\mu})(\pi \circ a, \pi \circ b)
  = (\DA_{A} \mu)(a, b)
  = -\omega(\rho a, \rho b) \\
  &= -\omega(\underline{\rho}\underline{a},
    \underline{\rho}\underline{b}) \,,
\end{split}
\end{equation*}
where $a$, $b$ are sections of $A$ for which $\pi\circ a = \underline{a}$ and $\pi\circ b = \underline{b}$.
\end{proof}

Putting together the results of this section, we obtain the Lie algebroid analog of Propositions \ref{prop:LieActQuot1} and~\ref{prop:LieActQuot2}:

\begin{Theorem}
Let $A$ be a regular Lie algebroid over the presymplectic manifold $(M,\omega)$ satisfying $\ker\rho = \ker\gamma$. Assume that $A$ is presymplectically anchored with respect to the connection $D$ and let $\underline{D}$ denoted the induced connection on $\underline{A} = A / \ker\rho$. Assume furthermore that the subbundle $\ker\rho \subset A$ has a $D$-invariant complement. Then
\begin{itemize}

\item[(i)] $\underline{A}$ is weakly hamiltonian with respect to $\underline{D}$ if and only $A$ is weakly hamiltonian with respect to $D$.

\item[(ii)] Assume that $A$ is hamiltonian with $D$-momentum section $\mu$. Then $\underline{A}$ is hamiltonian with respect to $\underline{D}$ if and only if $\langle\mu, T(A,A) \cap \ker \rho \rangle = 0$.

\end{itemize}
\end{Theorem}

\begin{proof}
(i) By Proposition \ref{prop--kernel2} $\underline{A}$ is presymplectically anchored with respect to $\underline{D}$. Assume that $\underline{A}$ is weakly hamiltonian with respect to $\underline{D}$. It follows from Proposition \ref{prop--kernel3} that $A$ is weakly hamiltonian with respect to $D$. Conversely, assume that $A$ is weakly hamiltonian. By lemma \ref{lem--kernel3b} the momentum section $\mu$ can be chosen to annihilate $\ker\rho$, so that it is the pullback of a section $\underline{\mu}$ of $\underline{A}^*$. It follows from Proposition \ref{prop--kernel3} that $\underline{\mu}$ is a momentum section with respect to $\underline{D}$.

(ii) Let $\mu$ be an equivariant $D$-momentum section for $A$. Assume that $\underline{A}$ is hamiltonian with bracket-compatible $\underline{D}$-momentum section $\underline{\mu}$. It follows from Proposition \ref{prop--kernel3} that the pullback $\mu' = \pi^*\underline{\mu}$ along the canonical epimorphism $\pi: A \to \underline{A}$ is a $D$-momentum section and from Proposition \ref{prop--kernel4} that it is bracket-compatible. Let $\nu := \mu - \mu'$. Since both, $\mu$ and $\mu'$ are momentum sections, $D\nu = D\mu - D\mu' = \gamma - \gamma = 0$. It follows that
\begin{equation}
\label{eq:nuannihilT}
\begin{split}
  (\DA \nu)(a,b) 
  &= \rho a \cdot \langle\nu, b\rangle 
   - \rho b \cdot \langle\nu, a\rangle
   - \langle\nu, [a,b] \rangle
  \\ 
  &= \langle\nu, D_{\rho a} b\rangle 
   - \langle\nu, D_{\rho b} a\rangle
   - \langle\nu, [a,b] \rangle
  \\
  &= \langle\nu, T(a,b) \rangle
\end{split}
\end{equation}
for all sections $a$ and $b$ of $A$. Since both $\mu$ and $\mu'$ are bracket-compatible, $\DA \nu = 0$ so that $\nu$ vanishes on the image $T(A,A)$ of the torsion. Since $\mu'$ vanishes on $\ker\rho$ it follows that $\langle\mu, T(A,A) \cap \ker\rho \rangle = \langle\nu, T(A,A) \cap \ker\rho \rangle = 0$.

Conversely, assume that $\langle\mu, T(A,A) \cap \ker\rho \rangle = 0$. By assumption $\ker\rho$ has a $D$-invariant complement $C$, $A = \ker\rho \oplus C$, so that the momentum section has two components $\mu := \nu + \mu'$, where $\nu$ is a section of $(\ker\rho)^*$ and $\mu$ a section of $C^*$. Since $\gamma$ vanishes on $\ker\rho$, the condition for $\mu$ to be a momentum section, $D\mu = D\nu + D\mu' = \gamma$ is equivalent to the two conditions $D\nu = 0$ and $D \mu' = \gamma$. This shows that $\mu' = \mu - \nu$ is a momentum section. Since $\mu'$ annihilates $\ker\rho$ it is the pullback $\mu' = \pi^*\underline{\mu}$ along the canonical epimorphism $\pi: A \to \underline{A}$ of a $\underline{D}$-momentum section for $A$. By Proposition \ref{prop--kernel4} $\underline{\mu}$ is bracket-compatible iff $\mu'$ is. By Eq.~\eqref{eq:nuannihilT} this is the case iff $\nu$ annihilates $T(A,A)$. Splitting the image of the torsion as $T(A,A) = (T(A,A) \cap \ker\rho) \oplus (T(A,A) \cap C)$ using that $\nu$ annihilates $C$ and observing that on $\langle \nu, a\rangle = \langle\mu, \nu\rangle$ for all $a \in \ker\rho$, we obtain $\langle \nu, T(A,A) \rangle = \langle \mu, T(A,A) \cap \ker\rho \rangle$, which vanishes by assumption.
\end{proof}

\subsection{Application to transitive Lie algebroids}

An anchored vector bundle $(A,\rho)$ over $M$ with surjective anchor can be split as the direct sum of the  isotropy bundle and a subbundle mapping isomorphically by $\rho$ to the tangent bundle of $M$. In Section \ref{sec:ReduceIsotropy} we have studied the role of the isotropy bundle and in Section \ref{sec:tangent} we have given conditions for the tangent bundle to be (weakly) hamiltonian. Now we will combine these results and obtain conditions for transitive Lie algebroids to be hamiltonian.

\begin{Proposition}
\label{prop:TransAlgdHam1}
Every transitive Lie algebroid $A$ over a regular presymplectic manifold $M$ can be presymplectically anchored.
\end{Proposition}
\begin{proof}
We choose a splitting of the surjective anchor to identify $A \cong  \ker \rho \oplus TM, $ with the anchor being given by the projection to $TM$. The  quotient bundle of Proposition \ref{prop--kernel2} can then be identified with the tangent bundle $\underline{A} = TM$ and the  quotient anchor with the identity map $\underline{\rho} = \id_{TM}$. The dualized anchor $\gamma = \gamma_1 + \gamma_2$ is given by the zero section $\gamma_1 = 0$ on $\ker\rho$ and $\gamma_2 = \tilde{\omega}$ on $TM$.

If $D_1$ and $D_2$ are arbitrary linear connections on the subbundles $\ker\rho$ and $TM$, respectively, then the connection on the quotient from Proposition \ref{prop--kernel2} is $\underline{D} = D_2$. The condition for $A$ to be presymplectically anchored with respect to $D =  D_1 + D_2$ is  $D\gamma = D_1 \gamma_1 + D_2 \gamma_2 = 0$. The zero section $\gamma_1 = 0$ is horizontal with respect to any linear connection $D_1$, which implies that $A$ is presymplectically anchored if and only if $D_2\gamma_2 = 0$. We conclude that $A$ is presymplectically anchored if $TM$ is, which by Proposition \ref{prop:SympConn} is always the case.
\end{proof}

\begin{Proposition}
\label{prop:TransAlgdHam2}
Let $A$ be a transitive Lie algebroid over a symplectic manifold $M$. If $TM$ is (weakly) hamiltonian then so is $A$.
\end{Proposition}
\begin{proof}
Let $D = D_1 + D_2$ be a direct sum of connections as above. Let $\underline{\mu}$ be a $\underline{D} = D_2$-momentum section for $\underline{A} = TM$. According to Proposition \ref{prop--kernel3}, $\mu := \mu_1 + \mu_2$, with $\mu_1 = 0$ and $\mu_2 = \underline{\mu}$, is a $D$-momentum section for $A$. According to Proposition \ref{prop--kernel4}, $\mu$ is bracket-compatible iff $\underline{\mu}$ is bracket-compatible.
\end{proof}

\begin{Remark}
The proof does \emph{not} show that $A$ being weakly hamiltonian implies that $TM$ is weakly hamiltonian, since we cannot exclude the case that $A$ has only $D$-momentum sections $\mu$ that do not annihilate $\ker\rho$ and, therefore, do not descend to $TM$. According to Remark ~\ref{rmk--kernel3c}, this can  occur only when $\ker\rho$ does not have a $D$-invariant complement. However, we know of no example that realizes this topological obstruction.
\end{Remark}

\begin{Corollary}
Any transitive Lie algebroid over a symplectic manifold that is either non-compact or compact with non-negative Euler characteristic is weakly hamiltonian.
\end{Corollary}
\begin{proof}
The statement follows from Propositions \ref{prop:TransAlgdHam1} and \ref{prop:TransAlgdHam2}, together with Theorem  \ref{thm:TMweaklyhamB}.
\end{proof}

\section{Reduction of action Lie groupoids and algebroids}
\label{sec:ReductionAction}

If $G$ is a Lie group acting smoothly on a manifold $M$, and if $H$ is a closed normal subgroup operating freely and properly on $M$, so that $M/H$ is a manifold, then the reduced action is a smooth action.  We will show in this section that, under the additional assumption that $H$ is acting freely on $M$, we still have a reduction of the action groupoid $G \ltimes M$ to a reduced groupoid over $M/H$, even when $H$ is not normal.  It is generally no longer an action groupoid, though it does agree with the action groupoid of the reduced action in the case where $H$ is normal. We will give examples, show the Morita invariance of the reduced groupoid to the original one, and describe the Lie algebroid of the reduced groupoid.  The reduced Lie algebroid (but not the groupoid) has appeared with slightly different assumptions in \cite{Lu:2008}.

This section is independent of the previous sections of the paper and may be of interest in itself.   In Section~\ref{sec:HamHReduced}, we will return to our main subject and study how hamiltonian Lie algebroid structures behave under reduction.

As throughout this paper, we will be working in the smooth category.  It should be noted, though, that there are similar versions in the topological category and simply in the category of sets (in which case properness of an action is not required for the existence of a nice quotient space).

\subsection{Quotient of an action groupoid by a subgroup}
\label{sec:RedAct1}

Recall that, if a group $H$ acts on $X$ from the right, and on $M$ from the left, then $X \times_H M$ is defined as $(X \times M)/H$, where the (left) action of $H$ is given by $h\cdot (x,m) = (x \cdot h^{-1}, h \cdot m)$.  If the action on $M$ is free and proper, then this is is a smooth fibre bundle over $M/H$ with fibre $X$, called the \textbf{associated $X$ bundle} to the principal $H$ bundle $M\to M/H$. 
 
For our purposes, $X$ will be a manifold $H\backslash G$ of left $H$-cosets, with the usual right $H$-action. The associated bundle is the quotient
\begin{equation*}
  H\backslash G \times_H M = (G \times M)/(H \times H)
\end{equation*}
with respect to the (left) $H \times H$-action given by $(h_1, h_2)\cdot (g,m) = (h_1g h_2^{-1}, h_2 \cdot m)$.

\begin{Remark}
The notation $M/H$ for the orbit space of an $H$-action is commonly used for both left and right actions.  An exception is the notation for the left and right coset spaces $H \backslash G$ and $G/H$, which leads to the unfortunate equation $M/H = H\backslash G$ when $M=G$.   We have grudgingly decided to bear with this inconsistency rather than fixing it by redefining the well-established notation for either cosets or group quotients.
To get the usual notation $K\backslash G/H$ for a double coset space, we consider $H$ acting from the right on the left coset space $K\backslash G$. 
\end{Remark}

\begin{Proposition}
\label{prop:GroupReduce}
Let $G$ be a Lie group acting from the left on a manifold $M$. Let $H \subseteq G$ be a closed Lie subgroup for which the restriction to $H$ of this action is free and proper. Then there is a unique Lie groupoid structure on $\tilde{\Gamma} := H\backslash G \times_H M$ over $\tilde{M} := M/H$ such that the quotient maps
\begin{equation*}
\xymatrix{
G \ltimes M \ar@{->>}[r]^-{\pi_1} \ar@<-3pt>[d]\ar@<+3pt>[d] & 
H \backslash G \times_H M \ar@<-3pt>[d]\ar@<+3pt>[d] 
\makebox[0pt][l]{$ {}= \tilde{\Gamma}$}
\\
M \ar@{->>}[r]^{\pi_0} & M/H
\makebox[0pt][l]{$ {}= \tilde{M}$}
}
\end{equation*}
form a homomorphism of Lie groupoids.
\end{Proposition}

\begin{proof}
Since $H$ is a closed subgroup, $H\backslash G$ is a smooth manifold on which $H$ has an induced action from the right.  Since the action of $H$ on $M$ is free and proper, $M \to M/H$ is a smooth principal bundle, so that the associated $H\backslash G$-bundle $H\backslash G \times_H M$ over $M/H$ is smooth. We will use the notation
\begin{equation*}
  [m] := \pi_0(m)\,,\quad [g,m] := \pi_1(g,m)
\end{equation*}
for elements of this quotient.  By definition of the group actions, we have $[h \cdot m] = [m]$ and $[h_1 g h_2^{-1}, h_2 \cdot m] = [g,m]$.

To be compatible with $\pi_0$, $\pi_1$, and the action groupoid structure, the left and right anchor maps (sometimes called source or target maps, or moment maps, depending upon conventions), the identity bisection, and the inverse on the quotient must be defined as:
\begin{equation}
\label{eq:RedGroupoid1}
\begin{gathered}
  l[g,m] := [g\cdot m]\,,\quad 
  r[g,m] := [m] \,,\\
  1_{[m]} := [1,m] \,,\quad
  [g,m]^{-1} := [g^{-1}, g \cdot m] \,.
\end{gathered}
\end{equation}
for all $g \in G$ and $m \in M$. When $m_1 = g_2 \cdot m_2$, then the multiplication of $[g_1,m_1]$, $[g_2,m_2] \in \tilde{\Gamma}$ must satisfy
\begin{equation*}
  [g_1,m_1]\, [g_2,m_2] = [g_1 g_2,m_2 ] \,.
\end{equation*}
In the general case, when $r[g_1,m_1] = [m_1]$ and $[g_2 \cdot m_2] = l[g_2,m_2]$ are equal, we observe that the freeness of the $H$ action on $M$ implies that there is a unique $\eta(m_1,g_2 \cdot m_2) \in H$ such that $m_1 = \eta(m_1, g_2 \cdot m_2) \cdot(g_2 \cdot m_2)$. The multiplication must then be defined as
\begin{equation}
\label{eq:RedGroupoid2}
  [g_1,m_1]\,[g_2,m_2] 
  := [g_1 \eta(m_1, g_2 \cdot m_2) g_2, m_2] \,.
\end{equation}
We will now show that the structure maps defined in Eqs.~\eqref{eq:RedGroupoid1} and \eqref{eq:RedGroupoid2} satisfy the axioms of a Lie groupoid. 

First, we observe that $\eta(m_1, g_2 \cdot m_2)$ is the image of the the smooth map $\eta: M \times_{\tilde{M}} M \to H$ that maps a pair of elements $m$ and $m'$ in the same fibre of $M \to \tilde{M}$ to the unique element $\eta(m,m') \in H$ satisfying $m = \eta(m,m') \cdot m'$. It is easy to see that $\eta$ satisfies the following relations:
\begin{align}
  \eta(h\cdot m, m') &= h \eta(m,m') 
  \tag{H1} \label{eq:etaleftlinear}\\
  \eta(m, m) &= 1 \,, 
  \tag{H2}\label{eq:etaidentity}
\end{align}
for all $(m,m') \in M \times_{M/H} M$ and $h \in H$. Having a smooth map with these properties is equivalent to the $H$-action on $M$ being free (cf.~Prop.~3.4 in \cite{Blohmann:2008}). Relations (H1) and (H2) imply
\begin{equation}
  \eta(m,h\cdot m') = \eta(m,m') h^{-1} 
  \tag{H3}\label{eq:etarightlinear}\,,
\end{equation}
which we will also need.

With these relations we can show that the multiplication is well-defined in the first factor:
\begin{equation*}
\begin{split}
  [h_1 g_1 h_2^{-1}, h_2 \cdot m_1]\, [g_2,m_2]
  &= \bigl[h_1 g_1 h_2^{-1} 
  \eta(h_2 \cdot m_1, g_2 \cdot m_2) g_2, m_2 \bigr] \\
  &= \bigl[h_1 g_1 h_2^{-1} h_2 
  \eta(m_1, g_1 \cdot m_2) g_2, m_2 \bigr] \\
  &= \bigl[h_1 g_1 \eta(m_1, g_2 \cdot m_2) g_2, m_2 \bigr] \\
  &= \bigl[g_1 \eta(m_1, g_2 \cdot m_2) g_2, m_2 \bigr] \,,
\end{split}
\end{equation*}
where in the second step we have used~\eqref{eq:etaleftlinear}. For the second factor we get
\begin{equation*}
\begin{split}
  [g_1, m_1]\, [h_1 g_2 h_2^{-1}, h_2 \cdot m_2]
  &= \bigl[g_1 \eta(m_1, h_1 g_2 h_2^{-1}h_2 \cdot m_2) 
  h_1 g_2 h_2^{-1}, h_2 \cdot m_2 \bigr] \\
  &= \bigl[g_1 \eta(m_1, g_2 \cdot m_2) 
  h_1^{-1} h_1 g_2 h_2^{-1}, h_2 \cdot m_2 \bigr] \\
  &= \bigl[g_1 \eta(m_1, g_2 \cdot m_2) 
  g_2 h_2^{-1}, h_2 \cdot m_2 \bigr] \\
  &= \bigl[g_1 \eta(m_1, g_2 \cdot m_2) 
  g_2, m_2 \bigr] \,,
\end{split}
\end{equation*}
where in the second step we have used~\eqref{eq:etarightlinear}. Next, we check associativity,
\begin{equation*}
\begin{split}
  [g_1, m_1]\, \bigl( [g_2, m_2]\,[g_3, m_3]  \bigr)
  &= [g_1, m_1]\, 
    \bigl[g_2\eta(m_2, g_3 \cdot m_3)g_3, m_3 \bigr] \\
  &= \bigl[ g_1 \eta\bigl(m_1, 
    g_2\eta(m_2, g_3 \cdot m_3)g_3 \cdot m_3 \bigr)
    g_2\eta(m_2, g_3 \cdot m_3)g_3, m_3 \bigr] \\
  &= \bigl[g_1 \eta(m_1, g_2 \cdot m_2)g_2
    \eta(m_2, g_3 \cdot m_3)g_3, m_3\bigr] \\ 
  &= \bigl[g_1 \eta(m_1, g_2 \cdot m_2)g_2, m_2\bigr] 
    \, [g_3, m_3] \\ 
  &= \bigl( [g_1, m_1]\, [g_2, m_2] \bigr)\, [g_3, m_3] \,,
\end{split}
\end{equation*}
where we have used that by definition $\eta(m_2, g_3 \cdot m_3)g_3 \cdot m_3 = m_2$.

It is a straightforward calculation to check that the left and right anchor maps, the identity, and the inverse are also well-defined and satisfy the axioms of a groupoid. For example, for the inverse we have
\begin{equation*}
\begin{split}
  [h_1 g h_2^{-1}, h_2 \cdot m]^{-1}
  &= \bigl[h_2 g^{-1} h_1^{-1}, 
     h_1 g h_2^{-1} \cdot (h_2 \cdot m) \bigr] \\
  &= \bigl[h_2 g^{-1} h_1^{-1}, 
     h_1 \cdot (g \cdot m) \bigr] \\
  &= [g^{-1}, g \cdot m ] \\
  &= [g,m]^{-1} \,.
\end{split}
\end{equation*}
Moreover, $[g,m]^{-1} [g,m] = [g^{-1}, g \cdot m]\,[g, m] = [g^{-1}g, m] = [1,m] = 1_{[m]}$. Verifying the remaining relations of a groupoid is equally straightforward.

The smoothness of the structure maps is a consequence of the smoothness of the quotient maps $\pi_0$ and $\pi_1$, of the map $\eta$, and of the structure maps of $G$. Since $\pi_0$ is a submersion it follows that the anchor maps are submersions.
\end{proof}

\begin{Remark}
\label{rmk:RightIsAd}
$Hgh = Hh^{-1}g h = H(\Ad_h^{-1}g)$, so that the right regular action and the right adjoint action of $H$ on $G$ induce the same right $H$-action on $H \backslash G$.
\end{Remark}

\subsection{Examples}

\begin{Example}
\label{ex:NormalReduce}
In the situation of Proposition \ref{prop:GroupReduce} assume that the closed subgroup $H \subset G$ is normal. Then the right action of $H$ on $H \backslash G$ is trivial, so that the associated $H \backslash G$ bundle is also trivial. More precisely, the map
\begin{equation*}
\begin{aligned}
  \phi: H \backslash G \times_H M &\longrightarrow  
  H \backslash G \times M/H
  \\
  [g,m] &\longmapsto (Hg, [m])
\end{aligned}
\end{equation*}
is an isomorphism. The left and right anchor maps of the reduced groupoid $\tilde{\Gamma}$ are given in terms of the trivial bundle by $l(Hg,[m]) = [g \cdot m]$ and $r(Hg, [m]) = [m]$, the groupoid multiplication by $(Hg_1,[g_2 \cdot m])(Hg_2, [m]) = (Hg_1 g_2, [m])$. We conclude that $\tilde{\Gamma}$ is the action groupoid of the quotient group $H \backslash G$ acting on the quotient space $M/H$.
\end{Example}

\begin{Example}
\label{ex:FreeTrans}
Let $M=G$ with the usual action of $G$ by left translation, for which the action groupoid is the pair groupoid $G\times G$.
Elements of our reduced groupoid 
 $$H\backslash G \times_H G = (G \times M)/(H \times H)$$
 are then equivalence classes of pairs $(g_1,g_2)$ of elements of $G$ with respect to the (left) $H \times H$-action given by $$(h_1, h_2)\cdot (g_1,g_2) = (h_1 g_1 h_2^{-1}, h_2 g_2).$$
From the definition of the left and right anchor maps, we have $l([g_1,g_2]) = Hg_1g_2$ and $r([g_1,g_2])=Hg_2$ in $H\backslash G$.  Thus, any pair $(Hg_1,Hg_2)$ of objects for the reduced groupoid is in the image of the morphism $[g_1 g_2^{-1},g_2]$, and this groupoid is therefore transitive. A simple computation then shows that this morphism is uniquely determined by its image, so that the reduced groupoid is isomorphic to the pair groupoid $H\backslash G \times H\backslash G$.
\end{Example}

\begin{Example}
Assume that the action of $G$ on $M$ is free but not transitive. Then we can apply the argument of the last example to each $G$-orbit and assemble the results.  The conclusion is that the reduced groupoid is isomorphic to the ``relative pair groupoid'' of the fibration $M/H \to M/G$, i.e.\ the fibre product $M/H \times_{M/G} M/H$, which is a Lie subgroupoid of the pair groupoid $M/H \times M/H$.
\end{Example}

\begin{Example}
\label{ex:TransReduce}
Assume that the $G$ action is transitive but not necessarily free. Then $M \cong G/K$ for some subgroup $K$ of $G$, and freeness of the $H$ action means that $H$ intersects each subgroup conjugate to $K$ only in the identity element. The argument in the free and transitive case (Example~\ref{ex:FreeTrans}) still shows that the reduced groupoid is transitive, so it must be the gauge groupoid of a principal bundle $B$ over $M/H$.

As for any gauge groupoid, the principal bundle can be recovered up to isomorphism as the fibre $B = r^{-1}(\tilde{m})$ over some point $\tilde{m} \in M/H$, with the left anchor map as bundle projection and the isotropy group of $\tilde{m}$ as gauge group acting by right groupoid multiplication. Here, $M/H$ is isomorphic to the double coset space $H\backslash G/K$. An isomorphism can be obtained explicitly by mapping the double coset $HgK$ to the $H$-orbit $H\cdot (g \cdot m)$ for some point $m \in M$. The total space $B$ of the principal bundle, defined as the right fibre over $HeK$, is isomorphic to $H \backslash G$, the isomorphism mapping $Hg$ to $[g,m]$. The bundle projection maps $Hg$ to $HgK$, so the gauge group is $K$.

We conclude that, for a transitive $G$-action on $M$, the $H$-reduced groupoid is isomorphic to the gauge groupoid of the principal $K$-bundle $H\backslash G \to H\backslash G/K$, where $K$ is the stabilizer group of some point $m \in M$.  (A different choice of $m$ has a conjugate stabilizer subgroup, so the gauge groupoids are isomorphic.) In the special case where $K$ reduces to the identity, so that $M \cong G$, we have the gauge groupoid of a principal $\{e\}$ bundle over $H\backslash G$, which is just the pair groupoid of $H\backslash G$, so we are back to our initial example.
\end{Example}

For the general case, we can again apply the argument for transitive Lie groupoids ($G$-) orbit-by-orbit. The assumption that the $H$ action is free immediately implies the condition that $H$ has zero intersection with all the isotropy groups, so there are clearly many examples of this.

\subsection{Morita equivalence}

In many applications, it is useful to view an action groupoid as a presentation of the quotient stack of the group action, which is a well-behaved structure even if the quotient is singular. For example, the equivariant cohomology of a manifold with a group action is the same as the stack cohomology of the quotient stack. This point of view leads to the cohomological interpretation of hamiltonian Lie algebroids of Section \ref{sec:CohomInterp}.

A differentiable stack can be viewed as Lie groupoid up to a generalized notion of isomorphism, called geometric Morita equivalence \cite{MoerdijkMrcun:Foliations}. The precise statement is that there is an equivalence between the bicategory of differentiable stacks and the bicategory that has Lie groupoids as objects, right principal bibundles as 1-morphisms, and biequivariant maps as 2-morphisms \cite{Blohmann:2008}. (The 2-morphisms will not play any role in this paper.) Let $\Gamma \rightrightarrows M$ and $\Gamma' \rightrightarrows M'$ be Lie groupoids. A \textbf{$\Gamma$-$\Gamma'$-bibundle} is a smooth manifold $B$ together with two smooth maps $l_B: B \rightarrow M$, $r_B: B \rightarrow M'$ together with a left $\Gamma$-action and a right $\Gamma'$-action that commute. A bibundle is called \textbf{right principal} if $l_B$ is a surjective submersion and if the right $\Gamma'$-action is free and transitive on the $l_B$-fibres. If in addition the left action is free and transitive on the $r_B$-fibres, then the bibundle is called a \textbf{Morita equivalence}.

For any morphism $\pi: \Gamma \to \Gamma'$ of Lie groupoids there is the associated groupoid bibundle given by
\begin{equation*}
  B := M \times_{M'}^{\pi_0, l} \Gamma'
\end{equation*}
with left bundle map $l_B(m, \gamma') = m$, right bundle map $r_B(m, \gamma') = r(\gamma')$, left action $\gamma \cdot (m,\gamma') = \bigl( l(\gamma), \pi_1(\gamma)\gamma' \bigr)$, and right action $(m,\gamma'_1) \cdot \gamma'_2 = (m,\gamma'_1 \gamma'_2)$. The right action is always principal.

The following result shows that the bibundle associated to the morphism of Proposition \ref{prop:GroupReduce} from an action groupoid to its reduced groupoid is a Morita equivalence, i.e.~an isomorphism of differentiable stacks.

\begin{Proposition}
\label{prop:MoritaReduce}
The morphism of Lie groupoids of Proposition \ref{prop:GroupReduce} is a Morita equivalence. 
\end{Proposition}
\begin{proof}
The bibundle associated to the morphism of Proposition \ref{prop:GroupReduce} is 
\begin{equation*}
 B = \bigl\{(m, [g', m']) \in M \times \tilde{\Gamma} 
 ~|~ m = hg' \cdot m' \text{ for some } h \in H\bigr\}
\end{equation*}
with left bundle map $l_B(m, [g', m']) = m$, right bundle map $r_B(m, [g', m']) = [m']$, and left action of $\Gamma = G \times M$ given by
\begin{equation*}
 (g,m) \cdot (m, [g', m']) = \bigl(g \cdot m, [g,m][g',m'] \bigr) \,.
\end{equation*}

Let $(m_1, [g_1', m_1'])$ and $(m_2, [g_2', m_2'])$ be two elements in the $r_B$-fibre over $\tilde{m} \in \tilde{M} = M/H$. This means that there are $h_1, h_2, h \in H$ satisfying $m_1 =  h_1 g_1' \cdot m_1'$, $m_2 =  h_2 g_2' \cdot m_2'$, and $h \cdot m'_1 = m'_2$. Now we check that
\begin{equation*}
\begin{split}
  &\bigl(h_2 g_2' h g_1'{}^{-1}h_1^{-1}, m_1) 
    \cdot (m_1, [g_1', m_1'] \bigr) \\
  ={} 
  &\bigl( h_2 g_2' h g_1'{}^{-1} h_1^{-1} \cdot m_1,
  [h_2 g_2' h g_1'{}^{-1}h_1^{-1}, m_1] [g_1', m_1'] \bigr) \\
  ={}
  &\bigl( h_2 g_2' h \cdot m'_1,
  [h_2 g_2' h g_1'{}^{-1}h_1^{-1}, h_1 g_1' \cdot m_1'] 
  [g_1', m_1'] \bigr) \\
  ={}
  &\bigl( h_2 g_2' \cdot m'_2,
  [h_2 g_2' h g_1'{}^{-1}, g_1' \cdot m_1']
  [g_1', m_1'] \bigr) \\ 
  ={}
  &\bigl( m_2, [h_2 g_2' h, m_1'] \bigr) \\ 
  ={}
  &\bigl( m_2, [h_2 g_2', h \cdot m_1'] \bigr) \\ 
  ={}
  &\bigl( m_2, [g_2', m_2'] \bigr) \,,
\end{split}
\end{equation*}
which shows that the left $\tilde{\Gamma}$-action is transitive on the $r_B$-fibres.

Assume that $(g,m) \cdot (m, [g',m']) = (m, [g',m'])$, where $m = hg' \cdot m'$ for some $h \in H$. This is equivalent to the two conditions
\begin{equation*}
  g \cdot m = m 
\quad\text{and}\quad
  [g,m][g',m'] = [g',m'] \,.
\end{equation*}
The second condition implies that $[g,m] = 1_{l([g',m'])} = [e, g' \cdot m']$, so that $g$ must lie in the subgroup $H$. Since the action of $H$ is free, the first condition $g \cdot m = m$ now implies that $g \in e$. We conclude that $(g,m) = (e, m) = 1_m$, that is, the left $\Gamma$-action is free.
\end{proof}

\subsection{The quotient groupoid in a local trivialization}

It is instructive, and will be helpful for the computation of the Lie algebroid bracket, to describe the quotient groupoid in terms of a local trivialization of the $H$-principal bundle $M \to \tilde{M} := M/H$. Every such trivialization over an open subset $\tilde{U} \subseteq \tilde{M}$ is obtained by choosing a  section $\sigma: \tilde{U} \to M$, $\pi_0 \sigma = \id_{\tilde{U}}$. The trivialization induced by $\sigma$ is given by $H \times \tilde{U} \to M|_{\tilde{U}}$, $(h,\tilde{m}) \mapsto  h \cdot \sigma(\tilde{m})$. The inverse is $M|_{\tilde{U}} \to H \times \tilde{U}$, $m \to \bigl( \eta(m, \sigma\pi_0 (m)), \sigma\pi_0 (m) \bigr)$.

A local trivialization of $M \to \tilde{M}$ induces a local trivialization of any associated bundle. For the bundle $r: \tilde{\Gamma} \to \tilde{M}$ the induced trivialization is given by
\begin{equation}
\label{eq:GrpdTriv}
\begin{aligned}
  H\backslash G \times \tilde{U} 
  &\stackrel{\Phi}{\longrightarrow} 
  \tilde{\Gamma}\bigr|_{\tilde{U}}
  \\
  (Hg,\tilde{m}) 
  &\longmapsto [g,\sigma(\tilde{m})] \,,
\end{aligned}
\qquad
\begin{aligned}
  \tilde{\Gamma}\bigr|_{\tilde{U}}
  &\stackrel{\Phi^{-1}}{\longrightarrow} 
  H\backslash G \times \tilde{U}
  \\
  [g,m] &\longmapsto
  \bigl( H \,
  \Ad^{-1}_{\eta(m, \sigma\pi_0(m))}g, \pi_0(m) \bigr) \,,
\end{aligned}
\end{equation}

Let us assume for simplicity that $\sigma$ is a global section, $\tilde{U} = \tilde{M}$. Then the left and right anchor maps, as well as the identity bisection of $\tilde{\Gamma}$ are given in the trivialization by
\begin{equation*}
  l(Hg, \tilde{m}) = \pi_0\bigl( g \cdot \sigma(\tilde{m}) \bigr)
  \,,\quad
  r(Hg, \tilde{m}) = \tilde{m} 
  \,,\quad
  1_{\tilde{m}} = (He, \tilde{m})
  \,.
\end{equation*}
The groupoid multiplication of $(Hg_1, \tilde{m}_1)$ and $(Hg_2, \tilde{m}_2)$ satisfying $\tilde{m_1} = \pi_0\bigl(g_2 \cdot \sigma(\tilde{m}) \bigr)$ is given by
\begin{equation*}
\begin{split}
  (Hg_1, \tilde{m}_1)(Hg_2, \tilde{m}_2)
  &= \Phi^{-1}\bigl( \Phi(Hg_1, \tilde{m}_1)\,
  \Phi(Hg_2, \tilde{m}_2) \bigr)
  \\
  &= \Phi^{-1}\bigl( [g_1,  \sigma(\tilde{m}_1)] \,
  [g_2, \sigma(\tilde{m}_2)] \bigr)
  \\
  &= \Phi^{-1}\bigl(
  [g_1 \eta(\sigma(\tilde{m}_1), 
  g_2 \cdot \sigma(\tilde{m}_2) ) g_2, \sigma(\tilde{m}_2) 
  ] \bigr)
  \\
  &= \bigl( H(\Ad_{\zeta(g_2, \tilde{m}_2)}
  g_1)g_2, \tilde{m}_2 \bigr) \,,
\end{split}
\end{equation*}
where $\zeta: G \times \tilde{M} \to H $ is defined by
\begin{equation}
\label{zetadefinition}
  \zeta(g,\tilde{m})
  := \eta\bigl( g\cdot \sigma(\tilde{m}),
  \sigma\pi_0( g \cdot   \sigma(\tilde{m}))
  \bigr) \,.
\end{equation}

In geometric terms, $\zeta(g,\tilde{m})$ is the vertical component of the action of $g$ on $\sigma(\tilde{m})$ with respect to the trivialization induced by $\sigma$. If we denote the horizontal component of the action, given by the left anchor map, by
\begin{equation}
\label{eq:HorActDef}
  g \triangleright \tilde{m} 
  := \pi_0\bigl( g \cdot \sigma(\tilde{m}) \bigr) \,,
\end{equation}
we can decompose the action on $\sigma(\tilde{m})$ as
\begin{equation*}
  g \cdot \sigma(\tilde{m})
  = \zeta(g, \tilde{m}) \cdot \sigma( g \triangleright \tilde{m} )
\end{equation*}
into its horizontal and vertical part. The groupoid inverse can now be expressed as
\begin{equation}
\label{eq:ActHorVertDecomp}
  (Hg, \tilde{m})^{-1}
  = \bigl( H\, Ad^{-1}_{\zeta(g,\tilde{m})} g^{-1},
  g \triangleright \tilde{m}\bigr) \,.
\end{equation}

\subsection{Lie algebroid of the quotient groupoid}

By definition, the Lie algebroid of a Lie groupoid is the vector bundle that has left invariant vector fields on the groupoid as sections. Let us review this construction for the action groupoid $\Gamma = G \ltimes M$.

The $l$-fibre over $m \in M$ is given by $l^{-1}(m) = \{(k,k^{-1} \cdot m)~|~ k \in G\}$. The action $L_{(g,m)}: l^{-1}(m) \to l^{-1}(g \cdot m)$ by left multiplication is
\begin{equation*}
  L_{(g,m)} (k, k^{-1}\cdot m) := 
  (g,m)(k, k^{-1}\cdot m) = (gk, k^{-1}\cdot m) \,.
\end{equation*}
Let $TL_{(g,m)}: T l^{-1}(m) \to Tl^{-1}(g \cdot m)$ denote its derivative. At the identity it is given by
\begin{equation*}
  TL_{(g,m)} \bigl( X, \rho(X,m) \bigr)
  = \bigl( TL_g X, \rho(X,m) \bigr) \,,
\end{equation*}
where $X = \frac{d}{dt} k_t |_{t=0}$ is the element of the Lie algebra $\frakg = T_e G$ of $G$ that is represented by the smooth path $k_t \in G$ through $k_0 = e$ and where 
\begin{equation}
\label{eq:rhominus}
  \rho(X, m) := \frac{d}{dt} 
  ( k^{-1}_t \cdot m ) \Bigr|_{t=0}
  = - \frac{d}{dt} ( k_t \cdot m ) \Bigr|_{t=0}
\end{equation}
is the anchor. The minus sign implies that the induced map on sections $\rho: \frakg \to \calX(G)$ is a homomorphism of Lie algebras (cf.~Section \ref{sec:conventions} for our sign conventions). Every $X \in \frakg$ induces a left invariant vector field on $\Gamma$ given by
\begin{equation*}
  v_X(g,m) = \bigl( TL_g X, \rho(X,m) \bigr) \,,
\end{equation*}
the commutator bracket of which is given by $[v_X, v_Y]_{\calX(G)} = v_{[X,Y]_\frakg}$. It follows that $A \cong \frakg \times M$ and that the Lie algebroid bracket of constant sections $X, Y \in \frakg \subseteq \Gamma(M,A)$ is given by $[X,Y]_A = [X,Y]_\frakg$.

We now turn to the quotient groupoid $\tilde{\Gamma}$. The left fibre $l^{-1}(\tilde{m})$ over $\tilde{m} \in \tilde{M}$ is $ \{[g, g^{-1} \cdot m]~|~g \in G\}$ for $m \in \pi_0^{-1}(\tilde{m})$. In order to determine the tangent space of $l^{-1}(\tilde{m})$ at $[e,m]$, we observe that every every tangent vector of $H\backslash G$ is represented by a smooth path in $G$, so that we have the natural isomorphism
\begin{equation*}
  T_{Hg} (H \backslash G) \cong T_g G / TR_g \frakh 
\end{equation*}
(Here and below, $\frakg$ and $\frakh$ are as usual the Lie algebras of $G$ and $H$, respectively.) At $He \in H\backslash G$ we thus obtain the isomorphism of vector spaces
\begin{equation*}
  T_{He} (H \backslash G) \cong \frakg / \frakh \,,
\end{equation*}
In light of remark~\ref{rmk:RightIsAd}, we thus obtain that the Lie algebroid of the reduced Lie groupoid $\tilde{\Gamma} = H\backslash G \times_H M$ has as vector bundle the associated bundle
\begin{equation*}
  \tilde{A} := \frakg/\frakh \times_H M
  \longrightarrow M/H \,,
\end{equation*}
where $\frakg/\frakh$ carries the right adjoint action $(X + \frakh) \cdot h := \Ad_{h^{-1}} X + \frakh$.

\begin{Corollary}
\label{cor:ReducedAlgd}
Assume the situation of Proposition \ref{prop:GroupReduce}. Let $\frakg$ and $\frakh$ be the Lie algebras of $G$ and $H$, respectively. Let $\frakg \ltimes M$ denote the action Lie algebroid. There is a unique Lie algebroid structure on the vector bundle $\frakg/\frakh \times_H M \to M/H$ such that the diagram
\begin{equation*}
\xymatrix@C+4ex{
\frakg \ltimes M \ar@{->>}[r]^-{T\pi_1} \ar[d] & 
\frakg/\frakh \times_H M \ar[d] \\
M \ar@{->>}[r]^{\pi_0} &M/H
}
\end{equation*}
where $T\pi_1$ is shorthand for the restriction of $T\pi_1$ to $\frakg\times M \subseteq T(G \times M)$, is a homomorphism of Lie algebroids.
\end{Corollary}
\begin{proof}
Follows from the functoriality of the map from Lie groupoids to Lie algebroids.
\end{proof}

The anchor of $\tilde{A}$ is given by
\begin{equation*}
  \tilde{\rho}\bigl( [X + \frakh,m] \bigr)
  := T_m\pi_0\, \rho(X,m) \,,
\end{equation*}
where $\rho: \frakg \times M \to TM$ is the Lie algebra action. Since the $H$-action is free the fundamental vector fields of $\frakh$ span the tangent spaces to the $H$-orbits which are the fibres of $\pi_0$, that is, $\rho(\frakg, m) = \ker T_m\pi_0$. This shows that $\tilde{\rho}$ is well defined in $\frakg/\frakh$. Moreover, for every smooth path $k_t \in G$ with $X = \frac{d}{dt} k_t |_{t=0}$, we have $k_t \cdot (h \cdot m) = h \cdot ( h^{-1}k_t k \cdot m )$. By differentiating this relation at $t=0$ and applying $T\pi_0$, we obtain
\begin{equation*}
\begin{split}
  \tilde{\rho}\bigl( [X, h \cdot m ]\bigr)
  &= T\pi_0\, \rho( X, h \cdot m)
  = T\pi_0\, \bigl(h \cdot \rho( \Ad_{h^{-1}}X, m ) \bigr)
  = T\pi_0\, \rho( \Ad_{h^{-1}}X, m ) \\
  &= \tilde{\rho}\bigl([\Ad_{h^{-1}}X, m] \bigr) \,,
\end{split}
\end{equation*}
which shows that $\tilde{\rho}$ is well-defined also with respect to the quotient of the associated bundle.

The Lie  algebroid bracket is more difficult to describe explicitly. Since the vector bundle $\tilde{A} \to \tilde{M}$ is generally not trivial we first have to choose a local trivialization. Every trivialization~\eqref{eq:GrpdTriv} of the Lie groupoid induces a trivialization of the Lie algebroid given by
\begin{equation}
\label{eq:tildeAtriv}
\begin{aligned}
  \frakg/\frakh \times \tilde{U} 
  &\longrightarrow 
  \tilde{A}\bigr|_{\tilde{U}}
  \\
  (X + \frakh,\tilde{m}) 
  &\longmapsto [X + \frakh,\sigma(\tilde{m})] \,,
\end{aligned}
\quad
\begin{aligned}
  \tilde{A}\bigr|_{\tilde{U}}
  &\longrightarrow 
  \frakg/\frakh \times \tilde{U}
  \\
  [X + \frakh,m] &\longmapsto
  \bigl( \Ad^{-1}_{\eta(m, \sigma\pi_0(m))}X + \frakh, 
  \pi_0(m) \bigr) \,.
\end{aligned}
\end{equation}

\begin{Proposition}
\label{prop:redLAtriv}
Assume the situation of Proposition \ref{prop:GroupReduce}. Let $\sigma: \tilde{U} \to M$ be a local section of the $H$-principal bundle $M \to \tilde{M}$ and $\theta: TM \to \frakh$ the connection 1-form of the induced trivialization~\eqref{eq:tildeAtriv}. In the local trivialization the anchor of $\tilde{A}$ is given by
\begin{equation}
\label{eq:RedAnchor}
  \tilde{\rho}(X + \frakh, \tilde{m})
  = T\pi_0 \, \rho\bigl(X,\sigma(\tilde{m}) \bigr) \,,
\end{equation}
and the Lie algebroid bracket by
\begin{equation}
\label{eq:RedBracket}
 [X + \frakh, Y + \frakh](\tilde{m})
 = 
 \bigl[ X + \theta\bigl(\rho(X,\sigma(\tilde{m}))\bigr),
        Y + \theta\bigl(\rho(Y,\sigma(\tilde{m}))\bigr)
 \bigl]_{\frakg}{} + \frakh \,,
\end{equation}
where $X, Y \in \frakg$ represent constant sections of $\tilde{A}$.
\end{Proposition}

\begin{proof}
By definition, the Lie algebroid bracket of $\tilde{A}$ is the Lie bracket of left invariant vector fields on $\tilde{\Gamma}$. In a first step we compute the left action of $\tilde{\Gamma}$ on vectors tangent to the $l$-fibres. We do this by representing a tangent vector by a smooth path $\tilde{\gamma}_t \in \tilde{\Gamma}$ and then taking the derivative with respect to $t$.

In the trivialization~\eqref{eq:tildeAtriv} of $\tilde{\Gamma}$ every smooth path in the left fibre over $\tilde{m}$ is of the form
\begin{equation*}
\begin{split}
  \tilde{\gamma}_t 
  &= (Hk_t^{-1}, \tilde{m})^{-1} 
  = \bigl( H \, Ad^{-1}_{\zeta(k_t^{-1},\tilde{m})} k_t,
    \pi_0(k^{-1}_t \cdot \sigma (\tilde{m}))\bigr)
\\
  &= \bigl( H\bar{k}_t,
     k^{-1}_t \triangleright \tilde{m} \bigr)\,,
\end{split}
\end{equation*}
for some smooth path $k_t \in G$, where we have introduced the abbreviation
\begin{equation*}
  \bar{k}_t 
  := Ad^{-1}_{\zeta(k_t^{-1},\tilde{m})} k_t 
  \,.
\end{equation*}
(Recall that $k^{-1}_t \triangleright \tilde{m} $ denotes the result of 
$k^{-1}_t$ acting on $\tilde{m}.$)

We will need the following relation satisfied by $\zeta$, as defined in
\eqref{zetadefinition}.

\begin{equation*}
\begin{split}
  \zeta(\bar{k}_t, k_t^{-1} \triangleright \tilde{m})
  &= \zeta\bigl(
     Ad^{-1}_{\zeta(k_t^{-1},\tilde{m})} k_t, 
     k_t^{-1} \triangleright \tilde{m} \bigr) 
     \\
  &= \zeta\bigl( 
    \zeta(k_t^{-1},\tilde{m})^{-1} k_t
    \zeta(k_t^{-1},\tilde{m}),
    k_t^{-1} \triangleright \tilde{m} \bigr)
    \\
  &= \zeta(k_t^{-1},\tilde{m})^{-1} 
     \zeta\bigl( 
       k_t \zeta(k_t^{-1},\tilde{m}),
       k_t^{-1} \triangleright \tilde{m} \bigr)
     \\
  &= \zeta(k_t^{-1},\tilde{m})^{-1} \,,
\end{split}
\end{equation*}
where in the last step we have used that $k_t \zeta(k_t^{-1},\tilde{m}) \cdot \sigma(k_t^{-1} \triangleright \tilde{m}) = \tilde{m}$ as a consequence of Eq.~\eqref{eq:ActHorVertDecomp}. Using the last equation, we obtain for the left multiplication of the smooth path by an element of the groupoid,
\begin{equation*}
\begin{split}
   L_{(Hg, \tilde{m})} \tilde{\gamma}_t
   &= (Hg, \tilde{m}) \tilde{\gamma}_t
   = \bigl(H (Ad_{\zeta(\bar{k}_t, 
     k_t^{-1} \triangleright \tilde{m})} g) \bar{k}_t,
     k^{-1}_t \triangleright \tilde{m} \bigr) 
   \\
   &= \bigl(H (Ad^{-1}_{\zeta(k_t^{-1}, \tilde{m})} g) \bar{k}_t,
     k^{-1}_t \triangleright \tilde{m} \bigr)
   \,.
\end{split}
\end{equation*}

The next step is to take the derivative of the last equation with respect to $t$. Assume that $k_0 = e$ and denote $X := \frac{d}{dt} k_t |_{t=0} \in \frakg$. We begin with
\begin{equation*}
\begin{split}
  \frac{d}{dt} (k^{-1}_t \triangleright \tilde{m} )\bigr|_{t=0}
  &= \frac{d}{dt} \pi_0 (k^{-1}_t \cdot \tilde{m} )\bigr|_{t=0}
  = T\pi_0\, \rho\bigl(X,\sigma(\tilde{m}) \bigr) 
  \\
  &=: \tilde{\rho}(X, \tilde{m}) \,,
\end{split}
\end{equation*}
which is the anchor of $\tilde{A}$ in the local trivialization. Next, we compute
\begin{equation*}
\begin{split}
  \frac{d}{dt}\bar{k}_t\bigr|_{t=0} 
  &= \Bigl( 
  \frac{d}{dt} Ad^{-1}_{\zeta(k_t^{-1},\tilde{m})} 
  \Bigr)_{t=0} k_0 +
  Ad^{-1}_{\zeta(k_0^{-1},\tilde{m})} 
  \Bigl( \frac{d}{dt} k_t \Bigr)_{t=0}
  \\
  &= X \,.
\end{split}
\end{equation*}

The next term involves the connection 1-form of the local trivialization which is the $H$-equivariant map $\theta: TM \mapsto \frakh$ defined for vectors at points in $\sigma(\tilde{M})$ by
\begin{equation*}
  -\rho\bigl( \theta(v_{\sigma(\tilde{m})}), \sigma(\tilde{m}) \bigr)
  =  v_{\sigma(\tilde{m})} -  
    (T\sigma\, T\pi_0) v_{\sigma(\tilde{m})} \,,
\end{equation*}
where the minus sign on the left hand side is the minus sign in Eq.~\eqref{eq:rhominus}. When $m_t \in M$ is a smooth path through $m_0 = \sigma(\tilde{m})$ representing $v_{\sigma(m)} = \frac{d}{dt} m_t |_{t=0} \in T_{\sigma(\tilde{m})} M$, we can express $\theta$ as derivative,
\begin{equation*}
  \frac{d}{dt} \eta\bigl(m_t, \sigma\pi_0(m_t) \bigr) \big|_{t=0}
  = \theta(v_{\sigma(\tilde{m})}) \,.
\end{equation*}
With this relation we get
\begin{equation*}
\begin{split}
  \frac{d}{dt} \zeta(k_t^{-1}, \tilde{m}) \bigr|_{t=0}
  &= \frac{d}{dt} \eta\bigl(k_t^{-1} \cdot \sigma(\tilde{m}), 
    \sigma\pi_0(k_t^{-1} \cdot \sigma(\tilde{m})) \bigr) \bigr|_{t=0}
  \\
  &= \theta\bigl( \rho(X,\sigma(\tilde{m})) \bigr) \,.
\end{split}
\end{equation*}
We thus obtain
\begin{equation*}
\begin{split}
  \frac{d}{dt}\, Ad^{-1}_{\zeta(k_t^{-1}, \tilde{m})} g \bigr|_{t=0}
  &= TL_g\, \theta\bigl( \rho(X,\sigma(\tilde{m})) \bigr)
   - TR_g\, \theta\bigl( \rho(X,\sigma(\tilde{m})) \bigr)
  \,,
\end{split}
\end{equation*}
Putting everything together, we can compute the derivative of a path in the $l$-fibre over $\tilde{m}$ and of its left translation by an element of the groupoid,
\begin{equation*}
\begin{aligned}
  \frac{d}{dt} \tilde{\gamma}_t \bigr|_{t=0} 
  &= \bigl( X + \frakh, 
  \tilde{\rho}(X + \frakh,\tilde{m}) \bigr) 
  \\
  \frac{d}{dt}(Hg,\tilde{m})\tilde{\gamma}_t\bigr|_{t=0} 
  &= \bigl( 
  TL_g X + TL_g\, \theta\bigl( \rho(X,\sigma(\tilde{m})) \bigr)
  + TR_g \frakh, \tilde{\rho}(X, \tilde{m}) \bigr) \,.
\end{aligned}
\end{equation*}
We conclude that the left action of the groupoid on vectors tangent to the $l$-fibres is given by
\begin{equation}
\label{eq:LeftTrans}
  TL_{(Hg,\tilde{m})} \bigl(X + \frakh, 
    \tilde{\rho}(X,\tilde{m}) \bigr)
  = \bigl( 
  TL_g X + TL_g\, \theta\bigl( \rho(X,\sigma(\tilde{m})) \bigr)
  + TR_g \frakh, \tilde{\rho}(X, \tilde{m}) \bigr) \,.
\end{equation}

We can now proceed to calculate the bracket of left invariant vector fields. In the trivialization~\eqref{eq:tildeAtriv} a section of the Lie algebroid is given by a map $\tilde{U} \to \frakg/\frakh$, $\tilde{m} \mapsto X(\tilde{m}) + \frakh$, where $X: \tilde{U} \to \frakg$ is a smooth map. For the computation of the Lie algebroid bracket it suffices to consider constant sections $X + \frakh$ and $Y+\frakh$ for $X, Y \in \frakg$. The bracket of non-constant sections can be deduced using the Leibniz rule. The left invariant vector field associated to a constant section $X$ is given by
\begin{equation*}
  v_X(Hg, \tilde{m}) := TL_{(Hg,\tilde{m})} \bigl(X + \frakh, 
    \tilde{\rho}(X,\tilde{m}) \bigr) \,.
\end{equation*}
Using Eq.~\eqref{eq:LeftTrans}, we can compute the bracket of two such vector fields,
\begin{equation*}
\begin{split}
  [v_X, v_Y](Hg, \tilde{m}) 
  = TL_{(Hg,\tilde{m})} \Bigl( &\bigl[
   X + \theta\bigl( \rho(X,\sigma(\tilde{m})) \bigr), 
   Y + \theta\bigl( \rho(Y,\sigma(\tilde{m})) \bigr) \bigr]_{\frakg}
   + \frakh,
   \\ 
  &[\tilde{\rho}(X,\tilde{m}), \tilde{\rho}(Y,\tilde{m})]_{\calX(\tilde{M})} \Bigr)
  \,.
\end{split}
\end{equation*}
This shows that the Lie algebroid bracket is given by Eq.~\eqref{eq:RedBracket}.
\end{proof}

\begin{Remark}
\label{rmk:meaninghor}
When $k_t \in H$ is a smooth path with $X = \frac{d}{dt} k_t |_{t=0} \in \frakh$, then
\begin{equation*}
\begin{split}
  \theta\bigl( \rho(X,\sigma(\tilde{m})) \bigr)
  &= \frac{d}{dt} \eta\bigl(
     k_t^{-1} \cdot \sigma(\tilde{m}), 
     \sigma\pi_0(k_t^{-1} \cdot \sigma(\tilde{m})) \bigr) \bigr|_{t=0} \\
  &= \frac{d}{dt} \zeta(k_t^{-1}, \tilde{m}) \bigr|_{t=0}
  = \frac{d}{dt} k_t^{-1} \bigr|_{t=0}
  \\
  &=  -X \,,
\end{split}
\end{equation*}
where the minus sign comes from the group inverse of the path $k_t^{-1}$. This shows that we have a well-defined map
\begin{equation*}
\begin{aligned}
  \Hor:
  \frakg/\frakh \times \tilde{M} 
  &\longrightarrow \frakg
  \\
  (X+\frakh, \tilde{m}) 
  &\longmapsto 
  X + \theta\bigl( \rho(X,\sigma(\tilde{m})) \bigr) \,,
\end{aligned}
\end{equation*}
in terms of which the bracket~\eqref{eq:RedBracket} can be written as
\begin{equation*}
  [X + \frakh, Y + \frakh](\tilde{m})
  = [\Hor(X+\frakh, \tilde{m}), \Hor(Y+\frakh, \tilde{m})]_\frakg + \frakh \,.
\end{equation*}
It is instructive to apply the anchor of the unreduced groupoid to this map,
\begin{equation*}
\begin{split}
  \rho\bigl(\Hor(X + \frakh, \tilde{m}) \bigr)
  &= \rho(X, \sigma(\tilde{m}) ) 
  + \rho\bigl(\theta\bigr(\rho(X, \sigma(\tilde{m}))\bigr), 
    \sigma(\tilde{m}) \bigr) 
    \\
  &= \rho(X, \sigma(\tilde{m}) ) 
  - \rho(X, \sigma(\tilde{m}) ) 
  + T\sigma T\pi_0\, \rho\bigl(X, \sigma(\tilde{m}) \bigr) 
    \\
  &= T\sigma\, \tilde{\rho}(X, m) \,.
\end{split}
\end{equation*}
This equation can be written as the following commutative diagram:
\begin{equation*}
\xymatrix@C+2ex{
\frakg/\frakh \times \tilde{M} 
\ar[d]_{\tilde{\rho}} \ar[r]^{\Hor \times \sigma} &
\frakg \times M \ar[d]^\rho 
\\
T\tilde{M} \ar[r]^{T\sigma} & TM
}
\end{equation*}
This shows that $\Hor$ is the map that maps $(X + \frakh, \tilde{m})$ to the element in $\frakg$ that acts on $\sigma(\tilde{m})$ by the horizontal lift of the anchor of the reduced Lie algebroid.
\end{Remark}

\begin{Remark}
\label{rmk:Lu}
In Corollary \ref{cor:ReducedAlgd} we have used the $H$-reduction of Lie groupoids in order to show that there is a unique Lie algebroid structure on $\tilde{A} = \frakg/\frakh \times_H M$ such that $\frakg \ltimes M \to \tilde{A}$ is a homomorphism of Lie algebroids. This result can be generalized to situations where the action Lie algebroid $\frakg \ltimes M$ does not integrate to an action groupoid, e.g.~when not all vector fields in the image of the action $\rho:\frakg \to \calX(M)$ are complete.

It suffices to assume that $H$-acts freely and properly on $M$ such that (i) the induced $\frakh$-action is the restriction of the $\frakg$-action; (ii) the infinitesimal $\frakh$-equivariance of the $\frakg$-action lifts to the global $H$-equivariance of $\rho: \frakg \to \calX(M)$, where $H$ acts by the left adjoint action on $\frakg$ and by the push forward of vector fields on $\calX(M)$. Under these conditions \cite[Def.~2.2]{Lu:2008} it was shown in \cite[Lem.~2.3]{Lu:2008} that the quotient $\frakg \ltimes M \to \tilde{A}$ induces a Lie algebroid structure on $\tilde{A}$.
\end{Remark}

\begin{Example}
\label{ex:grel}
The Lie algebroid of \cite{BFW} arises in the following way as the Lie algebroid of a reduced Lie groupoid. Let $S$ be a manifold representing spacetime, and let $\Sigma \subset S$ be an embedded codimension 1 submanifold, which will be the $t=0$ slice of the initial value problem.   The group $G = \Diff(S)$ acts by pushforward on the manifold $M$ of lorentzian metrics on $S$. Let $H \subset G$ be the subgroup of diffeomorphisms of $S$ that fix all points in $\Sigma$. Note that $H$ is not a normal subgroup. $G$ acts by composition on the space of embeddings of $\Sigma$ in $S$. $H$ is the stabilizer group of the inclusion $\Sigma \subset G$. It follows that the quotient $H \backslash G$ can be identified with the $G$ orbit of $\Sigma$, i.e.~with the space of those embeddings of $\Sigma$ in $S$ which extend to diffeomorphisms of $S$.

The Lie algebra $\frakg$ of $G$ is the Lie algebra of vector fields on $S$; $\frakh$ is the Lie subalgebra of vector fields on $S$ that vanish on $\Sigma$. The quotient $\frakg/\frakh$ can be identified with the space of vector fields on $S$ supported on $\Sigma$, i.e.\ with sections of the restriction of $TS$ to $\Sigma$. The quotient map $\frakg \to \frakg/\frakh$ is the restriction of a vector field on $S$ to $\Sigma$.

Now  we restrict the action groupoid $G \ltimes M$ on its base to the open submanifold $M_\Sigma \subset M$ consisting of those lorentzian metrics for which $\Sigma$ is a space-like submanifold. The restricted Lie groupoid $\Gamma_\Sigma$ we obtain in this way is no longer an action groupoid, since $M_\Sigma$ is not $G$-invariant. Since $M_\Sigma$ is invariant under diffeomorphisms close to the identity, though, the Lie algebroid of $\Gamma_\Sigma$ is the action Lie algebroid $\frakg \ltimes M_\Sigma$. Moreover, $M_\Sigma$ is $H$-invariant, so that our construction of the $H$-reduced Lie algebroid can still be carried out as explained in Remark \ref{rmk:Lu}. 

For the next step we will assume that $\Sigma$ is cooriented (i.e.~with an orientation on the normal bundle, which can be viewed as choice of a time direction). The stabilizer group of $\Sigma$ is now the group of diffeomorphisms that fix all points in $\Sigma$ and preserve the coorientation, which we will denote by $H^+$. Note that $H^+$ has the same Lie algebra as $H$. From the proof of \cite[Lem.~2.8]{BFW}, we conclude that an element of $\tilde{M} := M_\Sigma/H^+$ has a representative that is given by a unique gaussian metric on an open neighborhood of $\Sigma$. The upshot is that the reduced Lie algebroid $\tilde{A} := \frakg/\frakh \times_{H^+} M_\Sigma \to \tilde{M}$ is the Lie algebroid constructed in \cite{BFW} by an entirely different approach. It was shown in \cite[Prop.~2.6]{BFW} that, as a vector bundle, $\tilde{A} \cong \frakg/\frakh \times \tilde{M}$ is trivial.
\end{Example}

\section{Hamiltonian structures on the reduced Lie algebroid}
\label{sec:HamHReduced}

We turn to the relation between hamiltonian structures on the action Lie algebroid $A = \frakg \times M$ and on the $H$-reduced Lie algebroid $\tilde{A} = \frakg/\frakh \times_H M$. The main result is that a hamiltonan structure on $\tilde{A}$ always pulls back along the projection $\phi: A \to \tilde{A}$ of Corollary \ref{cor:ReducedAlgd} to a hamiltonian structure on $A$ (Theorem  \ref{thm:HamReduced1}). However, unless $H \subset G$ is a normal subgroup, the connection on $A$ of the pulled back hamiltonian structure is not the trivial connection (Proposition \ref{prop:NormalTriv}).

\subsection{Connections on the reduced Lie algebroid}

In order to apply Proposition \ref{prop:HamPullback} to $\phi: A \to \tilde{A}$, we must find a connection $D$ on $A$ such that $\phi$ is compatible with the connections $D$ and $\tilde{D}$. For this we first make a few general observations about connections on a vector bundle associated to a principal bundle.

\begin{Lemma}
\label{lem:InducedConn1}
Let $M \to \tilde{M}$ be a left $H$-principal bundle, $W$ an $H$-module, and $\pi: W \times M \to W \times_H M$ the quotient map to the associated vector bundle. For every connection $\tilde{D}$ on $W \times_H M \to \tilde{M}$ there is a unique connection $D$ on $W \times M \to M$ so that $\pi^* \tilde{D} = D \pi^*$.
\end{Lemma}
\begin{proof}
Let the structure homomorphism of the left action of $H$ on $W$ be denoted by $L: H \to \Aut(W)$ and of the infinitesimal action by $\lambda: \frakh \to \End(W)$. Let $\tilde{U} \subset \tilde{M}$ be a contractible open subset, so that the restricted principal bundle $M|_{\tilde{U}} \to \tilde{U}$ has a section $\sigma: \tilde{U} \to M$. This induces a local trivialization $W \times_H M \cong W \times \tilde{M}$ like in Eq.~\eqref{eq:tildeAtriv}. In this trivialization, the canonical epimorphism takes the form
\begin{equation*}
\begin{aligned}
  \pi_1: W \times U 
  &\longrightarrow W \times \tilde{U}
  \\
  (w,m) 
  &\longmapsto 
  \bigl(L_{\eta(m, \sigma\pi_0 m)}^{-1} w, \pi_0 m \bigr) \,, 
\end{aligned}
\end{equation*}
where $U := \pi_0^{-1}(\tilde{U})$ and where  $\eta$ is defined as in Section \ref{sec:RedAct1}. Let $\theta: TM \to \frakh$ denote the connection 1-form of the trivialization, which is given by
\begin{equation*}
  \theta(v_m)
  = \frac{d}{dt} \eta(m_t, \eta(m_0, \sigma\pi_0 m_0)
     \cdot \sigma \pi_0 m_t) \bigr|_{t=0}
     \,,
\end{equation*}
for any smooth path $m_t \in M$ representing $v_m = \dot{m}_0$. We thus obtain
\begin{equation*}
\begin{split}
  \frac{d}{dt} L_{\eta(m_t, \sigma\pi_0 m_t)}^{-1} w \bigr|_{t=0}
  &= \frac{d}{dt} L_{\eta(m_0, \sigma\pi_0 m_0)}^{-1}
     L_{\eta(m_t, \eta(m_0, \sigma\pi_0 m_0)
     \cdot \sigma\pi_0 m_t)}^{-1} w \bigr|_{t=0}
  \\
  &= - L_{\eta(m_0, \sigma\pi_0 m_0)}^{-1}
       \lambda(\theta(\dot{m}_0)) w
  \,.
\end{split}
\end{equation*}
It follows that the differential $T\pi_1: TW \times TU \to TW \times T\tilde{U}$ is given by
\begin{equation*}
  T\pi_1 ((w,u) , v_m) 
  = \bigl( \bigl(
  L_{\eta(m, \sigma\pi_0 m)}^{-1} w,  L_{\eta(m, \sigma\pi_0 m)}^{-1} 
  (u - \lambda(\theta(v_m)) w \bigr) ,
  T\pi_0 v_m \bigr)
  \,,
\end{equation*}
where we have identified $TW \cong W \times W$.

A linear connection on the trivial bundle $W \times U \to U$ is given by a map $TU \times W \to W \times W \times TU$, $(v_m, w) \mapsto (w, \alpha(v_m)w, v_m)$, where $\alpha: TU \to \End(W)$ is the connection 1-form. Analogously, a connection on $W \times \tilde{U} \to \tilde{U}$ is given by a 1-form $\tilde{\alpha}: T\tilde{U} \to \End(W)$. The covariant derivative of the connection is given by $D_{v_m} w := \iota_{v_m} dw + \alpha(v_m) w(m)$. The map $\pi$ is compatible with the connections if $T\pi_1((w, \alpha(v_m)), v_m) = ((w, \tilde{\alpha}(T\pi_0 v_m)), T\pi_0 v_m)$, which is the case iff
\begin{equation*}
  L_{\eta(m, \sigma\pi_0 m)}^{-1}
  \bigl( \alpha(v_m) - \lambda(\theta(v_m)) \bigr) w
  = \tilde{\alpha}(T\pi_0 v_m) L_{\eta(m, \sigma\pi_0 m)}^{-1} w \,.
\end{equation*}
This equation can be solved uniquely for
\begin{equation}
\label{eq:PullbackConn}
  \alpha(v_m) = \lambda(\theta(v_m)) 
  + L_{\eta(m, \sigma\pi_0 m)} 
    \tilde{\alpha}(T\pi_0 v_m)
    L_{\eta(m, \sigma\pi_0 m)}^{-1}
    \,.
\end{equation}
We conclude that for a given connection 1-form $\tilde{\alpha}: T\tilde{U} \to \End(W)$ there is a unique connection 1-form $\alpha: TU \to \End(W)$, such that the two connections are compatible with $\pi$. 

Let now $\{ \tilde{U}_i\}_{i \in I}$ be a good cover of $\tilde{M}$ and $\{\tilde{\chi}_i: \tilde{U}_i \to [0,1]\}$ a partition of unity of $\tilde{M}$. Then $\{ U_i := \pi_0^{-1}(\tilde{U}_i)\}_{i \in I}$ is a cover of $M$ and $\{\chi_i := \tilde{\chi}_i \pi_0: U_i \to [0,1] \}$ a partition of unity of $M$. Let $\tilde{\alpha}$ be a connection 1-form on $\tilde{M}$. For every $i$ let $\tilde{\alpha}_i$ be the restriction of $\tilde{\alpha}$ to $\tilde{U}_i$ and $\alpha_i: TU_i \to \End(W)$ the connection 1-form given by~\eqref{eq:PullbackConn}. Let $\tilde{D}_i$ and $D_i$ be the covariant derivatives of the connections, in terms of which the compatibility with $\pi$ reads $\pi^* \tilde{D}_i = D_i\pi^*$. Multiplying $\tilde{D_i}$ with the partition function $\tilde{\chi}_i$ yields
\begin{equation*}
\begin{split}
  \pi^* ( \tilde{\chi}_i \tilde{D}_i )
  &= (\pi_0^* \tilde{\chi}_i)(\pi^* \tilde{D}_i)
  = \chi_i (D_i \pi^*)
  \\
  &= (\chi_i D_i ) \pi^*
  \,.
\end{split}    
\end{equation*}
Summing this equation over $i$, we obtain $\pi^* \tilde{D} = D \pi^*$, where $\tilde{D} = \sum_i \tilde{\chi}_i \tilde{D}_i$ is the covariant derivative of the connection 1-form $\tilde{\alpha}$ and $D := \sum_i \chi_i D_i$ the covariant derivative of a connection on $W \times M \to M$. This finishes the proof.
\end{proof}

\begin{Lemma}
\label{lem:InducedConn2}
Let $\pi_0: M \to \tilde{M}$ be a left $H$-principal bundle, $W$ an $H$-module, and $\pi: W \times M \to W \times_H M$ the quotient map. Let $L: H \to \Aut(W)$ and $\lambda: \frakh \to \End(W)$ denote the representations of $H$ and its Lie algebra $\frakh$. Let $\alpha: TM \to \End(W)$ be the 1-form of a connection on $W \times M$. Assume that $\alpha$ satisfies:
\begin{itemize}

\item[(i)] $\alpha(v_m) = \lambda(v_m)$ for all $m \in M$ and vertical $v_m \in \ker T_m\pi_0 \cong \frakh$.

\item[(ii)] $\alpha(h \cdot v_m) = L_h \alpha(v_m) L_h^{-1}$ for all $v_m \in TM$ and $h \in H$.

\end{itemize}
Then $\alpha$ descends to a unique connection on $W \times_H M$.
\end{Lemma}
\begin{proof}
We use the same notation as in the proof of Lemma~\ref{lem:InducedConn1}. In particular, $\theta$ is the connection 1-form of the local trivialization of $W \times_H M \to \tilde{M}$ that is induced by a local section $\sigma$ of the principal bundle $M \to \tilde{M}$. Solving Eq.~\eqref{eq:PullbackConn} for $\tilde{\alpha}$, we obtain the condition
\begin{equation*}
  \tilde{\alpha}(T\pi_0 v_m)
  = L_{\eta(m, \sigma\pi_0 m)}^{-1} \bigl( 
    \alpha(v_m) - \lambda(\theta(v_m)) 
    \bigr) L_{\eta(m, \sigma\pi_0 m)}^{-1}
    \,.
\end{equation*}
We have to show that this equation can be viewed as a definition for $\tilde{\alpha}$.

First, we observe that for $v_m \in \ker T_m\pi_0 \cong \frakh$ we have $\theta(v_m) = v_m$, so that $\alpha(v_m) - \lambda(\theta(v_m)) = \alpha(v_m) - \lambda(v_m) = 0$ by assumption (i).

Secondly, we observe since $T\pi_0 (h \cdot v_m) = T\pi_0 v_0$, the right hand side evaluated on $h \cdot v_m$ must be independent of $h$. By construction, $\theta$ and $\lambda$ are $H$-equivariant
\begin{align*}
  \theta(h \cdot v_m) 
  &= \Ad_h \theta(v_m) \\
  \lambda(\Ad_h X) 
  &= L_h \alpha(X) L_h^{-1} \,,
\end{align*}
so that $\lambda(\theta(h \cdot v_m)) = L_h \lambda(\theta(h \cdot v_m)) L_h^{-1}$. By assumption (ii) $\alpha$ satisfies the analogous equivariance property. We conclude that
\begin{equation*}
\begin{split}
  \tilde{\alpha}(T\pi_0 (h \cdot v_m))
  &= L_{\eta(h \cdot m, \sigma\pi_0 (h \cdot m))}^{-1} \bigl( 
     \alpha(h \cdot v_m) - \lambda(\theta(h \cdot v_m)) 
     \bigr) L_{\eta(h \cdot m, \sigma\pi_0 (h \cdot m))}^{-1}
  \\
  &= L_{\eta(m, \sigma\pi_0 m)}^{-1} L_h^{-1} \bigl( 
     L_h \alpha(v_m) L_h^{-1} - L_h \lambda(\theta(v_m)) L_h^{-1} 
     \bigr) L_h L_{\eta(m, \sigma\pi_0 m)}^{-1}
  \\
  &= \tilde{\alpha}(T\pi_0 v_m) \,,
\end{split}
\end{equation*}
which shows that $\tilde{\alpha}$ is well-defined.
\end{proof}

\begin{Corollary}
\label{cor:TrivConnCompat}
The trivial connection on $W \times M$ is compatible with a connection on $W \times_H M$ if and only if the infinitesimal action of $\frakh$ on $W$ is trivial. In that case the compatible connection is that of the natural trivialization $W \times_H M  \cong W \times \tilde{M}$. 
\end{Corollary}
\begin{proof}
The trivial connection on $W \times M$ is given by the zero 1-form $\alpha = 0$. Looking at Eq.~\eqref{eq:PullbackConn}, we see that $\alpha(v_m) = \lambda(v_m)$ for $v_m \in \ker T_m \pi_0$, which is zero for all $v_m \in \frakh$ if and only if $\lambda = 0$. In that case $\tilde{\alpha}$ must be zero as well.
\end{proof}

\subsection{Pullback of the hamiltonian structure on the reduced Lie algebroid}

We now have all the technical prerequisites to prove the main result:

\begin{Theorem}
\label{thm:HamReduced1}
For every (weakly) hamiltonian structure on the $H$-reduced action Lie algebroid $\frakg/\frakh \times_H M$ there is a (weakly) hamiltonian structure on $\frakg \ltimes M$, such that the quotient map $\frakg \ltimes M \to \frakg/\frakh \times_H M$ is a morphism of (weakly) hamiltonian Lie algebroids.
\end{Theorem}

\begin{proof}
The quotient map $\phi: A \to \tilde{A}$ factors as $\frakg \times M \to \frakg/\frakh \times M \stackrel{\pi}{\to} \frakg/\frakh \times_H M$ through the map $\pi$ to which we can apply lemma~\ref{lem:InducedConn1}. So if $D$ is a connection on $\tilde{A}$, then there is a unique connection $D'$ on $\frakg/\frakh \times M$ such that $\pi^* \tilde{D} = D' \pi^*$. 

Let $\alpha': TM \to \End(\frakg/\frakh)$ be the connection 1-form of $D'$. A connection 1-form $\alpha: TM \to \End(\frakg)$ is compatible with respect to the projection $\frakg \times M \to \frakg/\frakh \times M$ if
\begin{equation}
\label{eq:alphaproj}
  \alpha(v_m)X + \frakh = \alpha'(v_m)(X + \frakh) \,,
\end{equation}
for all $v_m \in TM$ and $X \in \frakg$. We can obtain such an $\alpha$ by choosing a splitting $i:\frakg/\frakh \to \frakg$ of the quotient map $p: \frakg \to \frakg/\frakh$ and setting
\begin{equation*}
  \alpha(v_m)(X) :=  i \alpha'(X + \frakh) \,,
\end{equation*}
so that $\alpha(v_m)X + \frakh = pi\alpha'(v_m)(X + \frakh) = \alpha'(v_m)(X + \frakh)$.

By construction, the connection 1-form $\alpha$ descends to the connection 1-form $\alpha'$, which in turn descends to $\tilde{\alpha}$. We conclude that there is a connection $D$ on $A$ satisfying $D \phi^* = \phi^* \tilde{D}$.

Let $\tilde{\omega}$ be a presymplectic form on $\tilde{M}$ and $\tilde{\mu}$ a (weakly) hamiltonian $\tilde{D}$-momentum section of $(\tilde{A}, \tilde{\omega})$. We can apply Proposition \ref{prop:HamPullback}, which states that $(A, \phi^* \tilde{\omega}, D, \phi^*\mu)$ is a (weakly) hamiltonian Lie algebroid.
\end{proof}

\begin{Remark}
The connection $D$ of Theorem  \ref{thm:HamReduced1} is not unique. In fact, the set of connections that are compatible with a given connection $\tilde{D}$ on $\tilde{A}$ is a vector space that can be described as follows: The connection 1-form $\alpha$ defined by Eq.~\eqref{eq:alphaproj} satisfies $\alpha(v_m)(\frakh) = 0$, so that $\alpha$ can be viewed as an element of $\Omega^1(M) \otimes \frakg \otimes \frakh^\circ \subset \Omega^1(M) \otimes \frakg \otimes \frakg^* \cong \Omega^1(M) \otimes \End(\frakg)$, where $\frakh^\circ = \{ \phi \in \frakg^*~|~ \phi(\frakh) = 0\}$ is the annihilator of $\frakh$. Moreover, the difference of two connection 1-forms $\alpha_1$ and $\alpha_2$ subject to Eq.~\eqref{eq:alphaproj} satisfies $(\alpha_1(v_m) - \alpha_2(v_m)) X \in \frakh$. We conclude that the space of compatible connections is isomorphic to $\Omega^1(M) \otimes \frakh \otimes \frakh^\circ$.
\end{Remark}

Before we state the next result, we recall that, if $H$ is a normal subgroup, then the adjoint action of $H$ on $H\backslash G$ is trivial because, for any coset $Hg$ and any $h\in H$, there is some $h' \in H$ for which $gh^{-1}=h'g$, and so
\begin{equation*}
h(Hg)h^{-1} = hH(gh^{-1})= hH(h'g) =(hHh')g = Hg\,.
\end{equation*}
This implies that the adjoint action of $H$ on $\frakg/\frakh$ is trivial as well, so that the associated $\frakg/\frakh$ bundle has a natural trivialization, $\frakg/\frakh \times_H M \cong \frakg/\frakh \times \tilde{M}$.

\begin{Proposition}
\label{prop:NormalTriv}
Let $H$ and $G$ be connected. Then the trivial connection on $\frakg \ltimes M$ is compatible with some connection on $\frakg/\frakh \times_H M$ if and only if $H$ is a normal subgroup. In that case the compatible connection is the trivial connection of the reduced action Lie algebroid $\frakg/\frakh \times_H M \cong \frakg/\frakh \ltimes \tilde{M}$.
\end{Proposition}

\begin{proof}
The trivial connection on $\frakg \times M$ descends to the trivial connection on $\frakg/\frakh \times M$. By Corollary \ref{cor:TrivConnCompat} the trivial connection on $\frakg/\frakh \times M$ descends to a connection on $\frakg/\frakh \times_H M$ if and only if the adjoint action of $\frakh$ on $\frakg/\frakh$ is trivial. This is equivalent to $[\frakg,\frakh] \subset \frakh$, i.e.~$\frakh$ is a Lie algebra ideal. If $H$ and $G$ are connected this implies that $H$ is a normal subgroup.

If $H$ is a normal subgroup, then the $H$-reduced Lie algebroid is the action Lie algebroid of the induced action of $\frakg/\frakh$ on $\tilde{M}$. And by Corollary \ref{cor:TrivConnCompat} the trivial connection on $\frakg \times M$ descends to the trivial connection of $\frakg/\frakh \ltimes \tilde{M}$.
\end{proof}

From Proposition \ref{prop:NormalTriv} and Theorem \ref{thm:HamReduced1} we can draw the following remarkable conclusions:
\begin{itemize}

\item[(a)] The equivariant momentum map for a $G$-action cannot induce a hamiltonian Lie algebroid structure on the $H$-reduced Lie algebroid unless $H$ is normal.

\item[(b)] If the $H$-reduced Lie algebroid is hamiltonian, then the action Lie algebroid $\frakg \ltimes M$ is also hamiltonian but not with respect to the trivial connection unless $H$ is normal. 
 
\end{itemize}

In example~\ref{ex:cylinder2} we have already encountered the phenomenon that a Lie algebra action that is not hamiltonian in the usual sense may have a hamiltonian action Lie algebroid if we chose a connection other than the trivial one.  Observation (b) tells us that, for action groupoids with a hamiltonian reduction, this is the general case.

\section{Cohomological interpretation of hamiltonian Lie algebroids}
\label{sec:CohomInterp}

In their seminal paper \cite{AtiyahBott:1984} Atiyah and Bott made the following observation: The action of a Lie group $G$ or on a symplectic manifold is hamiltonian if and only if the symplectic form has a closed extension in a complex that computes the equivariant cohomology of the $G$-manifold $M$. This yields an insightful interpretation of the Duistermaat-Heckman formula for the symplectic volume of an integrable symplectic manifold as a localization formula of a more general type. In this section we will show that the Atiyah-Bott characterization of the equivariant momentum maps of Lie algebra actions generalizes to the bracket-compatible momentum sections of Lie algebroids.

\subsubsection{Conventions and notation for graded manifolds}
\label{sec:GradManNotation}

We adopt the following conventions for graded manifolds. All our graded manifolds arise from graded vector bundles, which are therefore used to denote the graded manifolds. (By Batchelor's theorem \cite{Batchelor1979}, this is always the case for $\bbZ_2$-graded manifolds; however, not all $\bbZ$-graded manifolds arise from graded vector bundles.) The structure rings involve a dualization of the fibre. That is, if $E$ is a graded vector bundle over the $n$-dimensional manifold $M$ with local trivialization $E|_U \cong U \times V$ over $U \subset M$, the structure sheaf is given by $\calO_{E}(U) = C^\infty(U) \otimes S(V^*)$, where $S(V^*)$ is the graded commutative algebra freely generated by the graded vector space $V^*$. Observe that the graded dualization inverts the degree, the degree $k$-component being given by $(V^*)_k = (V_{-k})^*$.

For degree shifts we use the cohomological postfix notation $V[k]$, whose degree $j$ component is given by $V[k]_j = V_{j+k}$. The shifted tangent functor $T[k]$ shifts only the degrees of the tangent directions: On a local trivialization we have $T[k]E|_U \cong U \times (\bbR^n[k] \oplus V \oplus V[k])$, where $\bbR^n$ is the fibre of $TU$. The degree shift of a graded manifold and of its tangent bundle commute, $T[j](E[k]) = (T[j]E)[k]$, so that the notation $T[j]E[k]$ is unambiguous. On a local trivialization we have
$T[j]E[k]|_U \cong U \times (\bbR^n[j] \oplus V[k] \oplus \times V[j+k] )$.

Besides $\bbZ$-gradings, we will also consider $(\bbZ \times \bbZ)$-graded manifolds, which will be called bigraded to mark the difference. (The sign rule is $\mathrm{sgn}(p,q) = (-1)^{p+q}$.) We will denote a shift by bidegree $(p,q)$ by $E[p,q]$, omitting the inner parentheses.

The global ring of functions on a supermanifold $E$ will be denoted by $\calO(E) \equiv \calO_E(M)$. A vector field of degree $k$ on a graded manifold $E$ is by definition a graded derivation of degree $k$ of $\calO(E)$. In this terminology, the de Rham differential on $\Omega(M)$ is a vector field of degree 1 on the graded manifold $T[1]M$.

Let $A \to M$ be an ordinary vector bundle, i.e.~with typical fibre $V$ concentrated in degree 0. On a trivialization $A \cong U \times V$ over $U \subset M$ we have the following structure rings:
\begin{align}
  \calO_{A[1]}(U)
  &= C^\infty(U) \otimes S(V[1]^*)
  \notag\\
  \calO_{T[1]A[1]}(U)
  &= C^\infty(U) \otimes S(\bbR^n[1]^* \oplus V[2]^* \oplus V[2]^*)
  \label{eq:OWeil1}\\
  \calO_{T[1,0]A[0,1]}(U)
  &= C^\infty(U) \otimes S(\bbR^n[1,0]^* \oplus V[0,1]^* \oplus V[1,1]^*)
  \,.
  \notag
\end{align}    
The graded commutative algebra generated by the odd vector space $V[1]^*$ is the exterior algebra $\wedge V^*$, the one generated by the even vector space $V[2]^*$ is the symmetric algebra $SV^*$.

\subsection{Generalization of equivariant cohomology models to Lie algebroids}

Constructions of models of the equivariant cohomology $H((M \times EG) / G)$ of a $G$-manifold $M$ proceed by a) constructing a model of the cohomology of $M \times EG$ which is built out of infinitesimal data yet big enough so that we can b) identify a differential subcomplex that is a model of the cohomology of the homotopy quotient $(M \times EG)/G$.

For step a) an infinitesimal model of the cohomology of $M \times EG$ can be given in terms of the action Lie algebroid $\frakg \ltimes M$ \cite{MehtaAlgebroids:2009}, in a way which generalizes to an arbitrary Lie algebroid $A$ as follows. The \textbf{Weil algebra} of $A$ is defined to be the graded algebra
\begin{equation*}
 W(A) := \calO(T[1]A[1]) \,,
\end{equation*}
of functions on the graded manifold $T[1]A[1]$, which is the graded algebra of differential forms on the graded manifold $A[1]$. For the purpose of this paper the most natural choice for the differential on $W(A)$ is
\begin{equation*}
 \DW = d + \Lie_{\DA} \,,
\end{equation*}
where $d$ is the de Rham differential and $\Lie_{\DA}$ is the Lie derivative with respect to the differential $\DA$ of the Lie algebroid cohomology of $A$, the latter considered as a vector field on $A[1]$. 
It is shown in \cite{MehtaAlgebroids:2009} (Corollary 5.16) that this differential complex is an infinitesimal model of  $H^\bullet(M) \cong H^\bullet(M \times EG)$. By analogy with the case of a group action it is  called there the BRST-model.

Step b) is more difficult. In the case of a group action we want to find the subcomplex of those elements of the model for $M \times EG$ that are the pullbacks of forms on $(M \times EG)/G$ along the bundle projection. To avoid technical subtleties about infinite dimensional manifolds, consider first a finite-dimensional $G$-principal bundle $P \to B$. Let $\alpha: \frakg \to \calX(P)$ be the corresponding infinitesimal action of the Lie algebra. A form $\phi \in \Omega(P)$ is the pullback of a form on $B$ if and only if (i) $\iota_{\alpha(X)} \phi = 0$ and (ii) $\Lie_{\alpha(X)} \phi = 0$ for all $X \in \frakg$. Forms which satisfy (i) are called \textbf{horizontal}, forms which satisfy (ii) are called \textbf{invariant}, and forms which satisfy both properties are called \textbf{basic}.\footnote{Note that, unlike horizontal and basic tangent vectors, horizontal and basic forms are well-defined without the use of a connection.} The space of basic differential forms is a differential subcomplex. When $G$ is connected and compact, this subcomplex is a model of the cohomology of $(M \times EG)/G$ (see \cite{GuilleminSternberg:1999} and references therein).

In the case of a general Lie algebroid $A$, we want to  generalize the definition of basic forms to obtain a subcomplex of $W(A)$ which, under assumptions analogous to compactness and connectedness of the group in the case of action Lie algebroids, is a model for the stack cohomology of the stack presented by the groupoid integrating $A$. The problem  of finding the most general assumptions such that the basic subcomplex which we will define below models the stack cohomology has been considered by a number of authors but remains open.\footnote{For related work see   \cite{AbadPhD:2008}, \cite{AbadCrainic:2011}, \cite{Crainic:2003}, \cite{LiblandMeinrenken:2014}, \cite{MehtaAlgebroids:2009}, \cite{WeinsteinXu:1991}.} For our purposes, this is not an issue, though, since we are working with the Lie algebroid itself, the role of the stack being merely heuristic.

\subsection{The basic subcomplex of the Weil algebra}

In the case that $A = \frakg \ltimes M$ is an action Lie algebroid, the shifted tangent bundle $T[1]A[1]$ naturally splits as a vector bundle over $M$
\begin{equation*}
T[1](M \times \frakg)[1] \cong T[1]M \oplus \frakg[1] \oplus \frakg[2].
\end{equation*}
Therefore, the Weil algebra factors as a tensor product, with
\begin{equation}
\label{eq:standardWeil}
 W(\frakg \ltimes M)_k
 \cong \bigoplus_{k = p + q + 2r} 
 \Omega^p(M) \otimes \wedge^q \frakg^* \otimes S^r \frakg^* 
 \,,
\end{equation}
which is Eq.~\eqref{eq:OWeil1} for $V = \frakg$.

\begin{Remark}
We write the action Lie algebroid as $\frakg \ltimes M$ with $\frakg$ on the left to indicate that it is a left action; i.e.~the anchor induces a homomorphism of Lie algebras (see also Section \ref{sec:conventions} on our conventions). In graded manifolds, however, and for de Rham cohomology with coefficients (as in Section \ref{sec:Category}) it is customary to place the fibre of a vector bundle on the right. We will follow this convention here, which is why we have to use the isomorphism of vector bundles $\frakg \ltimes M \cong M \times \frakg$ in the last two equations and in the next equation.
\end{Remark}

A graded function $\phi \in W(\frakg \ltimes M)$ is horizontal if $\IW_X \phi = 0$ for all $X \in \frakg$ where the interior derivative $\IW_X$ with respect to $X \in \frakg$ is given by inserting $X$ into the second factor, $\IW_X (\tau \otimes \alpha \otimes \beta) := (-1)^{|\tau|} \omega \otimes (\iota_X \alpha) \otimes \beta$. The subalgebra of horizontal functions is then given by
\begin{equation*}
 \bigl( W(\frakg \ltimes M)_{\mathrm{hor}} \bigr)_k 
 \cong \bigoplus_{k = p + 2r} 
 \Omega^p(M) \otimes S^r \frakg^* \,.
\end{equation*}

For a general vector bundle $A$, the tangent bundle $T[1]A[1]$ is not naturally a graded vector bundle over $M$, but it does have the natural structure of a graded manifold with core $M$ and structure ring $W(A) = \calO(T[1]A[1])$. A factorization like \eqref{eq:standardWeil} amounts to a splitting of the short exact sequence of graded manifolds,

\begin{equation}
\label{eq:TAseq}
\xymatrix{ 
0 \ar[r] 
& A[2] \ar[r]
& T[1]A[1] \ar[r] 
& T[1]M \oplus A[1] 
\ar[r] 
& 0 \,,
}
\end{equation}
by the horizontal  lift of a linear connection on $A$, which in the case of an action Lie algebroid we have taken to be the product connection on $M \times \frakg$.

\begin{Remark}
 A sequence of (non-negatively) graded manifolds with the same core like \eqref{eq:TAseq} is called exact if it induces an exact sequence of vector spaces on all degree $> 0$ components of the structure rings. Alternatively, \eqref{eq:TAseq} can be viewed as a short exact sequence of double vector bundles \cite{GraciasazMehta:2010}.
\end{Remark}

It is useful for our purposes to refine the grading to a bigrading $T[1,0]A[0,1]$. (The notation was explained in Section~\ref{sec:GradManNotation}.) The sequence~\eqref{eq:TAseq} then becomes a sequence of bigraded manifolds:
\begin{equation}
\label{eq:BigradedShortExact}
\xymatrix{ 
0 \ar[r] 
& A[1,1] \ar[r] 
& T[1,0]A[0,1] \ar[r]_-p 
& T[1,0]M \oplus A[0,1] \ar@<-1ex>[l]_-{h} 
\ar[r] 
& 0 \,.
}
\end{equation}
Since the horizontal lift $h: TM \oplus A \to TA$ of a linear connection on $A$ is linear in $TM$ and affine in $A$ it respects the bigrading, so that it is a right inverse of $p$ in the category of graded manifolds. In other words, $h$ induces an isomorphism $T[1,0]A[0,1] \cong T[1,0]M \oplus A[0,1] \oplus A[1,1]$.

On the rings of functions, the maps $p$ and $h$ induce morphisms
\begin{equation}
\label{eq:pStarhStar}
\xymatrix@C+2ex{
  W(A) = \calO(T[1,0]A[0,1])
  \ar@<+0.5ex>[r]^-{h^*}
  & \calO(T[1,0]M \oplus A[0,1])
  = \Omega(M,A) \ar@<+0.5ex>[l]^-{p^*}
} \,,
\end{equation}
where $\Omega(M,A)$ is the bigraded algebra of Section \ref{sec:Category}. The pullbacks satisfy $h^* p^* = (ph)^* = \id$. The map
\begin{equation*}
  (hp)^* : W(A) \longrightarrow W(A) \,,
\end{equation*}
is the projection onto the image of $p^*$. The complementary projection $\id - (hp)^*$ is the projection onto the kernel of $p^*$, which is a subalgebra of $W(A)$ that is isomorphic to $\calO(A[1,1])$. We conclude that the choice of a connection amounts to a factorization of the bigraded Weil algebra
\begin{equation}
\label{eq:Weilfactorization}
 W(A)_{p,q}
 \cong \bigoplus_{\substack{ p = j+l\\ q = k+l}}
 \Omega^j(M) \otimes_{C^\infty(M)} 
 \Gamma(\wedge^k A^*) \otimes _{C^\infty(M)}
 \Gamma(S^l A^*) \,,
\end{equation}
where we have replaced $\Omega(M,A)$ with its explicit form~\eqref{eq:OmegaMA}. 

For every section $a$ of $A[1]$ the \textbf{$A$-interior derivative} operator  $\IW_a$ on the factorized Weil algebra is defined by the insertion of $a$ into the second factor:
\begin{equation*}
\IW_a (\tau \otimes \alpha \otimes \beta) := (-1)^{|\tau|} \tau\otimes (\IA_a \alpha) \otimes \beta.
\end{equation*}
It will be used to define horizontal and invariant forms in Definition \ref{def:basic} below along with the \textbf{$A$-Lie derivative}, which is defined as the graded commutator
\begin{equation*}
 \LW_a := [\IW_a, \DW] = \IW_a \DW + \DW \IW_a \,, 
\end{equation*}
where $\DW = d + \Lie_{\DA}$ is the BRST-differential on $W(A)$. Since $d$ is of bidegree $(1,0)$ and $\Lie_{\DA}$ is of bidegree $(0,1)$, the $A$-Lie derivative has components of bidegrees $(1,-1)$ and $(0,0)$.

\begin{Definition}
\label{def:basic}
Let $A \to M$ be a Lie algebroid with a connection. A graded function $\phi \in W(A)$ is called
\begin{itemize}

\item[(i)] \textbf{horizontal} if $\IW_a \phi = 0$ for all sections $a$  of $A$,

\item[(ii)] \textbf{invariant} if $\LW_a \phi = 0$ for all sections $a$  of $A$,

\item[(iii)] \textbf{basic} if it is horizontal and invariant.

\end{itemize}
\end{Definition}

Since $\IW_a$ is a derivation of the same degree $(0,-1)$ for all $a$, the set of horizontal elements of $W(A)$ is a bigraded subalgebra. Since $\LW_a$ is a derivation of the same total degree $0$ for all $a$, the set of invariant elements is a graded (but not bigraded) subalgebra. It follows that  the set of all basic elements $W(A)$ is a graded subalgebra of $W(A)$. Moreover, for a basic function we have $\IW_a \DW \phi = \LW_a \phi = 0$ and $\LW_a \DW \phi = \DW \LW_a \phi = 0$. This implies the following:

\begin{Proposition}
The set $W(A)_\mathrm{bas}$ of basic elements of $W(A)$ is a differential graded subalgebra of $\bigl( W(A), \DW \bigr)$.
\end{Proposition}

\begin{Remark}
There are two bundles in play here. On the one hand, we have the vector bundle $A \to M$. On the other hand, we have the principal bundle $M \times EG \to (M \times EG)/G$ and, implicitly, its generalization to differentiable stacks. Therefore, we have also two notions of ``horizontal''. On the one hand, we have the notion of horizontal vectors of the connection $D$ on $A$. On the other hand, we have the notion of horizontal forms of Definition \ref{def:basic}. We believe that it is always clear from the context which one is meant, so that we do not distinguish them linguistically.
\end{Remark}

\subsection{Statement of the main theorem}

In the case of action Lie algebroids, equivariant momentum maps can be identified with closed basic extensions of symplectic forms regardless of whether the basic subcomplex is a model for equivariant cohomology or not.\footnote{Recall that we require a hamiltonian momentum map to be equivariant with respect to the action of the Lie algebra.} We will now prove that this statement generalizes to bracket-compatible momentum sections of Lie algebroids. Recall from~\eqref{eq:pStarhStar} that the projection $p: TA \to TM \oplus A$ and the horizontal lift $h: TM \oplus A \to TA$ of a connection induce morphisms $p^*$ and $h^*$ between the rings of graded functions.

\begin{Definition}
\label{def:WAextension}
Let $A \to M$ be a Lie algebroid with a connection. An \textbf{extension} of $\tau \in \Omega(M,A)$ to $W(A)$ is an element $\bar{\tau} \in W(A)$ of the same total degree as $\tau$ such that $h^* \bar{\tau} = \tau$.
\end{Definition}

\begin{Remark}
Since $h^*p^* = \id$, every $\tau \in \Omega(M,A)$ has the extension $p^* \tau \in W(A)$. However, $p^* \tau$ is generally not basic.
\end{Remark}

\begin{Theorem}
\label{thm:EquivExt}
A presymplectically anchored Lie algebroid $(A,D,\omega)$ is hamiltonian if and only if $\omega$ has a $\DW$-closed basic extension to $W(A)$.
\end{Theorem}

While the original result by Atiyah and Bott is more of an observation than a theorem, the proof of Theorem~\ref{thm:EquivExt} is quite involved. One method of proof is by lengthy calculations in local coordinates which involve the connection on $A$, its curvature, and its torsion. A better and more insightful proof, which we will present in Section \ref{subsec:proof} below, is obtained by relating the bigraded algebra of Section \ref{sec:Category} to the Weil algebra $W(A)$.

\subsection{The Cartan calculi}
\label{sec:CartanCalculi}

By definition, the algebra of functions on the graded manifold $A[1]$ is the graded algebra of exterior $A$-forms, $\calO(A[0,1])_k = \Gamma(M, \wedge^k A^* )$. The Lie algebroid interior derivative with respect to a section $a \in \Gamma(A)$ is the derivation of degree $-1$ given by $\IA_a \nu = \nu(a)$ for $\nu \in \Gamma(A^*)$. In the terminology of graded manifolds, $\IA_a$ is a vector field on $A[1]$ of degree $-1$. Geometrically, we may think of $\IA_a$ as the vector field constant on fibres of $A[1]$ whose value along each fibre is given by the value of $a$ at the basepoint in $M$. In this terminology the differential $\DA$ of Lie algebroid cohomology is a vector field on $A[1]$ of degree +1. The Lie algebroid Lie derivative by $a$ is then defined as the graded commutator $\LA_a = [\IA_a, \DA],$ which is a vector field of degree 0. Denoting the vector space of graded vector fields on $A[1]$ by $\calX(A[1])$ we have 
along with $ \DA   \in \calX(A[1])_{1} $ the vector fields
\begin{equation*}
 \IA_a \in \calX(A[1])_{-1} \,,\quad
 \LA_a \in \calX(A[1])_{0} \,,
\end{equation*}
for every section $a \in \Gamma(A)$. The subspace spanned by these derivations is a graded Lie subalgebra of $\calX(A[1])$ with commutation relations
\begin{equation*}
\label{eq:LieAlgdCartanCalc}
\begin{gathered}{}
  [\DA,\DA] = 0 \,,\\\
  [\IA_a, \IA_b] = 0 \,,\quad
  [\IA_a, \DA] = \LA_a \,,\\
  [\LA_a, \IA_b] = \IA_{[a,b]} \,,\quad
  [\LA_a, \LA_b] = \LA_{[a,b]} \,,\quad
  [\LA_a, d] = 0 \,,
\end{gathered}
\end{equation*}
for all $a, b \in \Gamma(A)$. We call this graded Lie algebra the \textbf{Lie algebroid Cartan calculus} or the \textbf{$A$-Cartan calculus}.

On $\calO(T[1,0]A[0,1])$, which is the bigraded algebra of differential forms on $A[0,1]$, we have the Cartan calculus of the de Rham complex consisting of the de Rham differential,
\begin{equation*}
 d \in \calX(T[1,0]A[0,1])_{1,0} \,,
\end{equation*}
together with the interior derivative and Lie derivative which are maps from bigraded vector fields on $A[0,1]$ to bigraded vector fields on $T[1,0]A[0,1]$,
\begin{equation*}
\begin{aligned}
 \iota: \calX(A[0,1]) &\longrightarrow 
 \calX(T[1,0]A[0,1])[-1,0] \,,\\ 
 \Lie: \calX(A[0,1]) &\longrightarrow 
 \calX(T[1]A[1])[0,0] \,.
\end{aligned}
\end{equation*}
The interior derivative is $\calO(A[0,1])$-linear and so, a fortiori, $C^\infty(M)$-linear. The de Rham differential, interior derivatives, and Lie derivatives also span a bigraded Lie subalgebra of the derivations on $W(A)$ with commutator relations 
\begin{equation*}
\label{eq:DeRhamCartanCalc}
\begin{gathered}{}
  [d,d] = 0 \,,\\
  [\iota_v, \iota_w] = 0 \,,\quad
  [\iota_v, d] = \Lie_v \,,\\
  [\Lie_v, \Lie_w] = \iota_{[v,w]} \,,\quad
  [\Lie_v, \Lie_w] = \Lie_{[v,w]} \,,\quad
  [\Lie_v, d] = 0 \,,
\end{gathered}
\end{equation*}
for all $v, w \in \calX(A[0,1])$. Moreover, the Lie derivative satisfies the Leibniz rule
\begin{equation*}
\label{eq:deRhamLeibniz}
\begin{split}
 \Lie_{fv}
 &= [\iota_{fv}, d] = f \iota_v d - (-1)^{|v|-1} d\, \iota_{fv} \\
 &= (-1)^{|v|} df\, \iota_v + f \Lie_v \,,
\end{split}
\end{equation*}
for every bigraded function $f$ and bigraded vector field $v$ on $A[0,1]$. We call this the \textbf{de Rham Cartan calculus} on $A[0,1]$.

Since the operations $\DA$, $\IA_a$, and $\LA_a$ of the Lie algebroid Cartan calculus are vector fields in $\calX(A[0,1])$, the interior derivative and Lie derivative maps of the de Rham Cartan calculus on $A[0,1]$ can be applied to them. We thus obtain for every section $a$ of $A$ the vector fields
\begin{equation*}
 \iota_{\IA_a}, \iota_{\LA_a}, \iota_{\DA}, \Lie_{\IA_a}, \Lie_{\LA_a}, \Lie_{\DA}, d \in \calX(T[1,0]A[0,1])
\end{equation*}
which have the following bidegrees:
\begin{alignat*}{3}
  \deg \iota_{\IA_a} &= (-1,-1) \,,\quad&
  \deg \iota_{\LA_a} &= (-1,0) \,,\quad&
  \deg \iota_{\DA} &= (-1,1)
  \\ 
  \deg \Lie_{\IA_a} &= (0,-1) \,,&
  \deg \Lie_{\LA_a} &= (0,0) \,,&
  \deg \Lie_{\DA} &= (0,1)
  \\ 
  && \deg d &= (1,0)\,. && 
\end{alignat*}
The subspace of $\calX(T[1,0]A[0,1])$ spanned by these bigraded vector fields is closed under the commutator bracket. The commutator relations are straight-forward to compute:

First, all interior derivatives commute:
\begin{equation*}
\begin{gathered}{}
 [\iota_{\IA_a}, \iota_{\IA_b}]
 = [\iota_{\IA_a}, \iota_{\LA_b}]
 = [\iota_{\IA_a}, \iota_{\DA}]
 = 0  \,,
\\
 [\iota_{\LA_a}, \iota_{\LA_b}]
 = [\iota_{\LA_a}, \iota_{\DA}] 
 = [\iota_{\DA}, \iota_{\DA}] = 0 \,.
\end{gathered}
\end{equation*}
The Lie derivative is a homomorphism of graded Lie algebras:
\begin{equation*}
\begin{gathered}{}
 [\Lie_{\DA}, \Lie_{\DA}] = 0 \,,\quad
 [\Lie_{\IA_a}, \Lie_{\IA_b} ] = 0 \,, 
 \quad [\Lie_{\IA_a}, \Lie_{\DA}] = \Lie_{\LA_a} \,, \\
 [\Lie_{\LA_a}, \Lie_{\DA}] = 0 \,, \quad
 [\Lie_{\LA_a}, \Lie_{\IA_b}] = \Lie_{\IA_{[a,b]} } \,,\quad
 [\Lie_{\LA_a}, \Lie_{\LA_b}] = \Lie_{ \LA_{[a,b]} } \,.
\end{gathered}
\end{equation*}
The Lie derivative acts on the interior derivative by the adjoint action on the argument:
\begin{equation*}
\begin{gathered}{}
 [\Lie_{\IA_a}, \iota_{\IA_b}] = 0 \,,\quad
 [\Lie_{\IA_a}, \iota_{\IA_b}] = \iota_{\IA_{[a,b]}} \,,\quad
 [\Lie_{\IA_a}, \iota_{\DA}] = \iota_{\LA_a} \,, 
\\
 [\Lie_{\LA_a}, \iota_{\IA_b}] = \iota_{\IA_{[a,b]}} \,,\quad
 [\Lie_{\LA_a}, \iota_{\LA_b}] = \iota_{\LA_{[a,b]}} \,,\quad
 [\Lie_{\LA_a}, \iota_{\DA}] = 0 \,,
\\
 [\Lie_{\DA}, \iota_{\IA_b}] = \iota_{\LA_b} \,,\quad
 [\Lie_{\DA}, \iota_{\LA_b}] = 0 \,,\quad
 [\Lie_{\DA}, \iota_{\DA}] = 0 \,.
\end{gathered}
\end{equation*}
The commutator of the interior derivative with the differential is the Lie derivative:
\begin{equation*}
 [\iota_{\IA_a}, d] = \Lie_{\IA_a} \,,\quad
 [\iota_{\LA_a}, d] = \Lie_{\LA_a} \,,\quad
 [\iota_{\DA}, d] = \Lie_{\DA} \,. 
\end{equation*}
Finally, the de Rham differential commutes with the Lie derivative:
\begin{equation*}
 [d, \Lie_{\IA_a}] = 
 [d, \Lie_{\LA_a}] = 
 [d, \Lie_{\DA}] = 0 \,. 
\end{equation*}

\subsection{Local bigraded coordinates}
\label{sec:LocCoord}

\subsubsection{Bigraded coordinates}

Let $\{x^\alpha\}$ be local coordinates on $U \subset M$ and $\{a_i \in \Gamma(U,A)\}$ a basis of local sections of $A$, which together define a local trivialization of the vector bundle $A$. Let $\{\theta^i\}$ be the basis of local sections of $A^*$ dual to $\{a_i\}$, $\theta^i(a_j) = \delta^i_j$. Then $\{x^\alpha, \theta^i \}$ are local bundle coordinates of $A$. We shall use Greek indices $\alpha, \beta$ to label the coordinates of base manifold $M$ and Latin indices $i,j,k$ for the fibre coordinates of $A$.

We view $x^\alpha$ and $\theta^i$ as local coordinate functions of the bigraded manifold $A[0,1]$ of degrees $\deg(x^\alpha) = (0,0)$ and $\deg(\theta^i) = (0,1)$ that generate the bigraded ring of functions $\calO(A[0,1])|_U = C^\infty(U) \otimes \Gamma(U, \wedge A^*)$. The corresponding coordinate vector fields have bidegrees $\deg(\frac{\partial}{\partial x^\alpha}) = (0,0)$ and $\deg(\frac{\partial}{\partial \theta^i}) = (0,-1)$. As it is customary for graded manifolds, we will denote the coordinate 1-forms dual to the coordinate vector fields by $\dot{x}^\alpha \equiv dx^\alpha$ and $\dot{\theta}^i \equiv d\theta^i$,
\begin{equation}
\label{eq:DualGradedCoord}
 \iota_{\frac{\partial}{\partial x^\alpha}}\dot{x}^\beta 
 = \delta^\beta_\alpha \,,\quad
 \iota_{\frac{\partial}{\partial \theta^k}} \dot{\theta}^j = \delta^j_k \,.
\end{equation}
The functions $\{ x^\alpha, \theta^i, \dot{x}^\alpha, \dot{\theta}^i \}$ are local coordinates of the graded manifold $T[1,0]A[0,1]$ of bidegrees
\begin{equation*}
 \deg(x^\alpha) = (0,0)\,,\quad
 \deg(\theta^i) = (0,1)\,,\quad
 \deg(\dot{x}^\alpha) = (1,0)\,,\quad
 \deg(\dot{\theta}^i) = (1,1) \,.
\end{equation*}

\subsubsection{The de Rham Cartan calculus}

The action of the de Rham differential on the local coordinate functions is given by
\begin{equation*}
 dx^\alpha = \dot{x}^\alpha \,,\quad 
 d\theta^i = \dot{\theta}^i \,, \quad
 d\dot{x}^\alpha = 0 \,,\quad
 d\dot{\theta}^i = 0 \,.
\end{equation*}
Viewed as bidegree $(1,0)$ vector field on $T[1,0]A[0,1]$, it can be written as
\begin{equation}
\label{eq:deRhamDcoord}
 d 
 = \dot{x}^\alpha \frac{\partial}{\partial x^\alpha} 
 + \dot{\theta}^i \frac{\partial}{\partial \theta^i} \,.
\end{equation}
The interior derivative of the graded de Rham complex of $A[1]$ is given by Eq.~\eqref{eq:DualGradedCoord} and by zero for all other interior derivatives of graded coordinates. Viewed as graded vector fields on $T[1]A[1]$ the interior derivatives of the coordinate vector fields are written as
\begin{equation*}
 \iota_{\frac{\partial}{\partial x^\alpha}} 
 = \frac{\partial}{\partial \dot{x}^\alpha}
 \,,\quad
 \iota_{\frac{\partial}{\partial \theta^k}} 
 = \frac{\partial}{\partial \dot{\theta}^k}\,.
\end{equation*}
The Lie derivative $\Lie_v = [\iota_v, d]$ with respect to the coordinate vector fields is given by
\begin{equation*}
 \Lie_{\frac{\partial}{\partial x^\alpha} } 
 = \frac{\partial}{\partial x^\alpha} \,,\quad
 \Lie_{\frac{\partial}{\partial \theta^i} } 
 = \frac{\partial}{\partial \theta^i} \,.
\end{equation*}

\subsubsection{The Lie algebroid Cartan calculus}

In terms of the basis of local sections of $A$ the Lie algebroid structure takes the form
\begin{equation*}
 [a_i, a_j] = c^k_{ij} a_k \,,\quad
 \rho(a_i) = \rho^\alpha_i \frac{\partial}{\partial x^\alpha} \,,
\end{equation*}
where $c^k_{ij}, \rho^\alpha_i \in C^\infty(M)$ are the structure functions. The Lie algebroid differential is given in terms of local bundle coordinates by
\begin{equation}
\label{eq:dAlocal}
 \DA
 = \rho^\alpha_i \theta^i \frac{\partial}{\partial x^\alpha}
 - \frac{1}{2} c^k_{ij} \theta^i \theta^j 
   \frac{\partial}{\partial\theta^k} \,.
\end{equation}
The Lie algebroid interior derivative is given by $\IA_{a_k} \theta^j = \theta^j(a_k) = \delta^j_k$, which can be viewed as graded vector field
\begin{equation*}
 \IA_{a_k} = \frac{\partial}{\partial\theta^k} \,.
\end{equation*}
For the Lie algebroid Lie derivative defined as $\LA_a = [\IA_a, \DA]$ we thus obtain
\begin{equation*}
 \LA_{a_i} 
 = \rho^\alpha_i \frac{\partial}{\partial x^\alpha}
 - c^k_{ij} \theta^j \frac{\partial}{\partial\theta^k} \,.
\end{equation*}

\subsubsection{The BRST-differential}

The de Rham interior derivative with respect to a vector field on $A[0,1]$ is given by
\begin{equation*}
  \iota_{\frac{\partial}{\partial x^\alpha}}
  = \frac{\partial}{\partial \dot{x}^\alpha} \,,\quad
  \iota_{\frac{\partial}{\partial \theta^i}}
  = \frac{\partial}{\partial \dot{\theta}^i} \,. 
\end{equation*}
For the interior derivative with respect to the Lie algebroid differential we thus obtain
\begin{equation}
\label{eq:iotaDA}
 \iota_{\DA} 
 = \rho^\alpha_i \theta^i \frac{\partial}{\partial \dot{x}^\alpha}
 - \frac{1}{2} c^k_{ij} \theta^i \theta^j 
   \frac{\partial}{\partial\dot{\theta}^k} \,.
\end{equation}
The de Rham Lie derivative with respect to $\DA$ is then given by Cartan's magic formula,
\begin{equation}
\label{eq:LiedCoord}
\begin{split}
 \Lie_{\DA}
 &=
 [\iota_{\DA}, d\, ] 
 = \Bigl[ \rho^\alpha_i \theta^i 
   \frac{\partial}{\partial \dot{x}^\alpha}
 - \frac{1}{2} c^k_{ij} \theta^i \theta^j 
   \frac{\partial}{\partial\dot{\theta}^k} ,
 \dot{x}^\beta \frac{\partial}{\partial x^\beta} 
 + \dot{\theta}^i \frac{\partial}{\partial \theta^i} \Bigr]
\\
 &= 
 \rho^\alpha_i \theta^i \frac{\partial}{\partial x^\alpha}
 - \frac{\partial \rho^\alpha_i}{\partial x^\beta} 
   \dot{x}^\beta \theta^i
   \frac{\partial}{\partial \dot{x}^\alpha} 
 - \rho^\alpha_i \dot{\theta}^i 
   \frac{\partial}{\partial \dot{x}^\alpha} 
 + \frac{1}{2}\frac{\partial c^k_{ij}}{\partial x^\beta} 
   \dot{x}^\beta \theta^i \theta^j 
   \frac{\partial}{\partial\dot{\theta}^k} 
\\
 &\quad
 - \frac{1}{2} c^k_{ij} \theta^i \theta^j 
   \frac{\partial}{\partial\theta^k}
 - c^k_{ij} \theta^i \dot{\theta}^j 
   \frac{\partial}{\partial\dot{\theta}^k}
\,.
\end{split}
\end{equation}
The BRST-differential in local coordinates is the sum of this expression and the one for the de Rham differential given in Eq.~\eqref{eq:deRhamDcoord}.

\subsection{Parallel projection of derivations}

Let $p^*$ and $h^*$ be the maps of~\eqref{eq:pStarhStar}. They induce a linear map between the spaces of bigraded endomorphisms,
\begin{equation}
\label{eq:HorProjLift}
\begin{aligned}
  P: \End\bigl( W(A) \bigr) 
  &\longrightarrow \End\bigl( \Omega(M,A) \bigr) 
  \\
  X &\longmapsto h^* X p^* \,.
\end{aligned}
\end{equation}
We call $P$ the \textbf{parallel projection} of bigraded endomorphisms.

\begin{Proposition}
\label{prop:DerivProject}
Let $A$ be a vector bundle with a linear connection. The parallel projection~\eqref{eq:HorProjLift} maps derivations to derivations.
\end{Proposition}

\begin{proof}
Let $X$ be a graded derivation on $W(A)$ and $\alpha,\beta \in \Omega(M,A)$. Then
\begin{equation*}
\begin{split}
  (PX)(\alpha \beta)
  &= (h^* X p^*)(\alpha \beta)
  = h^*\bigl[ X \bigl( (p^*\alpha)(p^*\beta) \bigr) \bigr]
  \\
  &= h^*\bigl[
  \bigl( X (p^*\alpha) \bigr)\, (p^*\beta)
  + (-1)^{|X|\,|\alpha|}
  (p^*\alpha)\, \bigl( X(p^*\beta) \bigr) 
  \bigr]
  \\
  &=
  h^*\bigl( X (p^*\alpha) \bigr)\, (h^*p^*\beta)
  + (-1)^{|X|\,|\alpha|}
  (h^*p^*\alpha) \, h^*\bigl( X(p^*\beta) \bigr)
  \\
  &=
  \bigl( (PX) \alpha \bigr)\, \beta
  + (-1)^{|PX|\,|\alpha|}
  \alpha \, \bigl( (PX)\beta \bigr) \,,
\end{split}
\end{equation*}
where we have used that $ph = \id$ and that $|PX| = |X|$.
\end{proof}

\begin{Remark}
For two derivations $X$, $Y$ of $W(A)$ we have $P(XY) = h^* XY p^*$ which is generally different from $(PX)(PY) = h^* X p^* h^* Y p^*$, since $p^* h^*$ is the projection onto the kernel of $h^*$. When $A$ has non-zero rank then $p^* h^*$ has a non-zero kernel, which shows that $P$ does not preserve the Lie bracket of derivations.
\end{Remark}

\subsubsection{Parallel projection in local coordinates}

Let $D: \Gamma(A) \to \Omega^1(M) \otimes_{C^\infty(M)} \Gamma(A)$ denote the covariant derivative of the connection on $A$. On the basis of local sections of $A$ the covariant derivative acts as
\begin{equation*}
  D a_i = \Omega^j_{\alpha i} dx^\alpha \otimes a_j \,,
\end{equation*}
where $\Omega^j_{\alpha j} \in C^\infty(M)$ are the connection coefficients. On dual coordinates the covariant derivative acts by $D \theta^i = - \Omega^i_{\alpha j} \dot{x}^\alpha \theta^j$. As we have seen in Section \ref{sec:CatLieAlgConn}, the covariant derivative can be extended to a derivation on $\Omega(M,A)$. In local coordinates this derivation is given by
\begin{equation}
\label{eq:DVectorField}
  D = \dot{x}^\alpha \frac{\partial}{\partial x^\alpha}
  - \Omega^j_{\alpha i} \dot{x}^\alpha 
  \theta^i \frac{\partial}{\partial \theta^j} \,.
\end{equation}

The covariant derivative and the horizontal lift of the connection are related by $D_v a = (Ta)(v) - h(v,a)$, where $Ta : TM \to TA$ is the derivative of the section $a:M \to A$. In local coordinates the horizontal lift of $v = v^\alpha \frac{\partial}{\partial x^\alpha}$ to $a = a^i a_i$ is, therefore, given by
\begin{equation*}
  h(v,a) = v^\alpha \frac{\partial}{\partial x^\alpha} 
  - \Omega^i_{\alpha j} v^\alpha a^j \frac{\partial}{\partial \theta^i} \,.
\end{equation*}
It follows that the pullback along $h$ acts on the coordinate 1-forms as
\begin{equation*}
\begin{aligned}
 (h^*\dot{x}^\alpha)(v,a) 
 &= v^\alpha = \dot{x}^\alpha(v) \,, \\ 
 (h^*\dot{\theta}^i)(v,a) 
 &= - \Omega^i_{\alpha j} v^\alpha a^j
 = - \Omega^i_{\alpha j} 
   \dot{x}^\alpha(v)\, \theta^j(a) \,.
\end{aligned}
\end{equation*}
Moreover, since $h(0_m, a) = 0_m$, the pullback along $h$ preserves the coordinates $x^\alpha$ and $\theta^i$. We conclude that $h^*$ acts on the local generators of $W(A)$ by
\begin{equation*}
 h^* x^\alpha = x^\alpha \,,\quad
 h^* \theta^i = \theta^i \,,\quad
 h^* \dot{x}^\alpha = \dot{x}^\alpha \,,\quad
 h^* \dot{\theta}^i = - \Omega^i_{\alpha j} \dot{x}^\alpha \theta^j \,.
\end{equation*}
The projection $p: T[1,0]A[0,1] \to T[1,0]M \oplus A[0,1]$ is given $p(v_a) = (T\pi\, v, a)$, where $\pi: A \to M$ is the bundle projection. The pullback along $p$ then acts on the the generators of $\Omega(M,A) = \calO(T[1,0]M\oplus A[0,1])$ by
\begin{equation*}
  p^* x^\alpha = x^\alpha \,,\quad
  p^* \dot{x}^\alpha = \dot{x}^\alpha \,,\quad
  p^* \theta^i = \theta^i \,.
\end{equation*}
With the formulas for $h^*$ and $p^*$ we can compute the parallel projection of the coordinates,
\begin{equation}
\label{eq:hpstarcoord}
  (hp)^* x^\alpha = x^\alpha \,,\quad
  (hp)^* \theta^i = \theta^i \,,\quad
  (hp)^* \dot{x}^\alpha = \dot{x}^\alpha \,,\quad
  (hp)^* \dot{\theta}^i = - \Omega^i_{\alpha j} \dot{x}^\alpha \theta^j \,.
\end{equation}
For the parallel projection of the coordinate vector fields we obtain
\begin{equation*}
  P\Bigl(\frac{\partial}{\partial x^\alpha}\Bigr)
  = \frac{\partial}{\partial x^\alpha} 
  \,,\quad
  P\Bigl(\frac{\partial}{\partial \dot{x}^\alpha}\Bigr)
  = \frac{\partial}{\partial \dot{x}^\alpha} 
  \,,\quad
  P\Bigl(\frac{\partial}{\partial \theta^i}\Bigr)
  = \frac{\partial}{\partial \theta^i} 
  \,,\quad
  P\Bigl(\frac{\partial}{\partial \dot{\theta}^i}\Bigr)
  = 0
   \,.
\end{equation*}

\subsection{The horizontal subalgebra}

According to the factorization~\eqref{eq:Weilfactorization}, the bigraded subalgebra of horizontal elements is isomorphic to the tensor product over $C^\infty(M)$ of two factors: The first factor is the image of $\Omega(M)$ under the pullback along $T\pi: T[1,0]A[0,1] \to T[1,0]M$. The second factor is the algebra of functions on the kernel of $p$. In local coordinates the first factor is generated by $\{x^\alpha, \dot{x}^\alpha\}$. The second factor is the image of the projection $\id - (hp)^*: W(A) \to W(A)$, which is generated as $C^\infty(M)$-algebra by
\begin{equation*}
  \eta^i 
  := (\id - p^* h^*) \dot{\theta^i}
  = \dot{\theta}^i 
 + \Omega^i_{\alpha j} \dot{x}^\alpha \theta^j \,.
\end{equation*}
We conclude that the subalgebra of horizontal functions of $W(A)$ is given on a local coordinate neighborhood $U \subset M$ by
\begin{equation*}
  W(A)_\mathrm{hor} \bigr|_U  
  \cong C^\infty(U) \otimes 
  \bbR[\Dot{x}^\alpha, \eta^i] 
  \cong \Omega(U) \otimes \bbR[\eta^i]\,.
\end{equation*}
From this we can deduce that the $A$-interior derivative with respect to a section $a = a^i a_i$ of $A$ is given in local coordinates by
\begin{equation*}
  \IW_{a} 
  =
    a^i\frac{\partial}{\partial \theta^i}
    + a^i \Omega^j_{\alpha i}
    \dot{x}^\alpha \frac{\partial}{
      \partial \dot{\theta}^j}
   \,.
\end{equation*}

In order to give a coordinate free interpretation of $\IW_a$ we observe that when $a$ is a horizontal section of $A$, that is, $d a^i + a^j\Omega^i_{\alpha j} \dot{x}^\alpha = 0$, then
\begin{equation*}
\begin{split}
  \Lie_{\IA_a} 
  &= [\iota_{\IA_a}, d] 
  \\
  &= [a^i \iota_{\IA_{a_i}}, d] 
  \\
  &= -(da^i) \iota_{\IA_{a_i}} +  a^i \Lie_{\IA_{a_i}}
  \\
  &= a^j \Omega^i_{\alpha j} \dot{x}^\alpha 
    \frac{\partial}{\partial \dot{\theta}^i}
     + a^i \frac{\partial}{\partial \theta^i}
  \\
  &= \IW_{a} \,.
\end{split}
\end{equation*}
We see that the $A$-interior derivative with respect to a horizontal section $a$ is the prolongation of the vertical vector field $\IA_a = a^i \frac{\partial}{\partial \theta^i}$ on $A$ to the tangent bundle $TA$. This follows the same approach as many other definitions in the paper: The definition of the interior derivative with respect to a Lie algebra is generalized to a Lie algebroid by requiring the usual condition to hold for all horizontal sections. In order to get a coordinate free expression for $\IW_a$ that holds for arbitrary sections $a$ we have to subtract the terms coming from the covariant derivative of $a$. This leads to
\begin{equation*}
  \IW_a = \Lie_{\IA_a} +\, \iota_{[D, \IA_a]} \,,
\end{equation*}
where $D$ is given by Eq.~\eqref{eq:DVectorField} and where we have implicitly extended the de Rham interior derivative $\Omega(M)$-linearly from $\calX(A)$ to $\Omega(M) \otimes_{C^\infty(M)} \calX(A)$, similar to what we did in Section \ref{sec:CatLieAlgConn}.

\subsection{Relation to the bigraded algebra of Section \ref{sec:Category}}

In Section \ref{sec:Category} we have constructed two  derivations $D$ and $\check{D}$ on the bigraded algebra $\Omega(M,A)$, which can be viewed as the algebra of bigraded functions on $T[1,0]M \oplus A[0,1]$. These derivations are related to the Cartan calculus on $W(A)$ as follows:

\begin{Proposition}
\label{prop:BRSTproject}
Let $A$ be a Lie algebroid with a connection. The bigraded derivations $D$ and $\check{D}$ on $\Omega(M,A) = \calO(T[1,0]M \oplus A[0,1])$ are the parallel projections
\begin{equation*}
  D = P(d) \,,\quad \check{D} = P(\Lie_\DA)
\end{equation*}
of the two components of the BRST-differential.
\end{Proposition}

\begin{proof}
$P(d) = h^* d p^*$ acts on the local coordinate functions as
\begin{align*}
  P(d) x^\alpha 
  &= h^* \dot{x}^\alpha = \dot{x}^\alpha = Dx^\alpha
  \\
  P(d) \dot{x}^\alpha &= 0 = D\dot{x}^\alpha
  \\
  P(d) \theta^i 
  &= h^* \dot{\theta}^i
  = - \Omega^i_{\alpha j} \dot{x}^\alpha 
  =  D \theta^i \,.
\end{align*}
We see that the actions of $D$ and $P(d)$ on the coordinate functions are equal. By Proposition \ref{prop:DerivProject}, $P(D)$ is a derivation. Since two derivations are equal if they are equal on the generators of an algebra, it follows that $D = P(d)$.

For the second equation we compare the action of $\check{D}$ and $P(\Lie_\DA)$ on the local coordinate functions. On the one hand, $\check{D}$ acts on $x^\alpha$ and $\theta^i$ as the Lie algebroid differential,
\begin{equation*}
  \check{D} x^\alpha = \rho^\alpha_i \theta^i
  \,,\quad
  \check{D}\theta^i = - \frac{1}{2} c^i_{jk} \theta^j \theta^k
  \,.
\end{equation*}
Let now $v = v^\alpha \frac{\partial}{\partial x^\alpha}$ be an arbitrary vector field on $M$ and $a = a^i a_i$ an arbitrary section of $A$. Then
\begin{equation*}
\begin{split}
  \langle \check{D}_a \dot{x}^\alpha, v \rangle
  &= \rho a \cdot \langle \dot{x}^\alpha, v \rangle
  - \langle \dot{x}^\alpha, \check{D}_a v \rangle 
  \\
  &= \rho a \cdot \langle \dot{x}^\alpha, v \rangle
  - \langle \dot{x}^\alpha, [\rho a, v] + \rho(D_v a) \rangle 
  \\  
  &= \langle \Lie_{\rho a} \dot{x}^\alpha, v \rangle
  - \langle \dot{x}^\alpha, \rho(D_v a) \rangle 
  \\  
  &= v^\beta \frac{\partial (a^i \rho^\alpha_i)}{\partial x^\beta}
  - \rho^\alpha_i v^\beta \Bigl( 
  \frac{\partial a^i}{\partial x^\beta}
  + a^j \Omega^i_{\beta j} \Bigr)
  \\
  &= \Bigl\langle (\IA_a \theta^i) \Bigl(
  \frac{\partial \rho^\alpha_i}{\partial x^\beta}
  - \rho^\alpha_i \Omega^i_{\beta j} \Bigr) \dot{x}^\beta,
  v \Bigr\rangle
  \,,
\end{split}
\end{equation*}
from which it follows that
\begin{equation*}
  \check{D} \dot{x}^\alpha
  = \Bigl(
  - \frac{\partial \rho^\alpha_i}{\partial x^\beta}
  + \rho^\alpha_j \Omega^j_{\beta i} \Bigr) 
  \dot{x}^\beta \theta^i
  \,.
\end{equation*}
On the other hand, applying $P(\Lie_\DA)$ to $x^\alpha$, $\dot{x}^\alpha$, and $\theta^i$ using the local coordinate expression~\eqref{eq:LiedCoord}, we obtain
\begin{align*}
  P(\Lie_\DA) x^\alpha 
  &= h^* \Lie_\DA x^\alpha 
  = h^* \rho^\alpha_i \theta^i = \rho^\alpha_i \theta^i 
  \\
  P(\Lie_\DA) \dot{x}^\alpha 
  &= h^* \Lie_\DA \dot{x}^\alpha 
  = h^* \Bigl( 
   - \frac{\partial \rho^\alpha_i}{\partial x^\beta} 
   \dot{x}^\beta \theta^i
 - \rho^\alpha_i \dot{\theta}^i 
  \Bigr)  
  = \Bigl( 
   - \frac{\partial \rho^\alpha_i}{\partial x^\beta} 
   + \rho^\alpha_j \Omega^j_{\beta i} 
  \Bigr) \dot{x}^\beta \theta^i
  \\
  P(\Lie_\DA) \theta^i 
  &= h^* \Lie_\DA \theta^i
  = h^* \Bigl(
   - \frac{1}{2} c^i_{jk} \theta^j \theta^k 
  \Bigr)
  =  - \frac{1}{2} c^i_{jk} \theta^j \theta^k 
  \,. 
\end{align*}
Since the action of $\check{D}$ and $P(\Lie_\DA)$ on the local coordinates of $\Omega(M,A)$ are equal, we conclude that $\check{D} = P(\Lie_\DA)$.
\end{proof}

\begin{Proposition}
\label{prop:irhocommute}
Let $\iota_\rho$ be the derivation on $\Omega(M,A)$ defined in Section \ref{sec:CatLieAlgConn}. Then
\begin{equation*}
  p^* \iota_\rho = \iota_{\DA} p^* \,.
\end{equation*}
\end{Proposition}
\begin{proof}
In local coordinates $\iota_\rho = \rho^\alpha_i \theta^i \frac{\partial}{\partial x^\alpha}$, whereas $\iota_\DA$ is given by Eq.~\eqref{eq:iotaDA}. The relation is now straightforward to check in local coordinates.
\end{proof}

\subsection{Some technical lemmas}
For the proof of Theorem  \ref{thm:EquivExt}, we need the following technical lemmas.

\begin{Lemma}
\label{lem:vertdvanish}
Let $\tau \in \Omega^{p,q}(M,A)$ with $q > 0$. The following are equivalent:
\begin{itemize}
\item[(i)] $\bigl(\id - (hp)^*\bigr)d p^* \tau = 0$.
\item[(ii)] $\tau = 0$.
\end{itemize}
\end{Lemma}
\begin{proof}
In local coordinates $\tau = \tau_{AI} \dot{x}^{\alpha_1} \cdots \dot{x}^{\alpha_p}
\theta^{i_1}\cdots \theta^{i_q}$, where $A = (\alpha_1, \ldots, \alpha_p)$ and $I = (i_1, \ldots, i_q)$ are multi-indices and $\tau_{AI}$ are smooth functions on $M$ that are totally antisymmetric in $A$ and $I$. Using the Leibniz rule for $d$, the antisymmetry of $\tau_{AI}$ in $A$, Eqs.~\eqref{eq:hpstarcoord}, and the definition of $\eta^i$, we obtain 
\begin{equation*}
  \bigl(\id - (hp)^*\bigr) dp^*\tau 
  = (-1)^p
  \tau_{AI} \dot{x}^{\alpha_1} \cdot \dot{x}^{\alpha_p}
  \eta^{i_1} \theta^{i_2}\cdots \theta^{i_q}
  \,.
\end{equation*}
This vanishes if and only if $\tau_{AI} = 0$, that is, if and only if $\tau = 0$.
\end{proof}

\begin{Lemma}
\label{lem:dpstar}
Let $\tau \in \Omega^{p,q}(M,A)$ with $q > 0$. Then $dp^* \tau = 0$ if and only if $\tau = 0$.
\end{Lemma}
\begin{proof}
If $\tau = 0$ then $dp^*\tau = 0$. Conversely, if $dp^*\tau = 0$ then $\bigl(\id - (hp)^*\bigr)dp^*\tau = 0$. Lemma~\ref{lem:vertdvanish} now implies that $\tau = 0$.
\end{proof}

\begin{Lemma}
\label{lem:pstarhor}
If $\phi \in W(A)$ is of bidegree $(p,0)$ or $(0,q)$, then it satisfies $\phi = (hp)^* \phi$.
\end{Lemma}
\begin{proof}
If $\phi$ is of bidegree $(p,0)$ or $(0,q)$ then its local coordinate expression does not contain any factor $\dot{\theta}$. The relation $\phi = (hp)^* \phi$ then follows from Eqs.~\eqref{eq:hpstarcoord}.
\end{proof}

\subsection{Proof of Theorem~\ref{thm:EquivExt}}
\label{subsec:proof}
We start by observing that in local coordinates every horizontal extension of $\omega$ is for degree reasons of the form
\begin{equation*}
  \bar{\omega} 
  = \tfrac{1}{2}\omega_{\alpha\beta} 
    \dot{x}^\alpha \dot{x}^\beta 
    + \mu_i \eta^i
\end{equation*}
for some smooth functions $\mu_i \in C^\infty(M)$. Every horizontal element of $W(A)$ which is $\DW$-closed is automatically basic, so that we only have to determine the conditions for $\DW\bar{\omega} = 0$.

The second term of the extension can be written as
\begin{equation*}
\begin{split}
  \mu_i \eta^i
  &= (\id - p^* h^*) d (\mu_i \theta^i)
  = (\id - p^* h^*) d p^* \mu \\
  &=  d p^* \mu - p^* D\mu \,,
\end{split}
\end{equation*}
where $\mu := \mu_i \theta^i$ is an arbitrary exterior 1-form on $A$. Our goal is now to show that
\begin{equation*}
  \bar{\omega} = p^* \omega + d p^* \mu - p^* D\mu
\end{equation*}
is $\DW$-closed if and only if $d\omega = 0$ and $\mu$ is a bracket-compatible $D$-momentum section.

Applying $\DW = d + \Lie_{\DA}$ to $\bar{\omega}$ produces terms of bidegrees $(3,0)$, $(2,1)$, and $(1,2)$ which have to vanish separately for $\DW \bar{\omega}$ to vanish. The three equations we thus obtain are:
\begin{subequations}
\begin{align}
 0 &= d p^*\omega 
 \label{eq:ExtCond1}\\
 0 &= \Lie_{\DA} p^* \omega - d p^* D\mu
 \label{eq:ExtCond2}\\
 0 &= \Lie_{\DA} (d p^* \mu - p^* D\mu)  \label{eq:ExtCond3}
\end{align}
\end{subequations}

For Eq.~\eqref{eq:ExtCond1} we use lemma~\ref{lem:pstarhor}, which yields
\begin{equation*}
  d p^*\omega = (hp)^* d p^*\omega 
  = p^* D\omega = p^* d\omega \,.
\end{equation*}
Since $p$ is surjective, its pullback $p^*$ is injective so that $p^* d\omega$ vanishes if and only if $d\omega = 0$. We conclude that Eq.~\eqref{eq:ExtCond1} is satisfied if and only if $d\omega = 0$.

We turn to Eq.~\eqref{eq:ExtCond2}. Using Cartan's magic formula and Eq.~\eqref{eq:ExtCond1}, we can rewrite it as
\begin{equation}
\label{eq:AtBo1}
\begin{split}
  0 &=
  \Lie_{\DA} p^* \omega - d p^* D\mu
  = (\iota_\DA d - d\iota_\DA) p^* \omega - d p^* D\mu
  \\
  &= - d p^* (\iota_\rho \omega + D\mu)
  \,,
\end{split}
\end{equation}
where in the last step we have used Proposition \ref{prop:irhocommute}. Applying lemma~\ref{lem:dpstar} to Eq.~\eqref{eq:AtBo1}, we conclude that Eq.~\eqref{eq:ExtCond2} holds if and only if $\iota_\rho \omega + D\mu = 0$. This is precisely the condition in Proposition \ref{prop:OmegaHamConds} for $\mu$ to be a momentum section.

We finally turn to Eq.~\eqref{eq:ExtCond3}, which we split into two equations by applying the complementary projections $(hp)^*$ and $\id - (hp)^*$. The first equation we obtain is
\begin{equation}
\label{eq:AtBo2}
\begin{split}
  0 
  &=
  (\id - p^*h^*) \Lie_{\DA} (d p^* \mu - p^* D\mu)
  \\
  &=
    (\id - p^*h^*) (-d \Lie_{\DA} p^* \mu 
    - \Lie_{\DA} p^* D\mu)
  \\
  &=
  - (\id - p^*h^*) (d p^* h^* \Lie_{\DA} p^* \mu 
    + \Lie_{\DA} p^* D\mu)
  \\
  &= - (\id - p^*h^*) ( dp^* \check{D} \mu 
  + \Lie_{\DA} p^* D\mu)
  \,,
\end{split}
\end{equation}
where we have used $\Lie_\DA d = - d \Lie_\DA$, then applied lemma~\ref{lem:pstarhor} to $\Lie_\DA p^* \mu$ which is of bidegree $(0,2)$, and finally used Proposition \ref{prop:BRSTproject}. In oder to rewrite the second term, we need the relation
\begin{equation*}
\begin{split}
  d p^* \iota_\rho \iota_\rho \omega
  &= d \iota_\DA \iota_\DA p^* \omega
  = (\iota_\DA d - \Lie_\DA) \iota_\DA 
    p^* \omega
  = - (\iota_\DA \Lie_\DA + \Lie_\DA \iota_\DA) 
    p^* \omega
  \\
  &= - 2 \Lie_\DA \iota_\DA p^* \omega
  \,,
\end{split}
\end{equation*}
where we have used $dp^*\omega = 0$. By substituting $\iota_\rho \omega = - D\mu$, we get
\begin{equation*}
  \Lie_\DA p^* D\mu
  = - \Lie_\DA p^* \iota_\rho \omega
  = - \Lie_\DA \iota_\DA p^* \omega
  = d p^* \tfrac{1}{2} \iota_\rho \iota_\rho \omega
  \,.
\end{equation*}
Inserting this into Eq.~\eqref{eq:AtBo2}, we arrive at the condition
\begin{equation*}
\label{eq:AtBo3}
  0 = (\id - p^*h^*) dp^* (\check{D} \mu 
  + \tfrac{1}{2} \iota_\rho \iota_\rho \omega )
  \,.
\end{equation*}
By lemma~\ref{lem:vertdvanish}, this equation is satisfied if and only if $\check{D} \mu + \tfrac{1}{2} \iota_\rho \iota_\rho \omega = 0$. This is precisely the condition in Proposition \ref{prop:OmegaHamConds} for the momentum section $\mu$ to be bracket-compatible.

We now apply the complimentary projection $(hp)^*$ to Eq.~\eqref{eq:ExtCond3} and obtain
\begin{equation*}
\begin{split}
  0 
  &=
  (hp)^* \Lie_{\DA} (d p^* \mu - p^* D\mu)
  \\
  &=  p^* h^* (- d \Lie_\DA p^* \mu 
  - \Lie_{\DA} p^* D\mu)
  \\
  &= - p^* h^* (d p^* h^* \Lie_\DA p^* \mu 
  + \Lie_{\DA} p^* D\mu)
  \\
  &= - p^* ( D \check{D} + \check{D} D)\mu \,,
\end{split}
\end{equation*}
where we have used the same identities and lemmas as before. Since $p^*$ is injective, this equation holds if and only if $[D, \check{D}] \mu = 0$. We have already established that $d\omega = 0$, $D\mu = - \iota_\rho \omega$, and $\check{D}\mu = -\tfrac{1}{2} \iota_\rho \iota_\rho \omega$ are necessary conditions for $\bar{\omega}$ to be a closed basic extension of $\omega$. With these three relations we obtain
\begin{equation*}
\begin{split}
  [D,\check{D}]\mu
  &= 
  D (-\tfrac{1}{2} \iota_\rho \iota_\rho \omega)
  + \check{D} (-\iota_\rho \omega)
  \\
  &= 
  -  \tfrac{1}{2} D \iota_\rho \iota_\rho \omega
  - (\iota_\rho D - D \iota_\rho D
  + \iota_T ) \iota_\rho \omega
  \\
  &= 
    \iota_\rho \check{D} \omega
  + \tfrac{1}{2} D \iota_\rho \iota_\rho \omega
  - \iota_T \iota_\rho \omega
  \,,
\end{split}
\end{equation*}
where we have used $D\omega = 0$, which implies that $\check{D}\omega = (\iota_\rho D - D \iota_\rho - \iota_T)\omega = - D\iota_\rho \omega$. We now need the following lemma:

\begin{Lemma}
\label{lem:Drhorhoomega}
Every 2-form $\omega \in \Omega^{2,0}(M,A)$ satisfies
\begin{equation}
\label{eq:AtBo4}
 \iota_\rho \check{D} \omega 
 + \tfrac{1}{2} D\iota_\rho \iota_\rho \omega
 - \iota_T \iota_\rho \omega
 = \tfrac{1}{2} \iota_\rho \iota_\rho d \omega
\end{equation}
\end{Lemma}

\begin{proof}
The proof follows from a lengthy calculation. Eq.~\eqref{eq:AtBo4} is an equation of elements in $\Omega^{1,2}(M,A)$, so we can evaluate it on a vector field $v$ and two sections $a$, $b$ of $A$. We consider every summand on the left hand side of~\eqref{eq:AtBo4} separately. For the first summand we first establish the relation
\begin{equation*}
\begin{split}
  \iota_v \iota_b \iota_a \iota_\rho
  &= \iota_v \iota_b (\iota_{\rho a} + \iota_\rho \iota_a)
  =  \iota_v \bigl(-\iota_{\rho a} \iota_b  
    + ( \iota_{\rho b} + \iota_\rho \iota_b) \iota_a \bigr)
  \\
  &=  \iota_v (\iota_{\rho b} \iota_a - \iota_{\rho a} \iota_b  
      + \iota_\rho \iota_b \iota_a ) 
  \,,
\end{split}
\end{equation*}
where we have used  $[\iota_a ,\iota_\rho] = \iota_a \iota_\rho -  \iota_\rho \iota_a = \iota_{\rho a}$. With this relation we compute
\begin{equation*}
\begin{split}
  A 
  &:= 
  \iota_v \iota_b \iota_a \iota_\rho \check{D} \omega
  = \iota_v \iota_{\rho b} \iota_a \check{D}\omega
   - \iota_v \iota_{\rho a} \iota_b \check{D}\omega
  \\
  &= \rho a \cdot \omega(\rho b, v) 
   - \omega(\check{D}_a \rho b, v)
   - \omega(\rho b, \check{D}_a v)
  \\
  &{}
   - \rho b \cdot \omega(\rho a, v)
   + \omega(\check{D}_b \rho a, v)
   + \omega(\rho a, \check{D}_b v)
  \,.
\end{split}
\end{equation*}
For the second summand, we recall from the proof of Proposition \ref{prop:OmegaHamConds} that $\iota_b \iota_a \tfrac{1}{2} \iota_\rho \iota_\rho \omega 
= \omega(\rho a, \rho b)$. Using the definition~\eqref{eq:DualAconn} of the dual connection of the $A$-connection $\check{D}$ on $TM$, we obtain for the second term
\begin{equation*}
\begin{split}
  B 
  &=
  \iota_v \iota_b \iota_a D \tfrac{1}{2} \iota_\rho \iota_\rho \omega
  = \iota_b \iota_a  \iota_v D \tfrac{1}{2} \iota_\rho \iota_\rho \omega
  = \iota_b \iota_a  D_v \tfrac{1}{2} \iota_\rho \iota_\rho \omega
  \\
  &= v \cdot \omega(\rho a, \rho b) 
     - \omega(\rho D_v a , \rho b) - \omega(\rho a, \rho D_v b)
  \,.
\end{split}
\end{equation*}
For the third term we use the definition~\eqref{eq:DefTorsion2} of the torsion,
\begin{equation*}
\begin{split}
  C 
  &:= 
  - \iota_v \iota_b \iota_a \iota_T \iota_\rho \omega
  = 
  - \iota_v \iota_{T(a,b)} \omega
  \\
  &= 
  - \omega(\rho D_{\rho a} b - \rho D_{\rho b} a - [\rho a, \rho b], v)
  \,.
\end{split}
\end{equation*}
Using the definition $\check{D}_v a = [\rho a, v] + \rho D_v a$ of the $A$-connection~\eqref{eq:DcheckDef}, the  sum of the three terms can be written as
\begin{equation*}
\begin{split}
  A + B + C 
  &= \rho a \cdot \omega(\rho b, v) 
   - \rho b \cdot \omega(\rho a, v)
        + v \cdot \omega(\rho a, \rho b)
  \\
  &{}\quad
  - \omega(\check{D}_a \rho b - \rho D_{\rho b} a, v)
  + \omega(\check{D}_b \rho a - \rho D_{\rho b} a, v)
  + \omega([\rho a, \rho b], v)
  \\
  &{}\quad
  + \omega(\check{D}_a v - \rho D_v a, \rho b) 
  - \omega(\check{D}_b v - \rho D_v b, \rho a)
  \\
  &= \rho a \cdot \omega(\rho b, v) 
   - \rho b \cdot \omega(\rho a, v)
        + v \cdot \omega(\rho a, \rho b)
  \\
  &{}\quad
  - \omega([\rho a, \rho b] , v)
  + \omega([\rho b, \rho a], v)
  + \omega([\rho a, \rho b], v)
  \\
  &{}\quad
  + \omega([\rho a, v], \rho b) 
  - \omega([\rho b, v], \rho a)
  \\
  &= \rho a \cdot \omega(\rho b, v) 
   - \rho b \cdot \omega(\rho a, v)
        + v \cdot \omega(\rho a, \rho b)
  \\
  &{}\quad
  - \omega([\rho a, \rho b] , v)
  + \omega([\rho a, v], \rho b) 
  - \omega([\rho b, v], \rho a)
  \\
  &= \iota_v \iota_{\rho b}\iota_{\rho a} d\omega
  \,.
\end{split}
\end{equation*}
With the relation
\begin{equation*}
  \tfrac{1}{2} \iota_a \iota_b \iota_\rho \iota_\rho
  = \iota_{\rho b} \iota_{\rho a} 
  + \iota_\rho (\iota_b \iota_{\rho a} - \iota_a \iota_{\rho b})
  + \tfrac{1}{2} \iota_\rho \iota_\rho \iota_b \iota_a \,,
\end{equation*}
which can be checked by a straightforward calculation, we get
\begin{equation*}
  \iota_v \iota_{\rho b}\iota_{\rho a} d\omega
  = \iota_ v \iota_a \iota_b  \tfrac{1}{2} \iota_\rho \iota_\rho d\omega \,.
\end{equation*}
We thus obtain the relation $A + B + C = \iota_ v \iota_a \iota_b \tfrac{1}{2}  \iota_\rho \iota_\rho d\omega$, which is Eq.~\eqref{eq:AtBo3} evaluated on $v$, $a$, and $b$.
\end{proof}

From Eq.~\eqref{eq:AtBo4} it follows that $[D,\check{D}]\mu = \tfrac{1}{2}  \iota_\rho \iota_\rho d\omega = 0$.

To conclude: We have established that every horizontal extension of $\omega$ is of the form $p^*\omega + dp^* \mu - p^* D\mu$ for some $\mu \in \Omega^1(A)$ and proved that this extension is closed if and only if the conditions $d\omega = 0$, $D\mu = - \iota_\rho \omega$, and $\check{D}\mu = -\tfrac{1}{2} \iota_\rho \iota_\rho \omega$ are satisfied. Since $(A,D,\omega)$ is assumed presymplectically anchored in the statement of Theorem~\ref{thm:EquivExt}, we conclude that $\omega$ has a closed basic extension if and only if $(A,D,\omega)$ is hamiltonian. This finishes the proof of Theorem~\ref{thm:EquivExt}. \qed

\begin{Remark}
The conditions for a horizontal extension of $\omega$ to be closed that we have derived in the proof of Theorem \ref{thm:EquivExt} are independent of whether the Lie algebroid is presymplectically anchored or not. So we have actually shown the following:
\end{Remark}

\begin{Proposition}
Let $A \to M$ be a Lie algebroid with connection $D$. A differential 2-form $\omega$ on $M$ has a closed basic extension to $W(A)$ if and only if $\omega$ is closed and $A$ has a bracket-compatible momentum section.
\end{Proposition}

\begin{Remark}
The expression $[D,\check{D}]\mu$ that appears in the proof of Theorem \ref{thm:EquivExt} can be rewritten in more familiar terms as follows:
\begin{equation*}
\begin{split}
  [D,\check{D}] 
  &= [D, [\iota_\rho, D] + \iota_T] \\
  &= [[D, \iota_\rho], D] + [\iota_\rho,[D,D]] + [D,\iota_T] \\
  &= [-\check{D} + \iota_T, D] + [\iota_\rho, [D,D]] + [D,\iota_T] \\
  &= - [D,\check{D}] + [\iota_\rho, [D,D]] + 2 [D,\iota_T]
  \,,
\end{split}
\end{equation*}
where we have used Eq.~\eqref{eq:DcheckTorsion}. Recalling that $\tfrac{1}{2} [D,D] = D^2 = R$ is the curvature operator defined in Section \ref{sec:ConnectionReview} and using $[D,\iota_T] = \iota_{DT}$, we obtain
\begin{equation}
\label{eq:DDcheckCurv}
  [D,\check{D}]\mu
  = \iota_\rho R\mu + \iota_{DT}\mu 
  = \langle\mu, \iota_\rho R + DT \rangle
  \,.
\end{equation}
\end{Remark}

\begin{Proposition}
\label{prop:CurvTorsBrack}
Let $A$ be a Lie algebroid with connection $D$ over a presymplectic manifold $(M,\omega)$. A bracket-compatible $D$-momentum section $\mu$ satisfies
\begin{equation}
\label{eq:CurvTorsBrackcomp}
  \langle\mu, \iota_\rho R + DT \rangle = 0 \,,
\end{equation}
where $R$ is the curvature operator defined in Eq.~\eqref{eq:CurvDef} and $T$ the torsion form defined in Eq.~\eqref{eq:DefTorsion2}. When, moreover, $D$ is presymplectically anchored, then $\langle\mu, DT \rangle = 0$.
\end{Proposition}
\begin{proof}
In the last step of the proof of Theorem \ref{thm:EquivExt}, we have shown that a bracket-compatible momentum section $\mu$ satisfies $[D, \check{D}]\mu = 0$. It follows from Eq.~\eqref{eq:DDcheckCurv} that Eq.~\eqref{eq:CurvTorsBrackcomp} holds.

When, moreover, $A$ is presymplectically anchored, then every momentum section satisfies $0 = D\gamma = D^2 \mu = R\mu$. Since $R$ is a tensor, $\langle R(v,w)\mu, a \rangle = \langle \mu, R(v,w) a\rangle$ for all vector fields $v$, $w$ and all sections $a$ of $A$. (Recall, that we do not distinguish notationally between the connection on $A$ and the dual connection on $A^*$, so that $R$ denotes the curvature operator on sections of $A$ as well as on sections of $A^*$.) It follows that $R\mu = 0$ implies that $\langle \mu, \iota_\rho R a\rangle = \iota_\rho \langle \mu, R a\rangle = \iota_\rho \langle R \mu, a\rangle = 0$ for all sections $a$ of $A$. With this relation, Eq.~\eqref{eq:CurvTorsBrackcomp} becomes $\langle \mu, DT \rangle = 0$.
\end{proof}

\begin{Remark}
Assume that the connection is flat, $R = 0$, so that the condition~\eqref{eq:CurvTorsBrackcomp} becomes $\langle \mu, DT \rangle = 0$. Let $a$ and $b$ be horizontal sections of $A$. Then the covariant derivative of the torsion in the direction of a vector field $v$ is given by
\begin{equation*}
\begin{split}
  (D_v T)(a,b)
  &= D_v T(a,b) - T(D_v a, b) - T(a, D_v b)
  \\
  &= D_v \bigl( D_{\rho a} b - D_{\rho b} a - [a,b] \bigr)
  \\
  &= - D_v [a,b]
  \,.
\end{split}    
\end{equation*}
If $DT = 0$, then the bracket of horizontal sections is again horizontal, so that $A$ is locally an action Lie algebroid. This leaves the possibility that there are more general hamiltonian Lie algebroids with a flat connection for which $D_v T$ is not zero but in the kernel of $\mu$.  (The simplest case of this is where the anchor is zero, in which case the Lie algebroid is hamiltonian with respect to any connection, with $\mu = 0$.)
\end{Remark}

\section{Open questions, conjectures, outlook}
\label{sec:OpenQuestions}

\subsection{Hamiltonian foliations}

\begin{Question} Viewing a regular foliation as a Lie algebroid with injective anchor, we obtain the notion of a \textbf{hamiltonian foliation} of a symplectic manifold.  In Proposition \ref{prop:FoliHam} (ii), we have shown that a symplectically complete foliation $\calF$ is weakly hamiltonian if there is a vector field $n$ that is a symmetry of the symplectic orthogonal foliation $\calF^\perp$ and is transverse to $\calF^\perp$.  If, in addition, the pullback of $\Lie_n \omega + \omega$ to $\calF$ vanishes, then $\calF$ is hamiltonian. Locally, such an $n$ always exists. What are the global obstructions?
\end{Question}

Note that $n$ is the presymplectic version of a Liouville vector field.  In the symplectic case, transversality of $n$ to $\calF^\perp$ is simply the nowhere vanishing of $n$.  In this case, that of hamiltonian tangent bundles, the existence of such an $n$ has been established, as we explained in Section \ref{sec:tangent}.

\subsection{Obstruction theory}

For the action of a Lie algebra $\frakg$ on $M$ the obstructions to the existence of an equivariant momentum map are the following. The action gives rise to a class in $H^1(\frakg)\otimes H^1(M)$ which must vanish for the action to be weakly hamiltonian. When it vanishes, the obstruction to it being hamiltonian lies in $H^2(\frakg)\otimes H^0(M)$. What is the obstruction theory for hamiltonian structures on Lie algebroids?

When there is a \emph{flat} symplectically anchored connection, the first obstruction can be easily generalized. For a flat connection, the operator $D$ on the bigraded algebra $\Omega^{\bullet,\bullet}(M,A)$ of Section \ref{sec:Category} is a differential. Then the Lie algebroid has a momentum section if and only if the $D$-cohomology class of the dualized anchor $[\iota_\rho \omega] \in H^{1,1}_D(M,A)$ is zero.

The failure of a $D$-momentum section $\mu$ to be bracket-compatible is given by the 2-form
\begin{equation*}
  \Theta := \check{D}\mu + \tfrac{1}{2}\iota_\rho \iota_\rho \omega 
  \in \Omega^{0,2}(M,A)\,.
\end{equation*}
On the edge complex $\Omega^{0,\bullet}(M,A)$ of $A$-forms, $\check{D}$ acts as the Lie algebroid differential $\DA$. Moreover, $\rho: A \to TM$ is a morphism of Lie algebroids, so that the pullback on forms $\rho^*: \Omega^{p,0}(M,A) \to \Omega^{0,p}(M,A)$ intertwines the differentials $\DA \rho^* = \rho^* d$. In particular, $\DA \tfrac{1}{2} \iota_\rho \iota_\rho \omega =  \DA \rho^* \omega = \rho^* d\omega = 0$. This shows that $\Theta$ is $\check{D}$-closed:
\begin{equation*}
  \check{D}\Theta 
  = \DA^2 \mu + \DA \tfrac{1}{2}\iota_\rho \iota_\rho \omega
  = 0 \,.
\end{equation*}
Applying the differential $D$ yields
\begin{equation*}
\begin{split}
  D\Theta
  &= D \check{D} \mu + D\tfrac{1}{2}\iota_\rho \iota_\rho \omega
  \\
  &= D(\iota_\rho D - D \iota_\rho + \iota_T) \mu 
   + D\tfrac{1}{2}\iota_\rho \iota_\rho \omega 
  \\
  &= - D \iota_\rho \iota_\rho \omega + D\iota_T \mu
    + D\tfrac{1}{2}\iota_\rho \iota_\rho \omega
  \\
  &= D\iota_T \mu
    - D\tfrac{1}{2}\iota_\rho \iota_\rho \omega
  \\
  &= D\iota_T \mu - \iota_T \iota_\rho \omega
  \\
  &= D\iota_T \mu + \iota_T D\mu 
  \\
  &= [D,\iota_T] \mu \,,
\end{split}
\end{equation*}
where we have used $D^2 = 0$, $D\mu = - \iota_\rho \omega$, and lemma \ref{lem:Drhorhoomega}. For a cohomological interpretation of $\Theta$, we will make the additional assumption that $[D, \iota_T] = \iota_{DT}$, which is equivalent to the condition that the torsion tensor be $D$-invariant, i.e.~$DT = 0$. This implies that $\Theta$ is $D$-closed. Moreover, we infer from Eq.~\eqref{eq:DDcheckCurv} that $[D,\check{D}] = 0$, so that $\check{D} = \DA$ induces a differential $\bar{\DA}$ on the graded vector space of $D$-cohomology classes $H_D^{0,\bullet}(M,A)$ at the edge of the bicomplex. We conclude that $\Theta$ represents a double cohomology class
\begin{equation*}
  [[\Theta]] \in H^2\bigl( H_D^{0,\bullet}(M,A), \bar{\DA} \bigr) \,.
\end{equation*}
This cohomology class is zero if and only if there is a $D$-closed 1-form $\nu \in \Omega^{0,1}(M,A)$ such that $\check{D} \nu = \check{D}\mu + \frac{1}{2} \iota_\rho \iota_\rho \omega$, which is the case if and only if $\mu' = \mu - \nu$ is a bracket-compatible momentum section. We conclude that $[[\Theta]]$ is the obstruction for the existence of a bracket-compatible momentum section.

The assumptions $D^2 =0$ and $DT = 0$ are quite strong, since they are equivalent to $A$ being locally an action Lie algebroid. For instance, they hold for a bundle of Lie algebras with zero anchor only if the parallel translation given by $D$ gives Lie algebra isomorphisms.

So the following question remains open:

\begin{Question}
What are the obstructions to the existence of a (bracket-compatible) $D$-momentum section in the general case?
\end{Question}

A small hint as to a general obstruction theory is given in Proposition \ref{prop:Rank1Obstr}, which states (for the rank 1 case) that non-vanishing curvature of $D$ at a singular point of the anchor is an obstruction to the existence of a $D$-momentum section. A more difficult problem yet is that any obstruction theory formulated in terms of the bigraded algebra $\Omega^{\bullet,\bullet}(M,A)$ will depend on the choice of $D$. As we have seen in Example~\ref{ex:cylinder2}, when there is a nonzero obstruction to a $D$-momentum section, it may still happen that there is another connection for which it does vanish.

\subsection{The Poisson case}

In Propositions \ref{prop:PointSymp} and \ref{prop:PointHam} we have shown that the axioms for a hamiltonian Lie algebroid $A$ can be derived from the following simple geometric principle. Given a connection $D$ on $A \to M$, we require the usual axioms for a hamiltonian action for every point $m \in M$ and all sections of $A$ that are horizontal at $m$. This principle can be applied without modification to a Lie algebroid over a Poisson manifold, which leads to the following definition:

\begin{Definition}
\label{def:HamLAPoiss}
Let $(A,\rho, [~,~])$ be a Lie algebroid over a Poisson manifold $(M,\Pi)$. Let $D$ be a connection on $A$ and $\check{D}$ the opposite $A$-connection on $TM$ (Definition \ref{def:OppConn}).
\begin{itemize}

\item[(P1)]
$A$ is \textbf{Poisson anchored} with respect to $D$ if
\begin{equation*}
  \check{D}\Pi = 0 \,.
\end{equation*}

\item[(P2)]
A section $\mu \in \Gamma(A^*) = \Omega^0(M, A^*)$ is a \textbf{$D$-momentum section} if
\begin{equation*}
  \tilde{\Pi}\circ D\mu = \rho \,,
\end{equation*}
where $\tilde{\Pi}: T^*M \to TM$ is the bundle map associated to $\Pi$ and $D\mu$ is viewed as map from $A \to T^* M$.

\item[(P3)]
A $D$-momentum section $\mu$ is \textbf{bracket-compatible} if
\begin{equation*}
  \langle \mu, T(a,b) \rangle = \Pi(\langle D\mu, a\rangle, \langle D\mu, b\rangle)
\end{equation*}
for all sections $a$ and $b$ of $A$
\end{itemize}
A Lie algebroid together with a connection $D$ and a section $\mu$ of $A^*$ satisfying (P1) and (P2) is called \textbf{weakly hamiltonian}. It is called \textbf{hamiltonian} if it satisfies (P1)-(P3).
\end{Definition}

Hamiltonian Lie algebroids over Poisson manifolds will be studied in a forthcoming paper \cite{BlohmannWeinstein:PoissonHamLA}.

\subsection{Hamiltonian Lie groupoids}

One of the good properties of a proper hamiltonian action of a Lie group $G$ is that the singularities of the zero locus $Z$ of the momentum map are conical, i.e.~there is always a local chart in which $Z$ is the zero locus of a quadratic polynomial in the coordinates. This is a simple consequence of the symplectic version \cite{Weinstein:1977} of Bochner's linearization theorem \cite{Bochner:1945}, which states that the action of $G$ on the neighborhood of a fixed point can be linearized by a choice of canonical coordinates, so that the momentum functions generating the fundamental vector fields of this linear action must be quadratic.

The linearizability of proper group actions generalizes to the linearization of proper Lie groupoids \cite{CrainicStruchiner:2013, Weinstein:2002, Zung:2006}. For an analogous statement about the singularities of the zero locus of the momentum section of a hamiltonian Lie algebroid we, therefore, need a suitable notion of hamiltonian Lie groupoid.The result we would like to be able to prove under reasonable assumptions is the following.

\begin{Conjecture}
\label{conj:HamGrpLin}
If a hamiltonian Lie algebroid over a symplectic manifold integrates to a proper hamiltonian Lie groupoid, then the zero locus of the momentum section has only conical singularities.
\end{Conjecture}

By analogy with the case of hamiltonian Lie algebroids, it seems that the notion of hamiltonian Lie groupoid should involve a connection for the submersion given by the source map, such that each bisection horizontal at a groupoid element $g$ induces a diffeomorphism of the base which is presymplectic at the source of $g$. Moreover, we should assume that the identity bisection is horizontal. Then the linearization of the connection at the identity bisection will induce  a presymplectically anchored connection on the Lie algebroid. 

The notion of momentum section for a Lie groupoid should be simply that for the Lie algebroid. The condition of bracket compatibility, however, will have to be replaced by a global condition on the Lie groupoid, just as the $\frakg$-equivariance of a hamiltonian momentum map for a Lie algebra action has to be replaced by $G$-equivariance for a Lie group action.

Determining the precise axioms for hamiltonian Lie groupoids and the additional conditions that make Conjecture~\ref{conj:HamGrpLin} true is a goal of ongoing work on generalized hamiltonian structures.

\subsection{Structure theory of reduced action groupoids}

The construction of Proposition~\ref{prop:GroupReduce} to reduce an action groupoid by a subgroup is quite simple and straightforward. The question of which Lie groupoids arise in this way, however, is not that simple.

In Proposition~\ref{prop:MoritaReduce} we have seen that the reduced groupoid is Morita equivalent to the action groupoid. A proper Lie groupoid is Morita equivalent to an action groupoid of a compact group if and only if it admits a faithful representation on a vector bundle of finite rank \cite{Trentinaglia:2010}. But even if this is the case, the Lie groupoid may not arise by reduction. For example, every transitive groupoid is equivalent to a gauge groupoid of a principal bundle, which is Morita equivalent to the gauge group, i.e.~to the action groupoid of the trivial action of the gauge group on a point. However, not every gauge groupoid arises by reduction.

In fact, let $P$ be a right $K$-principal bundle over $\tilde{M}$. Then the conditions for $\tilde{\Gamma} := \mathrm{Gauge}(P) = (P\times P)/K$ to be the reduction of an action groupoid are the following:
\begin{itemize}

\item[(i)] $P \cong H\backslash G$ is a homogeneous space.

\item[(ii)] The right action of $K$ on $P \cong H\backslash G$ lifts to a left $G$-equivariant right action of $K$ on $G$. 

\end{itemize}
Condition (i) implies that $\tilde{M} \cong H\backslash G/K$ is isomorphic to the quotient of $M := G/K$ by an $H$-action. Moreover, every left $G$-equivariant isomorphism of $G$ is given by the right multiplication by a unique element of $G$, so condition (ii) implies that $K$ is a subgroup of $G$ which acts by right multiplication. The action of $K$ on $P \cong H \backslash G$ is free, which implies that $K$ has trivial intersection with all conjugacy classes of $H$. All this shows that $\tilde{\Gamma}$ is isomorphic to the gauge groupoid of the principal $K$ bundle $H \backslash G \to H\backslash G / K$. By example~\ref{ex:TransReduce} this is the general form of the reduction of a \emph{transitive} action groupoid.

Every Lie groupoid is the union of the transitive Lie groupoids over its orbits.  In the case of an $H$-reduced groupoid of a $G$-action on $M$, the orbit groupoids are isomorphic to the gauge groupoids $(H \backslash G \times H\backslash G)/K_x$, where $K_x$ is a smooth family of subgroups of $G$ that is parametrized locally by a submanifold $X \subset M/H$ transverse to the orbits. Joining the orbit groupoids to a smooth manifold, however, is generally very subtle. For example, the dimension of the orbits and, hence, the gauge group may not be locally constant. But even in the regular case when the subgroups $K_x$ are all isomorphic to a fixed group $K$, the embedding of $K_x$ into $G$ may depend on the parameter $x$. It would be interesting to better understand the structure transverse to the orbits of reduced action groupoids. For example the following question is related to the question of hamiltonian Lie groupoids and relevant for the application to general relativity.

\begin{Question}
How is the linearization of a proper $G$-action on $M$ around an orbit related to the linearization of the $H$-reduced groupoid around the corresponding orbit in $M/H$. 
\end{Question}

\begin{Question}
Find a hamiltonian Lie algebroid which explains the coisotropic property of the constraint set for the initial value problem of Einstein's equations.
\end{Question}

\section{Related work}
\label{sec:related}

In \cite{LevinOlshanetsky:2000,Olshanetsky:2002}, Levin and Olshanetsky defined a  hamiltonian algebroid to be a vector bundle $A \to M$ over a symplectic manifold together with an antisymmetric bracket $[~,~]$ on the sections of $A$ and a bundle map $h: A \to M \times \bbR$, such that
\begin{equation*}
  h([a,b]) = \{ h(a), h(b)\}  \,,\qquad
  [a,fb] = f\,[a,b] + \{ h(a), f \}\, b \,, 
\end{equation*}
for all sections $a$, $b$ of $A$, all functions $f$ on $M$, where $\{~,~\}$ denotes the Poisson bracket. Since the bracket is not assumed to satisfy the Jacobi identity and since the map $\Gamma(A) \to \calX(M)$, $a \mapsto \{h(a), ~\}$ is not $C^\infty(M)$-linear, a hamiltonian algebroid is not a Lie algebroid. In fact, the requirement that $h$ be compatible with the brackets of all sections precludes the possibility that $A$ is a Lie algebroid in all but trivial cases.

In \cite[Def.~1.30]{Bos:2007}, Bos introduced a notion of hamiltonian action of a Lie algebroid on a fibre bundle. Let $\alpha: A \times_M S \to TS$ be an action of the Lie algebroid $A \to M$ on the bundle $S \to M$. Assume that the bundle $S \to M$ is equipped with a smooth family $\omega$ of symplectic forms on the $A$-orbits. The first condition \cite[Eq.~(1.8)]{Bos:2007} for a hamiltonian action is that the pullback of the family of symplectic forms along the action to a form on the action Lie algebroid $A \ltimes S$ is exact in the Lie algebroid cohomology, $\DA \mu = - \alpha^* \omega$, where $\mu \in \Gamma(S,A^* \times_M S)$ is a momentum ``map''. The second condition \cite[Eq.~(1.9)]{Bos:2007} is that $d\langle\mu, a\rangle = -\iota_{\alpha(a)} \omega$ holds on the fibres of $S \to M$.
To compare this with our notion of hamiltonian Lie algebroids we consider the case that $S = M \to M$ is the trivial bundle, so that $A \ltimes S = A$ and $\alpha = \rho$. The family $\omega$ of symplectic forms on the orbits of $A$ now pulls back to a Lie algebroid 2-form on $A$ given by $(\rho^* \omega)(a,b) = \omega(\rho a, \rho b)$ for all sections $a$ and $b$ of $A$. The first condition \cite[Eq.~(1.8)]{Bos:2007} is our condition (H3) for a momentum section to be bracket-compatible. The second condition \cite[Eq.~(1.9)]{Bos:2007} for the Lie algebroid action to be hamiltonian is vacuous, since the fibres of $S = M \to M$ are trivial. The upshot is that from the notion of hamiltonian Lie algebroid action we retrieve only condition (H3).

In March of 2015, we sent an early draft of this paper to Kotov and Strobl, the authors of \cite{KotovStrobl:2015,KotovStrobl:2016}, who were working independently on related questions. Our condition~(H1) for $A$ to be symplectically anchored appears in \cite[Eq.~(6)]{KotovStrobl:2016} and in the form of Proposition \ref{prop:Dcheckomega} in \cite[Eq.~(52)]{KotovStrobl:2016}.  The defining equation~(H2) for a momentum section appears in \cite[Eq.~(7)]{KotovStrobl:2016}. The study of Lie algebroids over symplectic manifolds satisfying these two conditions, weakly hamiltonian Lie algebroids in our terminology, was postponed in \cite[p.~7]{KotovStrobl:2016} to possible future work.

In \cite[Def.~2.4]{Lu:2008} the following notion of $H$-quotient of an action Lie algebroid $\frakg \ltimes M$ was introduced by Lu.  Let $H$ be a subgroup of a group $G$ that integrates the Lie algebra $\frakg$. Let $\frakh \subset \frakg$ be the Lie algebra of $H$. Assume further that there is a free and proper $H$-action on $M$ which integrates the action of $\frakh$, such that the action map $\frakg \to \calX(M)$ is $H$-equivariant (see Remark \ref{rmk:Lu}). Under these assumptions, it is shown in \cite[Lemma~2.3]{Lu:2008} that there is a Lie algebroid structure on the vector bundle $\frakg/\frakh \times_H M \to M/H$ such that the projection $\frakg \ltimes M \to \frakg/\frakh \times_H M$ is a homomorphism of Lie algebroids. As we have shown in Corollary \ref{cor:ReducedAlgd}, this is the infinitesimal counterpart of our construction of an analogous quotient for an action groupoid.

\bibliographystyle{hplain}
\bibliography{HamLAs}

\end{document}